\documentclass[12pt]{amsart}
\usepackage{amsmath, mathabx}
\usepackage{amsxtra}
\usepackage{amscd}
\usepackage{amsthm}
\usepackage{amsfonts}
\usepackage{amssymb}
\usepackage{mathrsfs}
\usepackage{mathbbol}

\usepackage {pstricks}
\usepackage {tikz}
\usetikzlibrary {decorations.pathmorphing}
\usepackage{pstricks,pst-node}
\usepackage[all]{xy}
\usepackage{mathbbol}

\usepackage[plainpages,backref,urlcolor=blue]{hyperref}
\usepackage{verbatim}

\usepackage{dashundergaps}

\usepackage{tikz-cd}
\usepackage{enumerate}

\usepackage{graphicx}

\newtheorem{theorem}{Theorem}[section]

\newtheorem{lemma}[theorem]{Lemma}

\newtheorem{corollary}[theorem]{Corollary}

\newtheorem{scholium}[theorem]{Scholium}

\theoremstyle{definition}
\newtheorem{definition}[theorem]{Definition}
\newtheorem{example}[theorem]{Example}

\theoremstyle{remark}
\newtheorem{remark}[theorem]{Remark}

\numberwithin{equation}{section}
\numberwithin{figure}{section}

\newcommand{\mc}{\mathcal}

\newcommand{\lr}{\longrightarrow}
\newcommand{\inv}{^{-1}}

\newcommand{\be}{\begin{equation}}
\newcommand{\ee}{\end{equation}}

\newcommand{\A}{{\mathbb A}}
\newcommand{\B}{{\mathbb B}}
\newcommand{\C}{{\mathbb C}}

\newcommand{\Z}{{\mathbb Z}}
\newcommand{\N}{{\mathbb N}}

\newcommand{\I}{{\mathbb I}}
\newcommand{\K}{{\mathbb K}}
\newcommand{\Q}{{\mathbb Q}}

\newcommand{\X}{{\mathbb X}}

\newcommand{\BH}{{\mathbb H}}
\newcommand{\BD}{{\mathbb D}}

\newcommand{\BB}{{\mathbb B}}
\newcommand{\BG}{{\mathbb G}}

\newcommand{\BT}{{\mathbb T}}

\newcommand{\BL}{{\mathbb L}}

\newcommand{\CB}{{\mathcal B}}

\newcommand{\CF}{{\mathcal F}}
\newcommand{\CG}{{\mathcal G}}
\newcommand{\CH}{{\mathcal H}}

\newcommand{\CJ}{\mc J}
\newcommand{\CK}{{\mathcal K}}

\newcommand{\CM}{{\mathcal M}}

\newcommand{\CO}{{\mathcal O}}

\newcommand{\CR}{{\mathcal R}}

\newcommand{\CT}{{\mathcal T}}

\newcommand{\CV}{{\mathcal V}}
\newcommand{\CW}{{\mathcal W}}

\newcommand{\SF}{{\mathscr F}}
\newcommand{\SH}{{\mathscr H}}

\newcommand{\mf}{\mathfrak}

\newcommand{\fg}{{\mf g}}

\newcommand{\fsl}{{\mathfrak {sl}}}
\newcommand{\fsp}{{\mathfrak {sp}}}
\newcommand{\fso}{{\mathfrak {so}}}

\newcommand{\gl}{{\mathfrak {gl}}}
\newcommand{\osp}{{\mathfrak {osp}}}

\newcommand{\id}{{\rm{id}}}

\newcommand{\ve}{\varepsilon}

\newcommand{\U}{{\rm{U}}}
\newcommand{\End}{{\rm{End}}}
\newcommand{\Hom}{{\rm{Hom}}}

\newcommand{\Sym}{{\rm{Sym}}}

\newcommand{\im}{{\rm{Im}}}

\newcommand{\wt}{\widetilde}

\newcommand{\sdim}{{\rm{sdim}}}

\newcommand{\ot}{\otimes}
\newcommand{\OSp}{{\rm OSp}}

\newcommand{\sdet}{{\rm{sdet}}}

\newcommand{\ATL}{{\mathcal{PTL}}}

\newcommand{\TLBC}{\mathcal{TLB}}

\newcommand{\ad}{{\rm ad}}

\newcommand{\beq}{\begin{eqnarray}}
\newcommand{\eeq}{\end{eqnarray}}

\newcommand{\baln}{\begin{aligned}}
\newcommand{\ealn}{\end{aligned}}

\newcommand{\lra}{\longrightarrow}

\newcommand{\bM}{\mathbf{M}}

\newcommand{\AB}{\mathcal{AB}}
\newcommand{\wh}{\widehat}

\newcommand{\ba}{{\mathbf a}}
\newcommand{\bc}{{\mathbf c}}
\newcommand{\bm}{{\mathbf m}}
\newcommand{\bs}{{\mathbf s}}
\newcommand{\bt}{{\mathbf t}}

\newcommand{\CBC}{\mathcal{CB}}

\newcommand{\HB}{\mathcal{HB}}

\newcommand{\MPB}{\mathcal{MPB}}
\newcommand{\UPB}{\mathcal{UPB}}
\newcommand{\PB}{\mathcal{PB}}

\newcommand{\MTL}{\mathcal{MTL}}

\begin{document}

\normalfont

\title[Polar Brauer categories]{Polar Brauer categories,  infinitesimal braids and Lie superalgebra representations}

\author{G.I. Lehrer and R.B. Zhang}
\address{School of Mathematics and Statistics,
University of Sydney, N.S.W. 2006, Australia}
\email{gustav.lehrer@sydney.edu.au, ruibin.zhang@sydney.edu.au}

\begin{abstract}
We define a class of monoidal categories whose morphisms are diagrams, and which are enhancements
and generalisations of the Brauer category obtained by adjoining infinitesimal braids, ``coupons'' and poles. Properties of these
categories are explored, particularly diagrammatic equations. We construct functors
from certain of them to categories of representations of Lie algebras and superalgebras. Applications
include a diagrammatic construction of the centre of the universal enveloping superalgebra  and certain ``characteristic identities'',
as well as an analysis of certain tensor representations. 
We show how classical diagram categories arising in invariant theory are
special cases of our constructions, placing them in a single unified framework. 
\end{abstract}

\makeatletter
\@namedef{subjclassname@2020}{\textup{2020} Mathematics Subject Classification}
\makeatother


\maketitle

\tableofcontents


\section{Introduction and background}\label{sect:intro}

We develop a diagrammatic framework for applying the algebra of infinitesimal braids to decompose tensor product modules of the type $V_1\ot\dots\ot V_r$ for Lie algebras and Lie superalgebras with Casimir operators. This provides a natural setting for many of the diagram categories which have played important roles in Lie theory at the classical level, and found applications in theoretical physics. 
Some of the general ideas behind this work are contained in \cite{LZ06} which gives a
unified treatment  of the first fundamental theorems of classical invariant theory,  and \cite{LZ15, LZ17-Ecate} on Brauer categories. The concept of ``polar diagrams''  introduced in \cite{ILZ} in the quantum group setting also plays an important role. 
Below we briefly describe the background of this paper. 

\subsection{Casimir operators and infinitesimal braids}\label{ss:tr}
The work \cite{LZ06} provided a unified treatment of 
the first fundamental theorems of invariant theory of classical Lie algebras and $G_2$ by analysing algebra homomorphisms from the algebra of infinitesimal braids to endomorphism algebras of tensor powers of the natural modules of the respective Lie algebras. Such algebra homomorphisms factor through the Casimir algebras of the  Lie algebras. A generalisation of the treatment to the corresponding quantum groups at generic $q$ was also given in op. cit..

The algebra $T_r$ of infinitesimal braids for a fixed integer $r\ge 2$ (see, e.g., \cite[\S XX.3]{K}) is a unital associative algebra with generating set $\{t_{ij}\mid 1\leq i<j\leq r\}$ and defining relations
\be\label{eq:deftr}
 [t_{ij},t_{k\ell}]=0,\;\;\;[t_{ik}+t_{i\ell},t_{k\ell}]=0,\;\;\; [t_{ij},t_{ik}+t_{jk}]=0,
\ee
where $i,j,k$ and $\ell$ are pairwise distinct indices and the square brackets denote the usual commutator. 
This algebra has some remarkable properties (see Section \ref{sect:BH-properties}). 
It occurs in the literature in several contexts, such as the holonomy of connections
on spaces of configurations (the KZ connection)  and representations of
the pure braid group (a good source is \cite[\S XX]{K}). 

The algebra $T_r$ enters representation theory through  
Casimir algebras of Lie algebras and Lie superalgebras.  
Let $\fg$ be a finite dimensional Lie superalgebra, which admits a non-degenerate ad-invariant bilinear form that is even and supersymmetric. 
It is well known that  its universal enveloping superalgebra $\U(\fg)$ carries a Hopf superalgebra structure, the co-multiplication of which will be 
denoted by $\Delta$. Using this one defines iterated co-multiplication $\Delta^{(r-1)}: \U(\fg)\lra\U(\fg)^{\ot r}$ for each $r\ge 2$, which is a superalgebra homomorphism.

Let $C$ be the quadratic Casimir operator in the centre of $\U(\fg)$. This may be expressed in terms of 
any pair of bases $\{X_a\mid a=1, 2, \dots, \dim(\fg)\}$ and $\{X^a\mid a=1, 2, \dots, \dim(\fg)\}$ of $\fg$, 
which are dual with respect to the non-degenerate ad-invariant bilinear form. We have $C=\sum_{a=1}^{\dim \fg}X_aX^a$.

Define the {\em tempered Casimir operator} $t\in\U(\fg)^{\ot 2}$ by
\be\label{eq:deftr}
t=\frac{1}{2}\left(\Delta(C)-C\ot 1-1\ot C\right),
\ee
which evidently commutes with $\Delta(\U(\fg))$. 
One checks easily that $t$ can be expressed as 
$t=\sum_{\alpha=1}^{\dim(\fg)}X_\alpha\ot X^\alpha$. 
Further, we have the elements $C_{ij} \in\U(\fg)^{\ot r}$, for $1\leq i<j\leq r$, given by
\be\label{eq:cij}
C_{ij}=\sum_{\alpha=1}^{\dim(\fg)}1^{\ot (i-1)}\ot X_\alpha\ot 1^{\ot (j-i-1)}\ot X^\alpha\ot 1^{\ot (r-j)}.
\ee
We refer to the subalgebra $C_r(\fg)$ of $\U(\fg)^{\ot r}$ generated by the elements $C_{ij}$ as the   
{\em Casimir algebra} of $\fg$ of degree $r$. 
It is known (see, e.g., \cite{LZ06}) that 
for all pairwise distinct indices $i, j, k, \ell$, 
\beq\label{eq:C-4-term}
[C_{i j}, C_{k \ell}] = 0, \quad [C_{i k} + C_{i\ell}, C_{k\ell}] = 0, \quad [C_{i j} , C_{i k} + C_{j k}] = 0.
\eeq
Thus one has the following result (see, e.g., \cite[Theorem 2.3]{LZ06}). 
\begin{theorem}\label{thm:TtoC}
The map $t_{ij}\mapsto C_{ij}$ ($1\leq i<j\leq r$) extends uniquely to a surjective algebra homomorphism $\psi:T_r\lr C_r(\fg)$.

\end{theorem}

Observe that for all $z\in C_r(\fg)$, we have 
\be\label{eq:tcomm}
[z, \Delta^{(r-1)}(\U(\fg))]=0.
\ee 
This provides a tool for studying endomorphism algebras of tensor products of representations of $\fg$.

\subsection{Enter representations} Let $V_1,V_2,\dots,V_r$ be arbitrary objects in the category $\fg$-Mod of $\fg$-modules. Then $V_1\ot V_2\ot\dots\ot V_r$ is a $\U(\fg)^{\ot r}$-module
in the obvious way, since the iterated co-multiplication $\Delta^{(r-1)}$ defines a $\U(\fg)$-action on $V_1\ot V_2\ot\dots\ot V_r$. This is usually considered to be 
the ``tensor product'' of the representations $V_i$, and analysis of the tensor product is a subject of a significant literature, for particular choices of $V_i$. 

The tensor module $V_1\ot\dots\ot V_r$ restricts to a $C_r(\fg)$-module, 
where the action commutes with the $\U(\fg)$-action by \eqref{eq:tcomm}. Thus we have 
an algebra homomorphism
\be\label{eq:C-end}
\mu_{V_1, \dots, V_r}: C_r(\fg)\lr \End_{\U(\fg)}(V_1\ot\dots\ot V_r).
\ee
The composition of this with the homomorphism $\psi:T_r\lr C_r(\fg)$ defines 
an action of $T_r$ on $V_1\ot\dots\ot V_r$, 
whence we have the algebra homomorphism
\be\label{eq:homend}
\eta_{V_1, \dots, V_r}: T_r\lr \End_{\U(\fg)}(V_1\ot\dots\ot V_r).
\ee

Clearly the utility of \eqref{eq:homend} depends to a large extent on the choice of the representations $V_1,\dots,V_r$.  At this stage they could be completely arbitrary. 
However in \cite{LZ06} it is shown that if $V_1=V_2=\dots=V_r=V$, a ``strongly multiplicity free''  module for a simple Lie algebra $\fg$, then $\eta_{V, \dots, V}(T_r)$ is essentially the full 
commutant of the $\U(\fg)$-module $V^{\ot r}$, and is semi-simple. Thus the module $V^{\ot r}$ may be analysed by studying the structure of $\eta_{V, \dots, V}(T_r)$.

Further, it is shown in \cite[\S 4]{LZ06} how the image of the Brauer algebra in $\End_{\fg}(V^{\ot r})$ appears in $\eta_{V, \dots, V}(T_r)$ when $\fg$ is the orthogonal or symplectic algebra, 
and $V$ the natural $\fg$-module.  This suggests that diagrammatics such as that for the Brauer category could be applied to study applications of $T_r$  in the representation theory of Lie superalgebras. 

\subsection{Diagrammatic method}
Following the pioneering work \cite{Br37}, it is now well established \cite{GL03,LZ10, LZ15, LZ17, LZ17-Ecate, LZZ20, LZ21} (see also \cite{Betal, BCNR, RSo, McS,We05} and references therein) 
that diagrams may be used to analyse certain tensor product modules $V_1\ot\dots\ot V_r$ for  Lie algebras,  Lie superalgebras, and the related quantum groups and quantum supergroups \cite{BGZ90, Z98}. 
In particular \cite{ILZ} shows how to identify the category of representations of 
$\U_q(\fsl_2)$ of the form $M\ot V^{\ot r}$ (where $M$ is a projective Verma module and $V$ is the two dimensional Weyl module) with a certain category of polar tangle diagrams.  

A general method resulting from the diagrammatics which we develop here enables one to investigate certain classes of tensor product modules collectively. One creates diagram categories and develops  functors from them to categories of such modules for $\fg$. The graphical description of morphisms of a diagram category usually leads to simpler descriptions of morphisms of the target category of $\fg$-modules, and hence
 a better understanding of the latter.

A widely studied example is the Brauer category introduced in \cite{LZ15} and functors from it to the categories of tensor powers $V^{\ot r}$ of the natural module for the orthogonal or symplectic Lie algebra, or the orthosymplectic Lie superalgebra.  Similarly, an oriented version of the Brauer category (see, e.g., \cite{LZ23}) may be applied to the study of repeated tensor products of the natural module and its dual for the general linear Lie (super)algebra. 

A modern categorical formulation of the invariant theory of the classical (super) groups \cite{LZ15, LZ17} and the corresponding quantum (super)groups \cite{LZZ20} has been developed using these diagram categories. 
Much new insight has been gained in understanding the endomorphism algebras $\End_{\U(\fg)}(V^{\ot r})$ \cite{LZ12, LZ15, Zy}, e.g., in developing presentations for them in terms of explicit generators and relations. 

In this work we develop general categorical and diagrammatic techniques which may be applied to describe diagrammatically the endomorphisms of tensor product modules $V_1\ot\dots\ot V_r$ for Lie algebras and Lie superalgebras, as well as characteristic identities for quantum family algebras and more generally, the structure of universal enveloping algebras. 

\subsection{Content of this paper} 

\subsubsection{ }
Let $\CM\supset\CV$ be finite sets, and let $\delta_\CV=\{\delta_v\in \K\mid v\in \CV\}$ 
be a set of fixed parameters over a field $\K$. 
We introduce a $\K$-linear diagrammatic monoidal category $\HB_{\CG_1}(\delta_\CV)$, 
which contains a subcategory isomorphic to the Brauer category coloured 
by $\CM$ with parameter set $\delta_\CV$. 
It includes analogues of infinitesimal braids \cite[\S XX]{K} as morphisms, and has a set $\CG_1$ 
of distinguished morphisms (called coupons) as additional input. 
The category $\HB_{\CG_1}(\delta_\CV)$ is designed for the study of modules 
of Lie superalgebras of the type $M_{a_1}\ot M_{a_2}\ot  \dots \ot  M_{a_r}$, 
for $a_1, a_2, \dots, a_r\in \CM$ and $r\in\N=\{0, 1, 2, \dots\}$, 
where any $M_v$ with $v\in\CV$ is a finite dimensional self-dual simple module, 
while the modules $M_a$ for all $a\not\in\CV$ are arbitrary. 

The category $\HB_{\CG_1}(\delta_\CV)$ generalises the infinitesimal symmetric category of \cite[Definition XX.4.1]{K} by including the set $\CG_1$ of coupons; it also differs from the latter in that additional conditions are imposed on infinitesimal braids to deal with the self-dual nature of objects in $\CV$. 
Some very special cases of the category 
with non-empty $\CG_1$ are the enhanced Brauer category\cite{LZ17-Ecate},  
and a category designed for studying $G_2$-representations \cite{BE} (also see \cite{Ros96, Ros04}).

\subsubsection{ }
We study extensively the case $\CM\supset\CV=\{v\}$ and 
$\CG_1=\emptyset$ in Section \ref{sect:hHB},  where
$\HB_{\CG_1}(\delta_\CV)$ is denoted by $\HB(\delta_\CV)$ with $\delta=\delta_v$. 
We define a quotient monoidal category $\wh\HB(\delta)$ of $\HB(\delta_\CV)$ in Definition \ref{def:mpolar}; the origin of which is an interesting algebra homomorphism 
from the algebra $T_r$ of infinitesimal braids to the Brauer algebra $B_r(\delta)$ of degree $r$ (see Lemma \ref{lem:T-H-B}).  

In the case $\CM=\{m, v\}\supsetneq\CV=\{v\}$, the monoidal category 
$\wh\HB(\delta)$ gives rises to a multi-polar Brauer category  $\MPB(\delta)$, whose morphisms have a pictorial representation 
in terms of polar enhancements of Brauer diagrams, referred to as multi-polar Brauer diagrams (see Figure \ref{fig:pdiag} for an example). 
It contains a full subcategory $\PB(\delta)$, called the polar Brauer category, 
 with objects $(m, v^r)$ for all $r\in\N$. Now $\PB(\delta)$  is not a monoidal category in itself, but has the structure of a module category over the usual Brauer category $\CB(\delta)$ as  monoidal category. 
It is shown in Theorem \ref{thm:RSo} that there 
is a category isomorphism (whic is not monoidal) from $\PB(\delta)$ to the affine Brauer category in \cite{RSo} (also see \cite{Betal}). 

A particular quotient category $\MTL(\delta)$ of $\wh\HB(\delta)$ 
is introduced in Section \ref{sect:TL}.  
A full subcategory $\ATL(\delta)\subset \MTL(\delta)$ with objects $(m, v^r)$ 
for all $r\in\N$,  called the polar Temperley-Lieb category, 
is closely related to the affine Temperley-Lieb category \cite{GL98} 
(also see \cite{ILZ}). 
We determine its structure by adapting the treatment 
of the affine Temperley-Lieb category in \cite{ILZ} to this case. 

\subsubsection{ }
We investigate the category $\HB_{\CG_1}(\delta_\CV)$ 
with $\CM\supset\CV=\{v\}$ and a non-trivial 
$\CG_1=\{\check\omega_3:\emptyset\to v^3\}$ (see Section \ref{sect:B-G2}). 
We construct a quotient monoidal category $\wh\HB_{\check\omega_3}(7)$ from this data, 
where the condition $\delta_v=7$ emerges  automatically. The construction is treated in Section \ref{sect:EB-G}

The category $\wh\HB_{\check\omega_3}(7)$ is interesting in its own right, and specifically, 
it is useful for studying representations of the Lie algebra $G_2$. 
In particular, the special case of $\wh\HB_{\check\omega_3}(7)$ with $\CM=\CV=\{v\}$
is isomorphic to the diagram category in \cite{BE} built from a cross product 
over $\C^7$, which is known to be isomorphic to the full subcategory of $G_2$-Mod with objects $(\C^7)^{\ot r}$ for all $r\in \N$.  
It will be very interesting to investigate the structure and applications of 
$\wh\HB_{\check\omega_3}(7)$ 
in the multi-polar setting with $\CM=\{m, v\}\supset \CV=\{v\}$. 

\subsubsection{ }
For any Lie superalgebra $\fg$ with Casimir algebra $C_r(\fg)$,  
let $\CM\supset\CV$ be finite sets of objects in $\fg$-Mod, where 
each $L\in\CV$ is a finite dimensional self-dual simple $\fg$-module.  
Denote by $\CT(\CM)$ the full subcategory of the category $\fg$-Mod whose objects are all possible tensor products of elements of $\CM$. This is a monoidal category with
 monoidal structure given by the usual tensor product of $\fg$-Mod. Let 
$\sdim_\CV=\{\sdim(L)\mid L\in\CV\}$. We construct a monoidal functor 
$\SF: \HB_{\CG_1}(\sdim_\CV) \lra \CT(\CM)$ in Theorem \ref{thm:main}, using the representations of Casimir algebras and  algebras of infinitesimal braids, respectively defined by \eqref{eq:C-end} and \eqref{eq:homend}.   
\subsubsection{ }
Assume that $\fg$ is the orthosymplectic Lie superalgebra $\osp(V; \omega)$. 
In the special case $\CM\supset\CV=\{V\}$, where $V$ is the natural $\fg$-module,   it is shown in Theorem \ref{thm:a-funct} that the monoidal functor factors through $\wh\HB(\sdim)$, via a uniquely defined monoidal functor 
$\wh\SF: \wh\HB(\sdim) \lra \CT(\CM)$.  
Given any $M\in\CM$, we have the corresponding multi-polar Brauer category $\MPB(\sdim)$. The functor $\wh\SF$ restricts to a monoidal functor
$\wh\SF_M: \MPB(\sdim) \lra \CT(M, V)$, which can be further restricted to the polar Brauer category, yielding   
a functor $\CF_M: \PB(\sdim) \lra \CT_M (V)$, 
where $\CT_M (V)$ is the full subcategory of $\CT(M, V)$ 
with objects $M\ot V^{\ot r}$ for all $r\in\N$. 
As shown in Lemma \ref{lem:mod-T}, the functor $\CF_M$ respects the 
module category structure of $\PB(\sdim)$ over the usual Brauer category $\CB(\sdim)$. 

\subsubsection{ }
When $M$ is the universal enveloping superalgebra $\U=\U(\fg)$ itself regarded as a left module,  
the functor $\CF_\U: \PB(\sdim)\lra\CT_\U(V)$ provides an effective tool 
for studying the universal enveloping superalgebra. 

We show in Section \ref{sect:centre-construct} that the centre of $\U(\fg)$ 
can be analysed by investigating $\CF_\U(\End_{\PB(\sdim)}(\U))$; 
in particular, explicit generators of the centre are obtained in this way. 

The restriction $\CF_\U|_1$ of $\CF_\U$ to $\End_{\PB(\sdim)}((M,V))$ has a non-trivial kernel, 
and non-zero elements of $\ker(\CF_\U|_1)$ lead to interesting identities in $\U(\fg)\ot\End_\C(V)$. 
In the special cases when $V=\C^m$ is purely even, or $V=\C^{2n}$ is purely odd,  
which respectively correspond to $\fg$ being $\mathfrak{so}_m$ or $\mathfrak{sp}_{2n}$, 
we obtain in Section \ref{sect:char-id} the characteristic identities for these Lie algebras 
discovered in \cite[\S 4.10]{B} and \cite{BG, G}.  
Characteristic identities are monic polynomial equations over the centre of $\U(\fg)$, 
which are satisfied by $\CF_\U(\BH)$ in the spirit of the Cayley-Hamilton. 
They provide a useful tool for explicit Lie theoretical computations, 
and have been widely applied in mathematical physics 
(see \cite{IWD} for a brief discussion and references).

We note that a notion of quantum family algebras (see Remark \ref{rmk:family}) was introduced by Kirillov in \cite{Ki1, Ki2} (and also Kostant in \cite{Ko}).   
The study of the centres of universal enveloping superalgebras and characteristic identities 
comprises a central part of this theory \cite{Rn}. 
Quantum family algebras are not a major theme in this paper, 
but it should be observed that investigation here  provides a diagram categorical approach to them.

\subsubsection{ }
When $V=\C^{0|2}$ so that $\sdim=-2$, the Lie superalgebra $\osp(V;\omega)$ 
reduces to the ordinary Lie algebra $\fsp_2(\C)$. When $M$ is
 the Verma module $M_\lambda$ with highest weight $\lambda$ or the simple module $L_\lambda$ for $\fsp_2(\C)$, we shall show that
 the functor $\CF_{M}:  \PB(-2)\lra \CT_{M}(V)$ factors through an analogue of the Temperley-Lieb category $\TLBC(-2, \lambda)$ of type $B$. 
 Furthermore, if $M_\lambda$ is projective in category $\mathcal O$ of $\fsp_2(\C)$, we show that $\TLBC(-2, \lambda)$ is isomorphic to $\CT_{M_\lambda}(V)$. 
 This result is a classical analogue of the equivalence between a certain Temperley-Lieb algebra of type B and a category of infinite dimensional $\U_q(\fsl_2)$-modules 
 proved in \cite{ILZ}; it provides an explicit description of homomorphisms in $\CT_{M_\lambda}(V)$.   

\medskip

We remark that affine Brauer categories \cite{BCNR, RSo} 
(at least the oriented case) are the central charge zero cases of 
a more general monoidal category 
known as the Heisenberg category of central charge $k$ \cite{BSW}. There also exists a quantum version of the Heisenberg category. 
We also mention that quantum versions of the polar Brauer categories already appeared in \cite{ILZ}, 
which are subcategories of the un-oriented tangle categories \cite{RT, T} 
involving the braid group of type B. 

\subsubsection{Summary of categories to be defined}\label{sss:list}
 To close this section, we provide a list of the principal categories we shall define, 
 for the reader's orientation and convenience.
Recall that we are given a base field $\K$, sets $\CM\supseteq\CV$ and a parameter set $\delta_{\CV}=\{\delta_v\mid v\in\CV\}$. The objects of all the categories below are certain sequences
$\mathbb{s}=(s_1,s_2,\dots,s_r)$ with $s_i\in\CM$

The categories we shall define are:
\begin{itemize}
\item The coloured Brauer category $\CBC(\delta_\CV)$ (cf. Definition \ref{def:cb}).
\item The enhancement $\CBC_\CG(\delta_\CV)$ of  $\CBC(\delta_\CV)$ obtained by adding ``coupons'' $\CG$ (Definition \ref{def:CBwC}).
\item Let $\CH=\{\BH(a,b)\mid a,b\in\CM\}$ (these are ``infinitesimal braids'') and assume $\CG\supseteq\CH$. Write $\CG=\CH\amalg\CG_1$. 
We then define the coloured Brauer category with infinitesimal braids $\HB_{\CG_1}(\delta_{\CV}):=\CBC_{\CG}(\delta_{\CV})/\CJ$, where $\CJ$
is a certain ideal depending on $\CG_1$ (cf. Definition \ref{def:CB-infB}).
\item $\HB(\delta_\CV)$ is the special case of $\HB_{\CG_1}(\delta_{\CV})$ where $\CG_1=\emptyset$. 
Also $\HB(\delta)$ is the special case of $\HB(\delta_{\CV})$
when $\CV=\{v\}$, a singleton.
\item $\wh\HB(\delta)=\HB(\delta)/\CJ$, where $\CJ$ is an ideal like the one above (Definition \ref{def:mpolar}).
\item $\MPB(\delta)$, the multipolar Brauer category, is the full subcategory of $\wh\HB(\delta)$ where $\CM=\{m,v\}$ and $\CV=\{v\}$ (see \S\ref{sect:mpolar}).
\item $\UPB(\delta)$, the unipolar Brauer category, is the full subcategory of $\MPB(\delta)$ with objects $(v^r,m,v^s)$, $r,s=0,1,2,\dots$. (\S\ref{ss:polar})
\item $\PB(\delta)$, the polar Brauer category, is the full subcategory of $\UPB(\delta)$ with objects $(m,v^r)$, $r=0,1,2,\dots$. ({\it loc. cit.})

\item $\MTL(\delta)=\MPB(\delta)/\langle \Theta\rangle$, the multi-polar Temperley-Lieb category, 
where the tensor ideal $\langle \Theta\rangle$ is generated by the morphism $\Theta$ given by Figure \ref{fig:TL}.

\item $\ATL(\delta)$,  the polar Temperley-Lieb category, which is the full subcategory of $\MTL(\delta)$ with objects $(m, v^r)$ for all $r\in\N$. 
\end{itemize}

\section{Coloured Brauer categories with infinitesimal braids}
%
%

%
%
%
%
%
%
\subsection{Coloured Brauer categories with coupons}\label{sect:CB} 

The material presented in this subsection may be regarded as a specialisation of the theory of 
coloured ribbon graphs (with coupons) \cite{RT, T} to the Brauer category setting. 

We fix a finite set $\CM$ with a subset $\CV$, and let $\K$ be a field. Choose a fixed scalar $\delta_v\in \K$ for each $v\in\CV$ and write $\delta_\CV=\{\delta_v\mid v\in\CV\}$. 
Given a Brauer $(k, \ell)$-diagram $D$ (without free loops) \cite{LZ15}, we mark
the end points at the bottom of $D$ by $1, 2, \dots, k$ from left to right, 
and similarly the top end points by $1, 2, \dots, \ell$ from left to right. 
We say that an arc  is {\em horizontal} if its end points are either 
both at the top or at the bottom of the diagram, 
and is {\em vertical} if it is not horizontal. 

We ``colour'' the Brauer  diagram $D$ as follows. 
We associate each arc of $D$ with an element of $\CM$ (which will be called the colour of the arc) in such a way that horizontal arcs are associated with elements of $\CV$. 
Then the $i$-th bottom end point (resp. $j$-th top end point)  is an end point of a string coloured by some $c_i$ (resp. $c'_j$). Denote the resulting diagram by $\wt{D}$, and let $\bs(\wt{D})=(c_1, c_2, \dots, c_k)$ and $\bt(\wt{D})=(c'_1, c'_2, \dots, c'_k)$. Call $\wt{D}$ a {\em coloured Brauer $(k, \ell)$-diagram} with source $\bs(\wt{D})$ and target $\bt(\wt{D})$.
For any sequences $\bs, \bt$ of elements of $\CM$, 
we denote by $\Hom(\bs, \bt)$ the vector space with basis consisting of 
the coloured Brauer diagrams with source $\bs$ and target $\bt$. 

Given any two coloured Brauer diagrams $A$ and $B$, if $\bt(A)=\bs(B)$, we place $B$ above $A$ and then join the bottom end points of $B$ with the top end points of $A$. This gives rise to a diagram which may contain free loops, each of which is coloured by some element of $\bt(A)=\bs(B)$. We delete these free loops to obtain a coloured Brauer diagram denoted by $B\# A$. If there are $f$ free loops, which are respectively coloured by the elements $\bs(B)_{i_1}$, $\bs(B)_{i_2}$, $\dots$, $\bs(B)_{i_f}$ of $\bs(B)$, we set 
\beq
B\circ A = \prod_{p=1}^f \delta_{\bs(B)_{i_p}} B\# A, 
\eeq
and call this element of $\Hom(\bs(A), \bt(B))$ the composition of $B$ and $A$. We will also write  $B A$ for $B\circ A$ for simplicity.

The juxtaposition $A\ot B$ of two coloured Brauer diagrams $A$ and $B$ is the coloured Brauer diagram obtained by placing $A$ on the left of $B$. 

\begin{definition}\label{def:cb} 
Denote by $\CBC(\delta_\CV)$ the $\K$-linear category whose objects are the sequences (including the empty sequence) of elements of $\CM$, and where $\Hom(\bs, \bt)$ 
is the space of morphisms from an object $\bs$ to an object $\bt$.  The composition of morphisms is defined by the bilinear extension of the  composition of coloured 
Brauer diagrams. Call $\CBC(\delta_\CV)$ a coloured Brauer category. 
\end{definition}

Note that all arcs coloured by elements of $\CM\backslash\CV$ are vertical. 
The following  fact is also easy to see.
\begin{theorem}
The coloured Brauer category $\CBC(\delta_\CV)$ has a strict monoidal structure given by the bi-functor 
$\ot: \CBC(\delta_\CV)\times \CBC(\delta_\CV)
\lra \CBC(\delta_\CV)$, which is defined as follows: 
\begin{enumerate}
\item
for any objects $\bs, \bs'$, we have $\bs\ot \bs'=(\bs, \bs')$, where $(\bs, \bs')$ is the sequence of colours 
obtained by joining $\bs'$ to the back of $\bs$; and 
\item for morphisms, $\ot$ is the bi-linear extension of 
juxtaposition of coloured Brauer diagrams.
\end{enumerate}
\end{theorem}

We have the following presentation of the category. 
\begin{theorem} \label{thm:tensor-cat}
As a monoidal category, the objects of $\CBC(\delta_\CV)$ are generated by elements of $\CM$, and the morphisms are generated by $I^a$, $X_{bc}$, $\cap_v$ and $\cup^v$, 
for all $a, b\in \CM$ and $v\in\CV$, which are respectively depicted by the following diagrams: 
\[
\begin{picture}(40, 40)(0,0)
\put(0, 0){\line(0, 1){40}}
\put(5, 0){,}
\put(-10, 20){$a$}
\end{picture}
\begin{picture}(40, 40)(0,0)
\qbezier(0, 40) (0, 40) (25, 0) 
\qbezier(25, 40) (25, 40) (0, 0) 
\put(30, 0){, }
\put(25, 30){$b$ }
\put(-5, 30){$c$ }
\end{picture}
\begin{picture}(120, 40)(-5,0)
\qbezier(20, 0)(40, 60)(60, 0)
\put(65, 0){, }
\put(38, 18){$v$}
\qbezier(95, 40)(115, -20)(135, 40)
\put(140, 0){. }
\put(112, 15){$v$}
\end{picture}
\]
The defining relations among the generators of morphisms are the following. 
\begin{enumerate}
\item Over and under crossings are inverses of each other: 
for all $a, b\in\CM$, 
\[
\setlength{\unitlength}{0.3mm}
\begin{picture}(70, 60)(65, 0)
\qbezier(60, 0)(100, 30)(60, 60)
\qbezier(80, 0)(40, 30)(80, 60)

\put(50, 30){$b$}
\put(85, 30){$a$}
\put(105, 30){$=$}
\end{picture}
\begin{picture}(50, 60)(110, 0)
\put(120, 0){\line(0, 1){60}}
\put(135, 0){\line(0, 1){60}}
\put(110, 30){$a$}
\put(140, 30){$b$}
\put(155, 5){$;$}
\end{picture}
\]

\item Braid relation:  for all $a, b, c\in\C$,

\[
\setlength{\unitlength}{0.3mm}
\begin{picture}(140, 60)(30, 0)
\qbezier(30, 60)(70, 40)(80, 0)
\qbezier(30, 0)(70, 20)(80, 60)
\qbezier(45, 60)(45, 60)(45, 0)

\put(25, 50){$c$}
\put(70, 50){$a$}
\put(36, 30){$b$}

\put(95, 30){$=$}

\qbezier(120, 60)(130, 20)(160, 0)
\qbezier(120, 0)(130, 40)(160, 60)
\qbezier(150, 60)(150, 60)(150, 0)

\put(160, 50){$a$}
\put(112, 50){$c$}
\put(155, 30){$b$}
\end{picture}
\]

\item Straightening relations: for all $v\in\CV$, 
\[
\setlength{\unitlength}{0.25mm}
\begin{picture}(150, 80)(0,0)

\qbezier(0, 0)(10, 80)(20, 30)
\qbezier(20, 30)(30, -30)(40, 70)
\put(50, 30){$=$}

\put(10, 55){$v$}


\put(70, 0){\line(0, 1){70}}

\put(85, 30){$=$}
\put(75, 55){$v$}


\qbezier(105, 70)(115, -30)(125, 30)
\qbezier(125, 30)(135, 80)(145, 0)
\put(135, 55){$v$}
\put(155, 0){;}
\end{picture}
\]

\item Sliding relations: for all $a\in\CM$ and $v\in\CV$, 

\[
\baln
&\setlength{\unitlength}{0.3mm}
\begin{picture}(80, 40)(0,0)
\qbezier(0, 40)(0, 40)(25, 0)
\qbezier(0, 0)(35, 50)(60, 0)
\put(-8, 35){$a$}
\put(32, 15){$v$}
\put(75, 15){$=$}
\end{picture}
\begin{picture}(80, 40)(-20,0)
\qbezier(60, 40)(60, 40)(35, 0)
\qbezier(0, 0)(25, 50)(60, 0)
\put(62, 35){$a$}
\put(22, 15){$v$}
\put(70, 0){,}
\end{picture}
\\
&\\
&\setlength{\unitlength}{0.3mm}
\begin{picture}(80, 40)(0,0)
\qbezier(20, 40)(20, 40)(0, 0)
\qbezier(0, 40)(35, -10)(60, 40)
\put(23, 35){$a$}
\put(32, 6){$v$}
\put(75, 15){$=$}
\end{picture}
\begin{picture}(80, 40)(-20,0)
\qbezier(40, 40)(40, 40)(60, 0)
\qbezier(0, 40)(35, -10)(60, 40)
\put(32, 35){$a$}
\put(32, 6){$v$}
\put(70, 0){;}
\end{picture}
\ealn
\]
 
 \item Twists: for all $v\in\CV$, 
 
 \[
 \setlength{\unitlength}{0.3mm}
 \begin{picture}(80, 40)(0,0)
 \qbezier(0, 40)(0, 40)(40, 0)
 \qbezier(0, 0)(0, 0)(40, 40)

\qbezier(0, 40)(-30, 20)(0, 0)

\put(-15, 40){$v$}
\put(55, 20){$=$}
\end{picture}
\begin{picture}(50, 20)(0,0)
\put(0, 0){\line(0, 1){40}}
\put(0, 40){$v$}
\put(20, 20){$=$}
\end{picture}
 \begin{picture}(40, 40)(0,0)
 \qbezier(0, 40)(0, 40)(40, 0)
 \qbezier(0, 0)(0, 0)(40, 40)

\qbezier(40, 40)(70, 20)(40, 0)

\put(55, 40){$v$}
\put(55, 5){$.$}
\end{picture}
 \]
\end{enumerate}
\end{theorem}
\begin{proof} 
This can be proved along the lines of the proof of the corresponding result for the usual Brauer category given in \cite{LZ15}. We omit the details. 
\end{proof}

We will need the following notation.  For any object $\bs=(c_1, c_2, \dots, c_k)$ 
of $\CBC(\delta_\CV)$, we denote by $|\bs|=k$ its length, and let 
\[
\bs^\vee=(c_k, c_{k-1}, \dots, c_1), \quad
\bs_j^\vee = (c_1, \dots, c_{j-1}, c_{j+1}, c_{j}, c_{j+2}, \dots, c_k), \   j<k.
\]

An advantage of the construction of $\CBC(\delta_\CV)$ above is its simplicity. It is also very flexible and amenable to generalisations. 
We now generalise it by including additional generators of morphisms (called coupons) as in the treatment of tangle categories in \cite[\S 4.5]{RT} and \cite[\S 1.3]{T}.  

Define a $\K$-linear strict monoidal category as follows. The objects are the sequences of elements of $\CM$; and the generators of morphisms are those of $\CBC(\delta_\CV)$
together with the elements of $\CG=\{g_{\bs \bt}: \bs \to \bt\mid \bs\in \mathbf{S}, \bt \in \mathbf{T}\}$, called {\em coupons}, where $\mathbf{S}$ and  $\mathbf{T}$ 
are two fixed sets of sequences of elements of $\CM$.   
Represent $g_{\bs \bt}$ graphically as 
\[
\begin{picture}(40, 55)(5,20)
\put(10, 60){\line(0, 1){20}}
\put(25, 60){\line(0, 1){20}}
\put(13, 65){\tiny$\dots$}
\put(15, 70){\tiny$\bt$}

\put(0, 60){\line(1, 0){35}}
\put(0, 40){\line(1, 0){35}}
\put(0, 60){\line(0, -1){20}}
\put(35, 60){\line(0, -1){20}}
\put(10, 50){$g_{\bs \bt}$}

\put(10, 40){\line(0, -1){20}}
\put(25, 40){\line(0, -1){20}}
\put(13, 34){\tiny$\dots$}
\put(15, 25){\tiny$\bs$}
\put(45, 25){.}
\end{picture}
\]
Impose the following relations on the generating morphisms:
\begin{enumerate}
\item Brauer relations:  the relations of $\CBC(\delta_\CV)$; 

\item duality of coupons:  the relations depicted in Figure \ref{fig:gvee} if the components of $\bs\in\mathbf{S}$ and $\bt\in\mathbf{T}$ all belong to $\CV$, and $\bt^\vee\in\mathbf{S}$ and $\bs^\vee\in\mathbf{T}$; and  

\item sliding arcs over coupons: the relations depicted in Figure \ref{fig:coupon} for all $\bs\in \mathbf{S}, \bt\in \mathbf{T}$ and any object $\bc$.  
\end{enumerate}

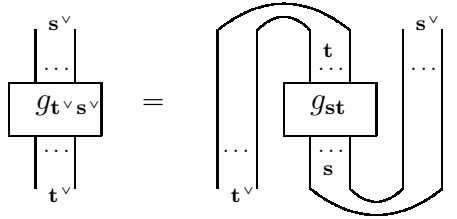
\begin{figure}[h]
\begin{picture}(100, 70)(5,10)
\put(10, 60){\line(0, 1){20}}
\put(25, 60){\line(0, 1){20}}
\put(13, 65){\tiny$\dots$}
\put(15, 80){\tiny$\bs^\vee$}

\put(0, 60){\line(1, 0){35}}
\put(0, 40){\line(1, 0){35}}
\put(0, 60){\line(0, -1){20}}
\put(35, 60){\line(0, -1){20}}
\put(10, 50){$g_{\bt^\vee \bs^\vee}$}

\put(10, 40){\line(0, -1){20}}
\put(25, 40){\line(0, -1){20}}
\put(13, 34){\tiny$\dots$}
\put(15, 15){\tiny$\bt^\vee$}
\put(50, 50){$=$}
\end{picture}
\begin{picture}(40, 70)(5,10)
\put(-10, 80){\line(0, -1){60}}
\put(-25, 80){\line(0, -1){60}}
\qbezier(-10, 80)(0, 90)(10, 80)
\qbezier(-25, 80)(0, 100)(25, 80)

\put(-23, 34){\tiny$\dots$}
\put(-20, 15){\tiny$\bt^\vee$}

\put(10, 60){\line(0, 1){20}}
\put(25, 60){\line(0, 1){20}}
\put(13, 65){\tiny$\dots$}
\put(15, 70){\tiny$\bt$}

\put(0, 60){\line(1, 0){35}}
\put(0, 40){\line(1, 0){35}}
\put(0, 60){\line(0, -1){20}}
\put(35, 60){\line(0, -1){20}}
\put(10, 50){$g_{\bs \bt}$}

\put(10, 40){\line(0, -1){20}}
\put(25, 40){\line(0, -1){20}}
\put(13, 34){\tiny$\dots$}
\put(15, 25){\tiny$\bs$}

\put(45, 80){\line(0, -1){60}}
\put(60, 80){\line(0, -1){60}}
\qbezier(25, 20)(35, 10)(45, 20)
\qbezier(10, 20)(35, 0)(60, 20)

\put(48, 65){\tiny$\dots$}
\put(50, 80){\tiny$\bs^\vee$}

\put(65, 20){.}
\end{picture}
\caption{Duality of coupons}
\label{fig:gvee}\label{fig:dual-g}
\end{figure}

\begin{figure}[h]
\begin{picture}(100, 80)(-25,20)
\put(10, 60){\line(0, 1){35}}
\put(25, 60){\line(0, 1){35}}
\put(13, 65){\tiny$\dots$}
\put(15, 95){\tiny$\bt$}

\put(0, 60){\line(1, 0){35}}
\put(0, 40){\line(1, 0){35}}
\put(0, 60){\line(0, -1){20}}
\put(35, 60){\line(0, -1){20}}
\put(10, 50){$g_{\bs \bt}$}

\put(10, 40){\line(0, -1){25}}
\put(25, 40){\line(0, -1){25}}
\put(13, 34){\tiny$\dots$}
\put(15, 22){\tiny$\bs$}

\qbezier(40, 95)(-20, 75)(-25, 50)
\qbezier(55, 90)(-5, 75)(-10, 50)
\put(-25, 50){\line(0, -1){35}}
\put(-10, 50){\line(0, -1){35}}
\put(-23, 34){\tiny$\dots$}
\put(-20, 15){\tiny$\bc$}

\put(60, 45){$=$}
\end{picture}
\quad 
\begin{picture}(80, 80)(-25,15)
\put(10, 60){\line(0, 1){25}}
\put(25, 60){\line(0, 1){25}}
\put(13, 65){\tiny$\dots$}
\put(15, 85){\tiny$\bt$}

\put(0, 60){\line(1, 0){35}}
\put(0, 40){\line(1, 0){35}}
\put(0, 60){\line(0, -1){20}}
\put(35, 60){\line(0, -1){20}}
\put(10, 50){$g_{\bs \bt}$}

\put(10, 40){\line(0, -1){35}}
\put(25, 40){\line(0, -1){35}}
\put(13, 34){\tiny$\dots$}
\put(15, 3){\tiny$\bs$}

\put(45, 50){\line(0, 1){35}}
\put(60, 50){\line(0, 1){35}}

\qbezier(45, 50)(45, 25)(-20, 10)
\qbezier(60, 50)(50, 20)(-5, 5)

\put(48, 65){\tiny$\dots$}
\put(50, 85){\tiny$\bc$}

\put(65, 10){,}
\end{picture}

\begin{picture}(110, 110)(-25,20)
\put(10, 60){\line(0, 1){35}}
\put(25, 60){\line(0, 1){35}}
\put(13, 65){\tiny$\dots$}
\put(15, 95){\tiny$\bt$}

\put(0, 60){\line(1, 0){35}}
\put(0, 40){\line(1, 0){35}}
\put(0, 60){\line(0, -1){20}}
\put(35, 60){\line(0, -1){20}}
\put(10, 50){$g_{\bs \bt}$}

\put(10, 40){\line(0, -1){25}}
\put(25, 40){\line(0, -1){25}}
\put(13, 34){\tiny$\dots$}
\put(15, 22){\tiny$\bs$}

\qbezier(-10, 95)(40, 85)(60, 50)
\qbezier(-25, 90)(40, 75)(45, 50)
\put(60, 50){\line(0, -1){35}}
\put(45, 50){\line(0, -1){35}}
\put(48, 34){\tiny$\dots$}
\put(50, 15){\tiny$\bc$}

\put(70, 45){$=$}
\end{picture}
\quad 
\begin{picture}(80, 110)(-25,15)
\put(10, 60){\line(0, 1){25}}
\put(25, 60){\line(0, 1){25}}
\put(13, 65){\tiny$\dots$}
\put(15, 90){\tiny$\bt$}

\put(0, 60){\line(1, 0){35}}
\put(0, 40){\line(1, 0){35}}
\put(0, 60){\line(0, -1){20}}
\put(35, 60){\line(0, -1){20}}
\put(10, 50){$g_{\bs \bt}$}

\put(10, 40){\line(0, -1){35}}
\put(25, 40){\line(0, -1){35}}
\put(13, 34){\tiny$\dots$}
\put(15, 3){\tiny$\bs$}
\put(-10, 50){\line(0, 1){35}}
\put(-25, 50){\line(0, 1){35}}
\put(-22, 65){\tiny$\dots$}
\put(-20, 90){\tiny$\bc$}

\qbezier(-10, 50)(0, 15)(60, 15)
\qbezier(-25, 50)(-15, 15)(45, 7)

\put(65, 10){.}
\end{picture}
\caption{Sliding arcs over coupons}
\label{fig:coupon}
\end{figure}
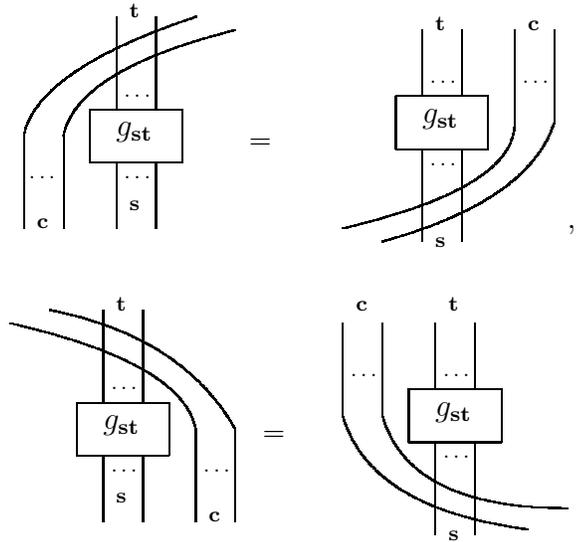

\begin{definition}\label{def:CBwC}
Denote this strict monoidal category by $\CBC_\CG(\delta_\CV)$, and call it a
coloured Brauer category with coupons. 
\end{definition}

\begin{remark}\label{rmk:quot}
We shall consider coloured Brauer categories with explicitly given coupons satisfying further relations.   
\end{remark}

We now introduce the following symbols to denote the particular morphisms depicted in the corresponding diagrams below. 
\begin{enumerate}[(a).]
\item The identity endomorphism $I(\bs): \bs\to\bs$. 
\[
\begin{picture}(70, 25)(5,5)
\put(-45, 10){$I(\bs)\ = \ $}
\put(10, 0){\line(0, 1){20}}
\put(25, 0){\line(0, 1){20}}
\put(13, 10){\tiny$\dots$}
\put(15, 22){\tiny$\bs$}
\put(30, 8){.}
\end{picture}
\]

\item Permutation 
$X_j(\bs): \bs \to \bs_j^\vee$. 
\[
\begin{picture}(175, 30)(5,0)
\put(-40, 10){$X_j(\bs)=$}
\put(10, 0){\line(0, 1){20}}
\put(6, 22){\tiny$c_1$}

\put(15, 10){\tiny$\dots$}

\put(30, 0){\line(0, 1){20}}
\put(24, 22){\tiny$c_{j-1}$}

\qbezier(50, 0)(50, 0)(80, 20)
\qbezier(50, 20)(50, 20)(80, 0)
\put(46, 22){\tiny$c_{j+1}$}
\put(76, 22){\tiny$c_{j}$}

\put(100, 0){\line(0, 1){20}}
\put(94, 22){\tiny$c_{j+2}$}

\put(105, 10){\tiny$\dots$}

\put(125, 0){\line(0, 1){20}}
\put(120, 22){\tiny$c_k$}

\put(130, 0){, }

\put(150, 10){$j < k$. }

\end{picture}
\]

\item Multiple arc permutation $X(\ba, \bs): (\ba, \bs)\to (\bs, \ba)$.
\[
\begin{picture}(150, 50)(-50,0)
\put(-50, 20){$X(\ba, \bs) \ = \ $ }

\put(86, 48){\tiny$\ba$}
\put(28, 48){\tiny $\bs$}
\qbezier(20, 45)(20, 45)(80, 5)
\qbezier(40, 45)(40, 45)(100, 5)
\qbezier(80, 45)(20, 5)(20, 5)
\qbezier(100, 45)(40, 5)(40, 5)

\put(42, 15){\tiny$\dots$}
\put(68, 15){\tiny$\dots$}

\put(26, 0){\tiny$\ba$}
\put(88, 0){\tiny $\bs$}
\end{picture}
\]

\item Multiple arc cap $\Cap(\bs):  (\bs, \bs^\vee)\to \emptyset$ and cup  $\Cup(\bs):  \emptyset\to (\bs, \bs^\vee)$. 
\[
\begin{picture}(150, 30)(-50,0)
\put(-55, 8){$\Cap(\bs) \ = \ $ }
\qbezier(0, 5)(35, 40)(70, 5)

\qbezier(10, 5)(40, 30)(60, 5)
\put(18, -3){\tiny $\bs$}

\put(16, 6){\tiny$\dots$}
\put(43, 6){\tiny$\dots$}

\qbezier(30, 5)(35, 20)(40, 5)
\put(48, -3){\tiny$\bs^\vee$}
\put(75, 5){, }
\end{picture}
\begin{picture}(80, 30)(-55,0)
\put(-55, 8){$\Cup(\bs) \ = \ $ }

\qbezier(0, 15)(35, -20)(70, 15)
\qbezier(10, 15)(40, -10)(60, 15)
\qbezier(30, 15)(35, 0)(40, 15)

\put(18, 18){\tiny $\bs$}
\put(48, 18){\tiny$\bs^\vee$}

\put(18, 10){\tiny$\dots$}
\put(42, 10){\tiny$\dots$}

\put(75, 5){. }
\end{picture}
\]
\end{enumerate}

\subsection{An example: the enhanced Brauer category for $\fso_m$}\label{ex:B-SO}
A very simple example of $\CBC_\CG(\delta_\CV)$ yields the category needed to describe the invariant theory of the special orthogonal group $\rm{SO}_m$. 
We take $\CV=\CM=\{v\}$ with a single colour $v$, and write $\delta$ for $\delta_\CV$ (we adopt this notation whenever $|\CV|=1$). 
Represent a sequence $(v, \dots, v)$ of length $r$ simply by $r$, thus the objects of the category are now the non-negative integers as in the usual Brauer category $\CB(\delta)$. 
Further assume that $\CG$ contains one coupon $\Delta_m: 0\to m$ only for a fixed integer $m>1$, which we depict pictorially as in Figure \ref{fig:Delta}. 
\begin{figure}[h]
\begin{picture}(60, 25)(-30,15)
\put(-38, 20){$\Delta_m=$}
\put(10, 20){\line(0, 1){20}}
\put(25, 20){\line(0, 1){20}}
\put(13, 25){\tiny$\dots$}
\put(15, 30){\tiny$m$}

\put(0, 20){\line(1, 0){35}}

\qbezier(0, 20)(18, -10)(35, 20)
\end{picture}
\caption{The coupon in $\CB_{\Delta_m}(\delta)$}
\label{fig:Delta}
\end{figure}
Denote the category by $\CB_{\Delta_m}(\delta)$. 

The enhanced Brauer category \cite[Definition 5.1]{LZ17-Ecate} can be obtained as a quotient category of $\CB_{\Delta_m}(\delta)$ (for appropriate parameter $\delta$), which we now describe.

Since the usual Brauer category $\CB(\delta)$ is contained in $\CB_{\Delta_m}(\delta)$ as a monoidal subcategory, $\Hom_{\CB_{\Delta_m}(\delta)}(m, m)$ contains the Brauer algebra $B_m(\delta)=\Hom_{\CB(\delta)}(m, m)$ as a subalgebra. 
Recall also that $B_m(\delta)$ contains the group algebra of $\Sym_m$ as a subalgebra. Represent $\Sigma(m) = \sum_{\sigma\in\Sym_m}(-1)^{\ell(\sigma)}\sigma$ by the diagram below.
\[
\begin{picture}(30, 50)(60,0)
\put(60, 15){\line(0, 1){20}}
\put(90, 15){\line(0, 1){20}}
\put(60, 35){\line(1, 0){30}}
\put(60, 15){\line(1, 0){30}}

\put(65, 15){\line(0, -1){15}}
\put(85, 15){\line(0, -1){15}}

\put(65, 35){\line(0, 1){15}}
\put(85, 35){\line(0, 1){15}}

\put(70, 40){\tiny$\dots$}
\put(70, 7){\tiny$\dots$}
\put(70, 0){\tiny$m$}
\put(70, 45){\tiny$m$}

\put(65, 22){\tiny$\Sigma(m)$}

\end{picture}
\]

Now an element $D\in B_m(\delta)$ may be composed with $\Delta_m$,  yielding $D \Delta_m$ as a morphism in $\CB_{\Delta_m}(m)$. Let $\Delta^*=\Cap(m) (\Delta_m \ot I(m))$, then we also have $\Delta_m^* D$ for any $D\in B_m(\delta)$. Represent $\Delta_m^*$ by the diagram below.

\[
\begin{picture}(25, 25)(50,0)

\put(50, 20){\line(1, 0){35}}
\put(60, 20){\line(0, -1){20}}
\put(75, 20){\line(0, -1){20}}
\qbezier(50, 20)(67, 45)(85, 20)

\put(62, 10){\tiny$\dots$}
\put(63, 2){\tiny$m$}
\end{picture}
\]

We impose the following conditions on the coupon,   
\begin{enumerate}
\item $\Delta_m$ is skew symmetric, i.e., it spans  the sign representation of 
$\Sym_m$; and 

\item $\Sigma(m)=\Delta_m \Delta_m^*$, which can be depicted diagrammatically as 
\[
\begin{picture}(70, 50)(60,0)
\put(60, 15){\line(0, 1){20}}
\put(90, 15){\line(0, 1){20}}
\put(60, 35){\line(1, 0){30}}
\put(60, 15){\line(1, 0){30}}

\put(65, 15){\line(0, -1){15}}
\put(85, 15){\line(0, -1){15}}

\put(65, 35){\line(0, 1){15}}
\put(85, 35){\line(0, 1){15}}

\put(70, 40){\tiny$\dots$}
\put(70, 7){\tiny$\dots$}
\put(70, 0){\tiny$m$}
\put(70, 45){\tiny$m$}

\put(65, 22){\tiny$\Sigma(m)$}

\put(100, 22){$=$}

\put(120, 35){\line(1, 0){30}}
\put(120, 15){\line(1, 0){30}}

\put(125, 15){\line(0, -1){15}}
\put(145, 15){\line(0, -1){15}}

\put(125, 35){\line(0, 1){15}}
\put(145, 35){\line(0, 1){15}}

\qbezier(120, 35)(135, 15)(150, 35)
\qbezier(120, 15)(135, 30)(150, 15)

\put(130, 40){\tiny$\dots$}
\put(130, 7){\tiny$\dots$}
\put(130, 0){\tiny$m$}
\put(130, 45){\tiny$m$}

\put(155, 5){.}
\end{picture}
\] 
\end{enumerate}
Denote the resulting monoidal category by $\CB'_{\Delta_m}$.

Note that the first condition implies that $\Delta_m$ is harmonic, i.e., $(\cap\ot I(m-2))\sigma \Delta_m=0$ for all $\sigma\in\Sym_m$.  The arguments in \cite{LZ17-Ecate}  
show that  that condition (2) above implies that $\delta=m$.

By inspecting the definition  of the enhanced Brauer category $\wt{\CB}(m)$ given in \cite[Definition 5.1]{LZ17-Ecate}, one immediately sees the following fact.  
\begin{theorem}
The enhanced Brauer category $\CB'_{\Delta_m}$ is isomorphic as monoidal category  to the enhanced Brauer category $\wt{\CB}(m)$ defined in \cite[Definition 5.1]{LZ17-Ecate}.
\end{theorem}
\begin{remark}
The category $\CB'_{\Delta_m}$is isomorphic  as monoidal category   to 
the full subcategory $\CT(\C^m)$ of $\rm{SO}_m(\C)$-Mod with objects $(\C^m)^{\ot r}$ for all $r\in\N$, since $\wt\CB(m)$ is known  \cite{LZ17-Ecate} to be isomorphic to $\CT(\C^m)$. 
\end{remark}

\medskip
\subsection{Infinitesimal braids as coupons}
\subsubsection{Coloured Brauer categories with infinitesimal braids}
We shall discuss a particular set of coupons in this section.
Assume that $\CG\supset \CH=\{\BH(a, b): (a, b)\to (a, b)\mid a, b\in\CM\}$.  
We represent $\BH(a, b)$ diagrammatically  as two vertical arcs connected by a wavy line which will be referred to as a {\em connector}, as shown in Figure \ref{fig:BH}.
\begin{figure}[h]
\begin{picture}(30, 40)(0,0)

\put(2, -8){\tiny$a$}
\put(2, 42){\tiny$a$}
\put(5, 0){\line(0, 1){40}}

\put(5, 20){\uwave{\hspace{7mm}}}

\put(23, -8){\tiny$b$}
\put(23, 42){\tiny$b$}
\put(25, 0){\line(0, 1){40}}

\end{picture} 
\caption{The coupon $\BH(a, b)$}
\label{fig:BH}
\end{figure}

By the duality of coupons depicted in Figure \ref{fig:dual-g}, we have, for all $a, b\in\CV$, 
\[
\BH(b, a)=(\Cap(b, a)\ot I(b,a)) (I(b, a)\ot \BH(a, b) \ot I(b, a)) (I(b, a)\ot \Cup(a,b)).
\]

For any $a, b, c$ in $\CM$, we introduce the notation 
\beq
\BH_{1 2}(a, b, c) &:=& \BH(a, b)\ot I^c, \quad \BH_{2 3}(a, b, c):= I^a\ot \BH(b, c).  \\
\BH_{1 3}(a, b, c) &:=& (I^a\ot X_{c b}) \BH_{12}(a, c, b) (I^a \ot X_{b c}). 
\eeq

By sliding the single arc $c$ over the coupon $\BH(a,b)$, as shown in Figure \ref{fig:coupon},  one sees that 
\beq\label{eq-basic hslide}
(I^a\ot X_{cb})(X_{ca}\ot I^b)(\BH_{23}(c,a,b))=(\BH_{12}(a,b,c))(I^a\ot X_{cb})(X_{ca}\ot I^b),
\eeq
Noting that $(I^a\ot X_{cb})\inv = (I^a\ot X_{bc})$ etc, it follows that \eqref{eq-basic hslide}
may be written
\beq\label{eq:conj}
(I^a\ot X_{bc})\BH_{1 2}(a, b, c) ( I^a\ot X_{cb})=(X_{ca}\ot I^b)\BH_{23}(c,a, b) (X_{ca}\ot I^b) . 
\eeq

It is instructive to give the diagrammatic version of \eqref{eq:conj}, as in Figure \ref{fig:conj}

%


\begin{figure}[h]
\begin{tikzpicture}
  \begin{scope}
    \draw (0,1.5) node[left] {$a$} -- (0,0.5) node[left] {$a$} -- (0,-0.5) node[left] {$a$}-- (0,-1.5) node[left] {$a$};%
    \draw (1,1.5) node[left] {$c$} -- (2,0.5) node[right] {$c$} -- (2,-0.5) node[right] {$c$}-- (1,-1.5) node[left] {$c$};%
    \draw (2,1.5) node[right] {$b$} -- (1,0.5) node[left] {$b$} -- (1,-0.5) node[left] {$b$}-- (2,-1.5) node[right] {$b$};%

    \draw[decorate,decoration={snake,amplitude=.4mm,segment length=1.5mm,post length=1mm}] (0,0) -- (1,0);
    {\color{red}
 \draw[densely dotted] (2.5,-0.5) -- (-0.5,-0.5);%
 \draw[densely dotted] (2.5,0.5) -- (-0.5,0.5);%
 }

    \foreach \x in {0,1,2}{%
      \foreach \y in {1.5,0.5,...,-1.5}{%
        \draw[fill] (\x,\y) circle[radius=1.25pt];%
      }%
    }%
  \end{scope}

  \begin{scope}[xshift=4cm]
    \draw (0,1.5) node[left] {$a$} -- (1,0.5) node[right] {$a$} -- (1,-0.5) node[right] {$a$}-- (0,-1.5) node[left] {$a$};%
    \draw (1,1.5) node[right] {$c$} -- (0,0.5) node[left] {$c$} -- (0,-0.5) node[left] {$c$}-- (1,-1.5) node[right] {$c$};%
    \draw (2,1.5) node[right] {$b$} -- (2,0.5) node[right] {$b$} -- (2,-0.5) node[right] {$b$}-- (2,-1.5) node[right] {$b$};%
    \draw(-.9,0)node{$=$};

    \draw[decorate,decoration={snake,amplitude=.4mm,segment length=1.5mm,post length=1mm}] (2,0) -- (1,0);%

{\color{red}
 \draw[densely dotted] (2.5,-0.5) -- (-0.5,-0.5);%
 \draw[densely dotted] (2.5,0.5) -- (-0.5,0.5);%
 }

    \foreach \x in {0,1,2}{%
      \foreach \y in {1.5,0.5,...,-1.5}{%
        \draw[fill] (\x,\y) circle[radius=1.25pt];%
      }%
    }%
  \end{scope}
\end{tikzpicture}

\caption{The relation $\eqref{eq:conj}$}
\label{fig:conj}
\end{figure}
\noindent
where the horizontal red dotted lines indicate 
how a diagram is a composition of sub-diagrams, 
and black dots indicate the end points of the sub-diagrams. 

It follows from Figure \ref{fig:conj} that we have the following alternative formula for $\BH_{13}(a,b,c)$.
\beq
\BH_{13}(a,b,c)=(X_{ba}\ot I^c)(\BH_{23}(b,a,c))(X_{ab}\ot I^c).
\eeq
This is because both sides of \eqref{eq:conj} are equal to $\BH_{13}(a,c,b)$.

\vskip .75truecm

Now for any $a\in\CM$ and $v\in\CV$, let
\[
\baln
&\BH^{rt}(a, v) :=(I^a\ot\cap_v\ot I^v) (\BH(a, v)\ot X_{v v})(I^a\ot\cup^v\ot I^v), \\
&\BH^{lt}(v, a) :=(I^v\ot\cap_v\ot I^a) (X_{v v}\ot \BH(v, a))(I^v\ot\cup^v\ot I^v).
\ealn
\]

\begin{definition}\label{def:CB-infB}
Retain the notation above, and assume that $\CG=\CH\cup\CG_1$. 
Let $\CJ$ be the tensor ideal generated by the following morphisms for all $a, b, c\in\CM$ and $v\in \CV$:
\[
\baln
&(1).\quad X_{b a} \BH(b, a) - \BH(a, b) X_{b a},\\ 
&(2).\quad \BH^{lt}(v, a) +\BH(v, a), \quad \BH^{rt}(a, v) + \BH(a, v), \\
&(3).\quad [\BH_{1 2}(a, b, c), \BH_{1 3}(a, b, c)+ \BH_{2 3}(a, b, c)], \\
&(4).\quad[\BH_{12}(a, b, c) +\BH_{1 3}(a, b, c), \BH_{2 3}(a, b, c)].
\ealn
\] 
Denote by $\HB_{\CG_1}(\delta_\CV)$ the quotient category $\CBC_\CG(\delta_\CV)/\CJ$,  
and refer to it as a {\em coloured Brauer category with infinitesimal braids}. If $\CG_1=\emptyset$, we denote this category by $\HB(\delta_\CV)$.
\end{definition}

\begin{remark}
A justification of the terminology is that the images of $\BH(a, b)$ under representation theoretic functors have properties of 
 infinitesimal braids (see, e.g., \cite[\S XX]{K}). In the terminology of chord diagrams (see op. cit.), the wavy line in a connector is a chord. 
\end{remark}

We will denote by $X_{a b}$, $\BH(a, b)$, $\BH_{1 2}(a, b, c)$ etc. the respective images in $\HB_{\CG_1}(\delta_\CV)$ of the relevant elements. 
The diagrams in $\HB_{\CG_1}(\delta_\CV)$ obtained by composition of $\BH(a, b)$, $X_{a b}$ and $I^a$ for all $a, b\in \CM$ will be 
referred to as  {\em infinitesimal braids}. 

\begin{remark}
The infinitesimal symmetric category defined by \cite[Definition XX.4.1]{K} is very similar to our $\HB_{\emptyset}(\delta_\CV)$ when $\CV=\CM$, but with the crucial difference that 
it does not have the relations $(2)$ on the generators  of $\CJ$. 
These relations are needed to deal with self-duality of  the objects $v\in \CV$ (i.e., the existence of the morphisms $\cup^v$ and $\cap_v$), 
reflecting the ``un-oriented" nature of $\HB_{\emptyset}(\delta_\CV)$. 
\end{remark}

\subsubsection{Graphical depictions of defining relations}
We refer to the relations in $\HB_{\CG_1}(\delta_\CV)$ arising from the generators $(1)$  of $\CJ$
 as symmetries of $\BH(a, b)$, those arising from generators $(2)$ as left skew symmetries of $\BH(v, a)$ and right 
 skew symmetries of $\BH(a, v)$ respectively, and those  arising from the generators $(3)$ and $(4)$ as the four term relations.  The relations may be described diagrammatically as follows.

\noindent $\bullet$ 
Symmetry of $\BH$ as shown in Figure \ref{fig:BH-S} for all $a, b\in\CM$. 

\begin{figure}[h]

\begin{picture}(60, 60)(0,-10)

\put(2, -18){\tiny$a$}
\put(2, 52){\tiny$a$}
\put(5, 10){\line(0, 1){20}}

\qbezier(5, 30)(5, 30)(25, 50)
\qbezier(5, 50)(5, 50)(25, 30)

\put(5, 20){\uwave{\hspace{7mm}}}

\put(23, -18){\tiny$b$}
\put(23, 52){\tiny$b$}
\put(25, 10){\line(0, 1){20}}

\qbezier(5, 10)(5, 10)(25, -10)
\qbezier(5, -10)(5, -10)(25, 10)

\put(40, 16){$=$}
\end{picture} 
\begin{picture}(30, 60)(0,-10)

\put(2, -18){\tiny$a$}
\put(2, 52){\tiny$a$}
\put(5, -10){\line(0, 1){60}}

\put(5, 20){\uwave{\hspace{7mm}}}

\put(23, -18){\tiny$b$}
\put(23, 52){\tiny$b$}
\put(25, -10){\line(0, 1){60}}

\end{picture} 

\caption{Symmetry of $\BH$}
\label{fig:BH-S}
\end{figure}

\noindent $\bullet$  
Left and right skew symmetries of $\BH$ as shown in Figure \ref{fig:skew} for all $a\in\CM$,  $v\in\CV$, 

\begin{figure}[h]
\setlength{\unitlength}{0.25mm}
\begin{picture}(125, 65)(0,0)

\qbezier(40, 40)(50, 55)(62, 40)
\qbezier(40, 20)(50, 5)(62, 20)

\qbezier(40, 40)(30, 5)(20, 0)
\qbezier(40, 20)(30, 55)(20, 60)

\put(62, 20){\line(0, 1){20}}

\put(18, -8){\tiny$v$}
\put(18, 62){\tiny$v$}

\put(60, 33){{\uwave{\hspace{6mm}}}}
\put(84, 0){\line(0, 1){60}}
\put(82, -8){\tiny$a$}
\put(82, 62){\tiny$a$}

\put(100, 27){$= - $}
\end{picture}  
\begin{picture}(40, 65)(-10,0)
\put(0, 0){\line(0, 1){60}}
\put(-2, -8){\tiny$v$}
\put(-2, 62){\tiny$v$}
\put(-3, 33){{\uwave{\hspace{7mm}}}}

\put(29, 0){\line(0, 1){60}}
\put(27, -8){\tiny$a$}
\put(27, 62){\tiny$a$}
\put(38, 5){; }
\end{picture} 
\quad\quad\quad
\begin{picture}(115, 65)(-5,0)
\put(0, 0){\line(0, 1){60}}
\put(-2, -8){\tiny$a$}
\put(-2, 62){\tiny$a$}
\put(-3, 33){{\uwave{\hspace{6mm}}}}

\qbezier(21, 40)(30, 55)(40, 40)
\qbezier(21, 20)(30, 5)(40, 20)
\put(21, 20){\line(0, 1){20}}

\qbezier(40, 40)(50, 5)(60, 0)
\qbezier(40, 20)(50, 55)(60, 60)

\put(62, -8){\tiny$v$}
\put(62, 62){\tiny$v$}

\put(72, 27){$=\ - $}
\end{picture} 
\begin{picture}(65, 65)(-10,0)
\put(0, 0){\line(0, 1){60}}
\put(-2, -8){\tiny$a$}
\put(-2, 62){\tiny$a$}
\put(-3, 33){{\uwave{\hspace{7mm}}}}

\put(29, 0){\line(0, 1){60}}
\put(27, -8){\tiny$v$}
\put(27, 62){\tiny$v$}
\end{picture} 
\caption{Left and right skew symmetries of $\BH$}
\label{fig:skew}
\end{figure}

\noindent $\bullet$  
Four-term relations as shown in Figure \ref{fig:C-4-term} for all $a, b,c\in\CM$.

\begin{figure}[h]
\begin{picture}(50, 70)(0,0)
\put(8, -8){\tiny$a$}
\put(8, 73){\tiny$a$}
\put(23, -8){\tiny$b$}
\put(23, 73){\tiny$b$}
\put(38, -8){\tiny$c$}
\put(38, 73){\tiny$c$}

\put(-30, 40){$i)$}

{
\put(10, 0){\line(0, 1){70}}
}

\put(10, 50){\uwave{\hspace{5mm}}}
\put(10, 25){\uwave{\hspace{5mm}}}

\put(25, 70){\line(0, -1){30}}
\put(40, 70){\line(0, -1){30}}

\qbezier(25, 40)(25, 40)(40, 30)
\qbezier(25, 30)(25, 30)(40, 40)

\put(25, 30){\line(0, -1){15}}
\put(40, 30){\line(0, -1){15}}
\qbezier(25, 15)(25, 15)(40, 0)
\qbezier(25, 0)(25, 0)(40, 15)
\put(45, 30){$-$}
\end{picture}
\begin{picture}(50, 70)(0,0)
\put(8, -8){\tiny$a$}
\put(8, 73){\tiny$a$}
\put(23, -8){\tiny$b$}
\put(23, 73){\tiny$b$}
\put(38, -8){\tiny$c$}
\put(38, 73){\tiny$c$}
{
\put(10, 0){\line(0, 1){70}}
}

\put(10, 50){\uwave{\hspace{5mm}}}
\put(10, 25){\uwave{\hspace{5mm}}}

\qbezier(25, 55)(25, 55)(40, 70)
\qbezier(25, 70)(25, 70)(40, 55)

\put(25, 55){\line(0, -1){15}}
\put(40, 55){\line(0, -1){15}}

\put(25, 30){\line(0, -1){30}}
\put(40, 30){\line(0, -1){30}}
\qbezier(25, 30)(25, 30)(40, 40)
\qbezier(25, 40)(25, 40)(40, 30)
\put(45, 30){$+$}
\end{picture}
\begin{picture}(50, 70)(0,0)
\put(8, -8){\tiny$a$}
\put(8, 73){\tiny$a$}
\put(23, -8){\tiny$b$}
\put(23, 73){\tiny$b$}
\put(38, -8){\tiny$c$}
\put(38, 73){\tiny$c$}
{
\put(10, 0){\line(0, 1){70}}
}

\put(10, 50){\uwave{\hspace{5mm}}}
\put(25, 25){\uwave{\hspace{5mm}}}


\put(25, 70){\line(0, -1){70}}
\put(40, 70){\line(0, -1){70}}

\put(45, 30){$-$}
\end{picture}
\begin{picture}(50, 70)(0,0)
\put(8, -8){\tiny$a$}
\put(8, 73){\tiny$a$}
\put(23, -8){\tiny$b$}
\put(23, 73){\tiny$b$}
\put(38, -8){\tiny$c$}
\put(38, 73){\tiny$c$}
{
\put(10, 0){\line(0, 1){70}}
}

\put(10, 25){\uwave{\hspace{5mm}}}
\put(25, 50){\uwave{\hspace{5mm}}}


\put(25, 70){\line(0, -1){70}}
\put(40, 70){\line(0, -1){70}}

\put(50, 30){$=\ 0$}

\put(80, 10){;}
\end{picture}

\vspace{12mm}

\begin{picture}(50, 70)(0,0)
\put(-30, 40){$ii)$}

\put(8, -8){\tiny$a$}
\put(8, 73){\tiny$a$}
\put(23, -8){\tiny$b$}
\put(23, 73){\tiny$b$}
\put(38, -8){\tiny$c$}
\put(38, 73){\tiny$c$}
{
\put(10, 0){\line(0, 1){70}}
}

\put(10, 50){\uwave{\hspace{5mm}}}
\put(25, 15){\uwave{\hspace{5mm}}}

\qbezier(25, 55)(25, 55)(40, 70)
\qbezier(25, 70)(25, 70)(40, 55)

\put(25, 55){\line(0, -1){15}}
\put(40, 55){\line(0, -1){15}}

\put(25, 30){\line(0, -1){30}}
\put(40, 30){\line(0, -1){30}}
\qbezier(25, 30)(25, 30)(40, 40)
\qbezier(25, 40)(25, 40)(40, 30)
\put(45, 30){$-$}
\end{picture}
\begin{picture}(50, 70)(0,0)
\put(8, -8){\tiny$a$}
\put(8, 73){\tiny$a$}
\put(23, -8){\tiny$b$}
\put(23, 73){\tiny$b$}
\put(38, -8){\tiny$c$}
\put(38, 73){\tiny$c$}
{
\put(10, 0){\line(0, 1){70}}
}

\put(25, 60){\uwave{\hspace{5mm}}}
\put(10, 25){\uwave{\hspace{5mm}}}

\put(25, 70){\line(0, -1){30}}
\put(40, 70){\line(0, -1){30}}

\qbezier(25, 40)(25, 40)(40, 30)
\qbezier(25, 30)(25, 30)(40, 40)

\put(25, 30){\line(0, -1){15}}
\put(40, 30){\line(0, -1){15}}
\qbezier(25, 15)(25, 15)(40, 0)
\qbezier(25, 0)(25, 0)(40, 15)
\put(45, 30){$+$}
\end{picture}
%
\begin{picture}(50, 70)(0,0)
\put(8, -8){\tiny$a$}
\put(8, 73){\tiny$a$}
\put(23, -8){\tiny$b$}
\put(23, 73){\tiny$b$}
\put(38, -8){\tiny$c$}
\put(38, 73){\tiny$c$}
{
\put(10, 0){\line(0, 1){70}}
}

\put(10, 50){\uwave{\hspace{5mm}}}
\put(25, 25){\uwave{\hspace{5mm}}}


\put(25, 70){\line(0, -1){70}}
\put(40, 70){\line(0, -1){70}}

\put(45, 30){$-$}
\end{picture}
\begin{picture}(50, 70)(0,0)
\put(8, -8){\tiny$a$}
\put(8, 73){\tiny$a$}
\put(23, -8){\tiny$b$}
\put(23, 73){\tiny$b$}
\put(38, -8){\tiny$c$}
\put(38, 73){\tiny$c$}
{
\put(10, 0){\line(0, 1){70}}
}

\put(10, 25){\uwave{\hspace{5mm}}}
\put(25, 50){\uwave{\hspace{5mm}}}


\put(25, 70){\line(0, -1){70}}
\put(40, 70){\line(0, -1){70}}

\put(50, 30){$=\ 0$}

\put(80, 10){.}
\end{picture}

\caption{Four-term relations}
\label{fig:C-4-term}
\end{figure}
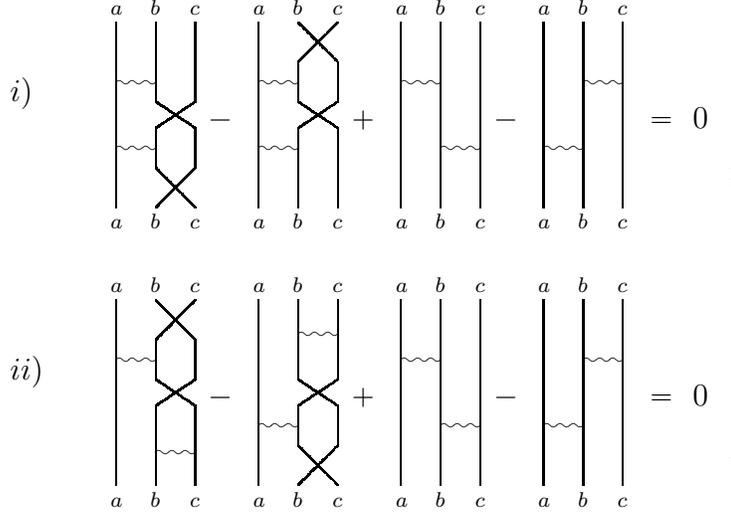  

\subsubsection{Conceptual interpretation of the relations} We next re-interpret the relations above in a more conceptual way.
We shall use the following {\em notation}. Let $P_{12}(a,b,c)=X_{ab}\ot I^c$, $P_{23}(a,b,c)=I^a\ot X_{bc}$. We shall drop the argument $(a,b,c)$
when it is clear from the context, and write merely $P_{12}$ and $P_{23}$. Thus $P_{12}\inv=P_{12}$ should be interpreted as $P_{12}(a,b,c)\inv= P_{12}(b,a,c)$, etc.
The diagrams $P_{12}$ and $P_{23}$ generate, under composition,
a group $P$ isomorphic to $\Sym_3$, of which one element is $P_{13}=P_{13}(a,b,c)=P_{12}\circ P_{23}\circ P_{12}=P_{23}\circ P_{12}\circ P_{23}$.

We also introduce the diagrams $\BH^{(2)}_\ell$ and $\BH^{(2)}_r$ as depicted in Fig. \ref{fig:delt}. 
\begin{figure}[h]
\begin{picture}(100, 60)(-30,-10)

\put(-35, 15){$\BH^{(2)}_\ell\;=$}

\put(2, -18){\tiny$a$}
\put(2, 52){\tiny$a$}
\put(5, -10){\line(0, 1){60}}

\put(5, 35){\uwave{\hspace{7mm}}}
\put(25,15){\uwave{\hspace{7mm}}}

\put(23, -18){\tiny$b$}
\put(23, 52){\tiny$b$}
\put(25, -10){\line(0, 1){60}}
\put(45, -10){\line(0, 1){60}}
\put(43, -18){\tiny$c$}
\put(43, 52){\tiny$c$}
\end{picture}
\begin{picture}(100, 60)(30,-10)
\put(55, 15){$\BH^{(2)}_r\;=$}

\put(92, -18){\tiny$a$}
\put(92, 52){\tiny$a$}
\put(95, -10){\line(0, 1){60}}

\put(95, 15){\uwave{\hspace{7mm}}}
\put(115,35){\uwave{\hspace{7mm}}}

\put( 113, -18){\tiny$b$}
\put(113, 52){\tiny$b$}
\put(115, -10){\line(0, 1){60}}
\put(135, -10){\line(0, 1){60}}
\put(133, -18){\tiny$c$}
\put(133, 52){\tiny$c$}

\end{picture} 

\caption{The diagrams $\BH^{(2)}_\ell,  \BH^{(2)}_r$}
\label{fig:delt}
\end{figure}

\begin{lemma}\label{lem:deltalr}
Given the relations in Figures \ref{fig:conj} and \ref{fig:BH-S}, we have 
\be\label{eq:dlr}
\BH^{(2)}_\ell=P_{13}\BH^{(2)}_rP_{13}\;(\iff \BH^{(2)}_r=P_{13}\BH^{(2)}_\ell P_{13}).
\ee
\end{lemma}

\begin{proof} We refer to the chain of diagrammatic equalities in Figs. \ref{fig:pfdelt1} and \ref{fig:pfdelt2}. The rightmost equality in Fig. \ref{fig:pfdelt1}
can be seen by applying \eqref{eq:conj} to the bottom of the left side and then applying the equality in Fig. \ref{fig:BH-S} to the top left of the left side.
Each of the other equalities follows by applying one of the relations $P_{12}^2=I\ot I\ot I$, \eqref{eq:conj} or the one depicted in Fig. \ref{fig:BH-S}.

The result of this calculation is a proof that $P_{12}\BH^{(2)}_\ell P_{12}=P_{23}P_{12}\BH^{(2)}_rP_{12}P_{23}$, which implies the stated relation by pre- and post-multiplying 
both sides by $P_{12}$.
\end{proof}

\begin{figure}[h]
\begin{picture}(150, 60)(0,-10)

\put(-75, 15){$P_{12}\BH^{(2)}_\ell P_{12}\;=$}

\put(5, -7){\line(0, 1){57}}

\put(5,50){\line(2,1){20}}
\put(25,50){\line(-2,1){20}}

\put(5,-7){\line(2,-1){20}}
\put(25,-7){\line(-2,-1){20}}

\put(5, 40){\uwave{\hspace{7mm}}}
\put(25,10){\uwave{\hspace{7mm}}}

\put(25, -7){\line(0, 1){57}}
\put(45, -7){\line(0, 1){57}}

\put(65, 15){$=$}

\put(95, 30){\line(0, 1){20}}
\put(95, -7){\line(0, 1){27}}

\put(95, 40){\uwave{\hspace{7mm}}}
\put(115,10){\uwave{\hspace{7mm}}}

\put(115, -7){\line(0, 1){27}}
\put(115, 30){\line(0, 1){20}}

\put(135, -7){\line(0, 1){57}}

\put(95,20){\line(4,1){20}}
\put(95,25){\line(4,-1){20}}

\put(95,25){\line(4,1){20}}
\put(95,30){\line(4,-1){20}}

\put(95,50){\line(2,1){20}}
\put(115,50){\line(-2,1){20}}

\put(95,-7){\line(2,-1){20}}
\put(115,-7){\line(-2,-1){20}}

\put(140, 15){$\overset{}{=}$}


\put(165, -7){\line(0, 1){57}}

\put(165, 12){\uwave{\hspace{7mm}}}
\put(165,45){\uwave{\hspace{7mm}}}

\put(185, 3){\line(0, 1){20}}
\put(185, 33){\line(0, 1){17}}

\put(205, 3){\line(0, 1){20}}
\put(205, 33){\line(0, 1){17}}


\put(185,-7){\line(2,1){20}}
\put(185,3){\line(2,-1){20}}

\put(185,23){\line(2,1){20}}
\put(185,33){\line(2,-1){20}}

\put(95,50){\line(2,1){20}}
\put(115,50){\line(-2,1){20}}

\put(95,-7){\line(2,-1){20}}
\put(115,-7){\line(-2,-1){20}}

\end{picture} 

\caption{Proof of Lemma \ref{lem:deltalr}}
\label{fig:pfdelt1}
\end{figure}

\vskip .5truecm

\begin{figure}[h]
\begin{picture}(260, 80)(25,0)

\put(-15, 25){$=$}

\put(5, -7){\line(0, 1){77}}

\put(5, 12){\uwave{\hspace{7mm}}}
\put(5,45){\uwave{\hspace{7mm}}}

\put(25, 3){\line(0, 1){20}}
\put(25, 33){\line(0, 1){17}}

\put(45, 3){\line(0, 1){20}}
\put(45, 33){\line(0, 1){17}}


\put(25,-7){\line(2,1){20}}
\put(25,3){\line(2,-1){20}}

\put(25,60){\line(2,1){20}}
\put(25,70){\line(2,-1){20}}

\put(25,50){\line(2,1){20}}
\put(25,60){\line(2,-1){20}}


\put(25,23){\line(2,1){20}}
\put(25,33){\line(2,-1){20}}



\put(65, 25){$=$}

\put(95, -7){\line(0, 1){30}}
\put(95, 33){\line(0, 1){17}}
\put(95, 60){\line(0, 1){10}}

\put(95, 12){\uwave{\hspace{7mm}}}
\put(115,45){\uwave{\hspace{7mm}}}

\put(115, 3){\line(0, 1){20}}
\put(115, 33){\line(0, 1){17}}

\put(135, 3){\line(0, 1){57}}
\put(135, 33){\line(0, 1){17}}


\put(115,-7){\line(2,1){20}}
\put(115,3){\line(2,-1){20}}

\put(115,60){\line(2,1){20}}
\put(115,70){\line(2,-1){20}}

\put(95,50){\line(2,1){20}}
\put(95,60){\line(2,-1){20}}


\put(95,23){\line(2,1){20}}
\put(95,33){\line(2,-1){20}}


\put(150, 25){$\overset{}{=}$}


\put(175, -7){\line(0, 1){10}}
\put(175, 13){\line(0, 1){37}}
\put(175, 60){\line(0, 1){10}}

\put(175, 25){\uwave{\hspace{7mm}}}
\put(195,45){\uwave{\hspace{7mm}}}

\put(195, 13){\line(0, 1){20}}
\put(195, 33){\line(0, 1){17}}

\put(215, 3){\line(0, 1){57}}
\put(215, 33){\line(0, 1){17}}


\put(195,-7){\line(2,1){20}}
\put(195,3){\line(2,-1){20}}

\put(195,60){\line(2,1){20}}
\put(195,70){\line(2,-1){20}}

\put(175,50){\line(2,1){20}}
\put(175,60){\line(2,-1){20}}


\put(175,3){\line(2,1){20}}
\put(175,13){\line(2,-1){20}}

\put(225, 25){$\overset{}{=}P_{23}P_{12}\BH^{(2)}_rP_{12}P_{23}.$}

\end{picture} 

\caption{Proof of Lemma \ref{lem:deltalr} (cont.)}
\label{fig:pfdelt2}
\end{figure}
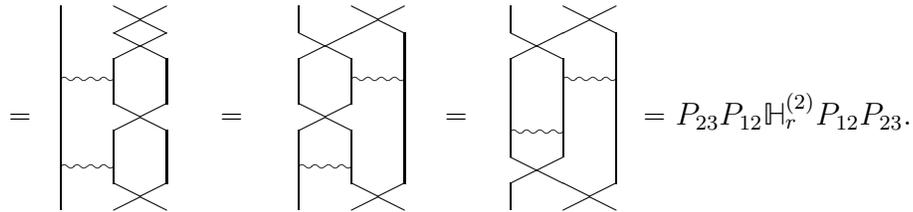

\begin{corollary}\label{cor:i-ii}
The 4-term relations i) and ii) are equivalent.
\end{corollary}

\begin{proof}
Applying \eqref{eq:conj} and then Fig. \ref{fig:BH-S} to the leftmost term in $i)$, we see that it is equal to $P_{12}\BH^{(2)}_\ell P_{12}$. Reflecting this relation 
in a horizontal, we see that the second term in $i)$ is equal to $P_{12}\BH^{(2)}_r P_{12}$. Thus $i)$ is equivalent to 
\beq\label{eq:Delt}
P_{12}\Delta P_{12}=-\Delta, \text{ where }\Delta=\BH^{(2)}_\ell-\BH^{(2)}_r.
\eeq
Similarly, $ii)$ is equivalent to 
\[
P_{23}\Delta P_{23}=-\Delta.
\]

Now observe that 
$$P_{12}\Delta P_{12}=P_{12}(P_{13}\BH^{(2)}_r P_{13}-\BH^{(2)}_r)P_{12}=P_{23}P_{12}\BH^{(2)}_r P_{12}P_{23}-P_{12}\BH^{(2)}_r P_{12}.$$
If we pre- and post-multiply this relation by $P_{23}$, we see that $P_{23}P_{12}\Delta P_{12}P_{23}$ $ =$ $ -P_{12}\Delta P_{12}$. Taking into account the above characterisation
of the relations $i)$ and $ii)$, this shows that $i)$ implies $ii)$. Similarly, $ii)$ implies $i)$, whence the result.
\end{proof}

The above proof yields immediately
\begin{scholium}
The 4-term relations are equivalent to the following statement.
Let $P\cong\Sym_3$ be the group generated by $P_{12}, P_{23}$. Then for $g\in P$, we have $g\Delta g\inv=\ve(g)\Delta,$ where $\ve(g)$ is the sign of the permutation $g$.
\end{scholium}

We also have the following result. 
\begin{lemma} 
Given the symmetry of $\BH(a, b)$ shown in Fig. \ref{fig:BH-S}, 
the left and right skew symmetries of $\BH$ are  equivalent. 
\end{lemma} 
\begin{proof}



Consider $X_{v a}\BH^{lt}(v, a) X_{a v}$. By sliding arcs over the coupon, this expression may be transformed
 into the form $(I^a\ot\cap_v\ot I^v) (X_{v a}\BH(v, a)X_{a v}\ot X_{v v})(I^a\ot\cup_v\ot I^v)$, which is equal to $\BH^{rt}(a, v)$ by the symmetry of $\BH$. 
 Hence $X_{v a}\BH^{lt}(v, a) X_{a v}= \BH^{rt}(a, v)$, which implies the equivalence of the left and right skew symmetries.
\end{proof}

\begin{remark}
The space $\Hom_{\HB_{\CG_1}(\delta_\CV)}(\bs, \bt)$ of morphisms for arbitrary objects $\bs, \bt\in \HB_{\CG_1}(\delta_\CV)$ 
is $\Z_+$-graded, with the homogeneous subspace $\Hom_{\HB_{\CG_1}(\delta_\CV)}(\bs, \bt)_k$ of degree $k$ spanned by diagrams with $k$ connectors. 
\end{remark}

Write $\End_{\HB_{\CG_1}(\delta_\CV)}(\bs) = \Hom_{\HB_{\CG_1}(\delta_\CV)}(\bs, \bs)$ 
for any object $\bs$. 
This is a $\Z_+$-graded associative algebra. 
Assume that $\bs=(c_1, c_2, \dots, c_r)$. Let $\BH_{i j}(\bs): \bs \to \bs $ be the morphisms given by Figure \ref{fig:Hijr} for $i, j=1, 2, \dots, r$ with $i<j$.  
\begin{figure}[h]
\begin{picture}(180, 70)(-25,-5)
\put(-25, 28){$\BH _{i j}(\bs)\  =$}
{
}

\put(38, -5){\tiny$c_1$}
\put(58, -5){\tiny$c_i$}
\put(160, -5){\tiny$c_r$}
\put(126, -5){\tiny$c_j$}

\put(40, 0){\line(0, 1){60}}
\put(45, 30){...}

\put(60, 0){\line(0, 1){60}}

\put(60, 32){\uwave{\hspace{7mm}}}
\put(80, 22){\line(0, 1){15}}

\qbezier(80, 22)(80, 20)(125, 0)
\qbezier(80, 37)(80, 39)(125, 60)

\put(90, 0){\line(0, 1){60}}
\put(110, 0){\line(0, 1){60}}
\put(95, 30){...}

\put(140, 0){\line(0, 1){60}}
\put(145, 30){...}
\put(160, 0){\line(0, 1){60}}

\put(170, 0){,}
\end{picture} 
\caption{The diagrams $\BH _{i j}(\bs)$}
\label{fig:Hijr}
\end{figure}

Define the elements $\Theta_j(\bs)\in \End_{\HB_{\CG_1}(\delta_\CV)}(\bs)$  by
\[
\Theta_j(\bs)=\sum_{i=1}^{j-1}\BH_{i j}(\bs), \quad 1< j \le r.
\]
Then the following results hold. 
\begin{lemma}\label{lem:Murphy}
Fix any object $\bs$ of length $r$.  
\begin{enumerate}
\item The following relations hold for all distinct $i, j, k=1, 2, \dots, r$.
\[
\baln
&[\BH_{ij}(\bs), \BH_{k\ell}(\bs)]=0,\\
&[\BH_{ik}(\bs)+\BH_{i\ell}(\bs),\BH_{k\ell}(\bs)]=0,\\
&[\BH_{ij}(\bs), \BH_{ik}(\bs)+\BH_{jk}(\bs)]=0.
\ealn
\]

\item The elements $\Theta_i(\bs)$ satisfy the following relations, 
\[
\baln
&[\Theta_i(\bs), \Theta_j(\bs)]=0, \quad \forall i, j, \\
&X_j(\bs_j^\vee)\Theta_i(\bs_j^\vee) X_j(\bs) = 0, \quad j\not\in\{i, i+1\},\\ 
&X_i(\bs_i^\vee)\Theta_i(\bs_i^\vee) X_i(\bs) = \Theta_{i+1}(\bs)  - \BH_{i, i+1}(\bs), \quad i<r. 
\ealn
\]
\end{enumerate}
\end{lemma}
\begin{proof}
To prove part (1), we note that the first relation follows from 
the definition of $\BH_{i j}(\bs)$. The second and third relations follow from the four term relations by sliding arcs over coupons (which are infinitesimal braids in this case). 
Part (2) follows from part (1) using the definition of $\Theta_i(\bs)$. 
\end{proof}
\subsubsection{Further properties of $\BH(a, b)$}\label{sect:BH-properties}
Let us make some further comments on the  properties of $\BH(a, b)$.  
More details can be found in \cite[\S XX]{K}.

We may turn $\BH$ into a ``functorial morphism'' by defining $\BH(\bs, \bt): (\bs, \bt)\lra (\bs, \bt)$ recursively as follows for any objects 
$\bs, \bt$ of $\HB_{\CG_1}(\delta_\CV)$.  If $\bs=\bs'\ot \bs''$, and $\bt=\bt'\ot\bt''$, then
\[
\baln
\BH(\bs'\ot \bs'', \bt)= (I(\bs')\ot X_{\bt, \bs''}) \left(\BH(\bs', \bt)\ot I(\bs'')\right)  (I(\bs')\ot X_{\bs'', \bt}) + I(\bs')\ot \BH(\bs'', \bt), \\
\BH(\bs, \bt'\ot \bt'')= (X_{\bt', \bs}\ot I(\bt'')) \left(I(\bt')\ot \BH(\bs, \bt'')\right)  ( X_{\bs, \bt'}\ot I(\bt'')) + \BH(\bs, \bt')\ot I(\bt'').
\ealn
\]
Then the four-term relations $3)$ and $4)$ generalise to 
\beq
&[\BH(\bs_1, \bs_2)\ot I(\bs_3), \BH(\bs_1\ot \bs_2, \bs_3)]=0, \\
&[\BH(\bs_1, \bs_2\ot \bs_3), I(\bs_1)\ot \BH(\bs_2, \bs_3)]=0, 
\eeq
for any objects $\bs_1, \bs_2$ and $\bs_3$.

The subalgebra $\BT(\bs)$ of $\End_{\HB_{\CG_1}(\delta_\CV)}(\bs)$, generated by the elements $\BH_{ij}(\bs)$ for all $i<j$, has a Hopf algebra structure. 

The co-multiplication $\Delta: \BT(\bs)\lra \BT(\bs)\ot \BT(\bs)$ is defined in the following way. 
Given a single diagram $\BD$ in $\BT(\bs)$ (i.e., $\BD$ is a product of elements $\BH_{i j}(\bs)$), let $Conn(\BD)$ denote the set of connectors in $\BD$. 
If $Conn(\BD)'$ is a subset of $Conn(\BD)$, let $Conn(\BD)''$ be its complement. Denote by $\BD''$ (resp. $\BD'$) the diagram obtained by deleting from $\BD$ the connectors
 in $Conn(\BD)'$ (resp. $\BD''$). Then 
\beq
\Delta(\BD)=\sum_{Conn(\BD)'\subseteq Conn(\BD)} \BD'\ot \BD''. 
\eeq
Here $Conn(\BD)'$ may be empty or $Conn(\BD)$ itself; in particular, if $Conn(\BD)=\emptyset$, we have the unique subsets $Conn(\BD)'=Conn(\BD)''=\emptyset$. In the case $Conn(\BD)=\emptyset$, we have $\BD=I(\bs)$, and $\Delta(I(\bs))=I(\bs)\ot I(\bs)$. 
Note in particular that $\Delta(\BH_{i j}(\bs))=\BH_{i j}(\bs)\ot I(\bs) + I(\bs)\ot \BH_{i j}(\bs)$ for all $i<j$. 

The co-unit $\varepsilon: \BT(\bs)\lra \K$ is the usual augmentation map with $\varepsilon(I(\bs))=1$ and $\varepsilon(\BH_{i j}(\bs))=0$ for all $i<j$. The antipode $S: \BT(\bs)\lra \BT(\bs)$ is an algebra anti-isomorphism such that $S(\BH_{i j}(\bs))=-\BH_{i j}(\bs)$ for all $i<j$. 

\begin{remark}
This Hopf algebra structure may be understood by considering the 
Lie algebra generated by the elements $\BH_{i j}(\bs)$, for $i<j$,  
subject to the relations in Lemma \ref{lem:Murphy}(1). Then 
$\BT(\bs)$ is its universal enveloping algebra,  which has a canonical Hopf algebra structure which is the one described above. 
\end{remark}

\subsection{Representations of the monoidal category $\HB_{\CG_1}(\delta_\CV)$} \label{sect:HB-rep}

 In this section, we construct monoidal functors from $\HB_{\CG_1}(\delta_\CV)$ to certain categories of Lie superalgebra representations. 

We work in the category of $\Z_2$-graded vector spaces over $\C$. Any object $M$ is the direct sum $M=M_{\bar 0}\oplus M_{\bar 1}$ of its even and odd subspaces. 
For any homogeneous element $w\in M_{\bar i}$, denote by $[w]=i$ its degree. 
If $M$ is finite dimensional, write
$\sdim(M): =\dim(M_{\bar 0}) - \dim(M_{\bar 1})$ for the super dimension of $M$. 

The canonical symmetry of the category is the functorial morphism
\beq\label{eq:tau-funct}
&\tau_{M M'}:  M \ot M'\lra M' \ot M
\eeq
for any objects $M, M'$ such that 
$
\tau_{M M'}(w\ot w')=(-1)^{[w][w']}w'\ot w$ for all homogeneous $w, \in M, w'\in M'. 
$
It has the properties
\beq
&& \tau_{M\ot M', M''} = (\tau_{M M''}\ot \id_{M'}) (\id_{M}\ot \tau_{M' M''}),  
	\label{eq:tau1}\\
&&  \tau_{M, M'\ot M''} = (\id_{M'}\ot \tau_{M M''}) (\tau_{M M'} \ot \id_{M''}), 
	\label{eq:tau2}\\
&&(\tau_{M' M''} \ot \id_{M})  (\id_{M'}\ot \tau_{M M''}) (\tau_{M M'} \ot \id_{M''})				\label{eq:tau3}\\  
&& = (\id_{M''}\ot \tau_{M M'})  (\tau_{M M''}\ot \id_{M'}) (\id_{M}\ot \tau_{M' M''}) \nonumber
\eeq
for all $\Z_2$-graded vector spaces  $M, M', M''$.

Let $\fg=\fg_{\bar0}+\fg_{\bar1}$ be a finite dimensional Lie superalgebra over $\C$, which reduces to an ordinary Lie algebra if $\fg_{\bar 1}=0$.  We assume that $\fg$ has a non-degenerate supersymmetric even bilinear form $\kappa:\fg\times \fg\lr\C$ which is $\ad(\fg)$-invariant. Thus we have Casimir algebra $C_r(\fg)$.

Let $L=L_{\bar 0}+ L_{\bar 1}$ be a finite dimensional simple module  
for a Lie superalgebra $\fg$. It is self dual if there exists an even $\fg$-module isomorphism 
$\iota_L: L \lra L^*$, where $L^*$ is the dual space of $L$ with the $\fg$-action defined, 
for any $X\in\fg$ and $\bar{f}\in L^*$,  by 
\[
X.\bar{f}(v) = (-1)^{[x][f]} \bar{f}(X. v), \quad \forall v\in L. 
\]
This isomorphism is unique up to scalar multiples by Schurs lemma. 
It gives rise to two related $\fg$-module (even) homomorphisms in a canonical way. 

One of the homomorphisms, denoted by $\omega_L:  L\ot L\lra \C$, is defined by 
\[
\omega_L(v, v') = \iota_L(v)(v'), \quad \forall v, v'\in L.
\]
This map is either supersymmetric or skew supersymmetric, that is, satisfies either $\omega_L(v, v')= (-1)^{[v][v']}\omega_L(v', v)$ or $\omega_L(v, v')= - (-1)^{[v][v']}\omega_L(v', v)$, for all $v, v'\in L$. By reversing the parity of $L$, we can change this map from being skew supersymmetric to supersymmetric and vice versa. Thus we shall assume that $\omega_L$ is supersymmetric. The other homomorphism, denoted by  $\check{\omega}_L: \C\lra L\ot L$, is defined by 
\[
(\id_V\ot \iota_L)\check{\omega}_L(1) (v)= v, \quad \forall v\in L.
\]
\begin{lemma}\label{lem:duality}
Retain the notation above. For any $\fg$-module $M$ and self dual simple 
$\fg$-module $L$, the canonical symmetry and the $\fg$-module homomorphisms $\omega_L$ and $\check{\omega}_L$ defined above satisfy the following relations
\beq
& \tau_{L L} \check{\omega}_L = \check{\omega}_L, \quad  \omega_L\tau_{L L}=\omega_L, \label{eq:form0}\\
&(\id_L\ot \omega_L) (\check{\omega}_L\ot \id_L)  = \id_L = (\omega_L\ot \id_L) (\id_L\ot \check{\omega}_L), \label{eq:form1} \\
&\omega_L\check{\omega}_L(1)= \sdim(L); \label{eq:form2}\\
&({\omega}_L\ot\id_M) (\id_L\ot \tau_{M L})= (\id_M\ot{\omega}_L) (\tau_{L M}\ot\id_L), \label{eq:form3}\\
&(\tau_{M L}\ot \id_L)(\id_M\ot \check{\omega}_L)=(\id_L\ot \tau_{L M})(\check{\omega}_L\ot \id_M). \label{eq:form4}
\eeq
\end{lemma}
\begin{proof} 
We have chosen $\omega_L$ (and hence $\check{\omega}_L$) to be superymmetric, thus \eqref{eq:form0} holds. 

To prove the other relations we use a basis. Choose a homogeneous basis $\{e^a\mid a=1, 2, \dots, \dim(V)\}$, 
and a dual basis $\{e_a\mid a=1, 2, \dots, \dim(V)\}$ with respect to $\omega_L$, that is,  $\omega_L(e_a, e^b)=\delta_{a b}$ for all $a, b$. Then $\check{\omega}_L(1) = \sum_a e^a \ot e_a$. 
Now the relation \eqref{eq:form2} follows from the following calculation:
\[
\omega_L \check{\omega}_L(1)= \sum_a \omega_L(e^a,  e_a) = \sum_a (-1)^{[e_a]} \omega_L( e_a, e^a) = \sum_a (-1)^{[e_a]} =\sdim(L).
\] 
The first equality in \eqref{eq:form1} is the definition of $\check{\omega}_L$ written differently, 
and the relation \eqref{eq:form3} is also clear.  Using the expression of $\check{\omega}_L$ in the bases, the second equality in \eqref{eq:form1} follows easily;
 \eqref{eq:form4} also follows, using some simple linear algebra. 
\end{proof}

Let us fix a finite dimensional simple Lie superalgebra $\fg$, and choose a 
finite set $\CM$ of arbitrary $\U(\fg)$-modules.  Let $\CV\subset \CM$ be a subset consisting of  
finite dimensional self dual simple $\U(\fg)$-modules, where each $L\in \CV$ is equipped with the supersymmetric module homomorphisms $\omega_L$ and $\check{\omega}_L$ discussed above. 
We denote by $\CT(\CM)$ the full subcategory of $\U(\fg)$-modules consisting of repeated tensor products of modules in $\CM$. This is a monoidal category. 

Throughout  we adopt the following notation:  
for any $\U(\fg)$-module $M$, we denote by $\mu_W: \U(\fg)\lra \End_\C(W)$ the corresponding representation of $\U(\fg)$. 

The tempered Casimir operator acts on the tensor product $M \ot M'$ of two $\U(\fg)$-modules by
\beq\label{eq:t-2}
(\mu_M\ot  \mu_{M'})(t): M \ot M'\lra M' \ot M.
\eeq
This is a $\U(\fg)$-module endomorphism by \eqref{eq:tcomm}.

Given $\fg$-modules $M_i$ with $1\le i\le r$, write 
$\mathbf{M}= M_1\ot M_2\ot \dots\ot M_r$.  We have the representation $
\mu_{M_1}\ot \mu_{M_2}\ot\dots\ot \mu_{M_r}: \U(\fg)^{\ot r}\lra \End_\C(\bM)
$
of $\U(\fg)$. 
It follows \eqref{eq:tcomm} that this map sends $C_r(\fg)$ to $\End_\fg(\bM)$, thus its  restriction yields the following representation of the Casimir algebra, 
\beq\label{eq:mu}
\mu_{\bM}:  C_r(\fg)\lra \End_\fg(\bM). 
\eeq

Now we consider the category $\HB_{\CM}(\sdim_\CV)$ with $\sdim_\CV:=\{\sdim(L)\mid L\in\CV\}$.  

\begin{theorem} \label{thm:main}
Retain the above notation. Let $\CM$ be a finite set of $\fg$-modules, and let $\CV$ be a subset of $\CM$ consisting of finite dimensional self dual simple $\fg$-modules.  Then there exists a unique covariant monoidal functor 
\beq
\SF: \HB_{\CG_1}(\sdim_\CV)\lra \CT(\CM), 
\eeq
which maps any sequence $(M_1, \dots, M_r)$ with $M_i\in\CM$ to the $\fg$-module $M_1\ot \dots\ot M_r$, and maps the generators of morphisms in $\HB_{\CG_1}(\sdim_\CV)$ to the following $\fg$-module homomorphisms.
\beq
&I^M\mapsto \id_M, \quad  X_{M M'}\mapsto \tau_{M M'},  \label{eq:funct1}\\
&\BH(M, M') \mapsto (\mu_M\ot  \mu_{M'})(t), 
	\quad \forall M, M'\in\CM,  \label{eq:funct2}\\
& \cap_V \mapsto \omega_V, \quad \cup^V\mapsto \check{\omega}_V, 
	\quad \forall V\in\CV,  \label{eq:funct3} \\
& \SF(g_{\bs \bt})\in \Hom_{\CT(\CM)}(\CF(\bs), \CF(\bt)),  
\  \text{$\forall g_{\bs \bt}\in\CG_1$,  such that}   \label{eq:funct4}
\eeq
the $\fg$-module homomorphism $\SF(g_{\bs^\vee \bt^\vee})$ is equal to 
\[(\SF(\Cap(\bt^\vee))\ot\id_{\SF(\bs^\vee)}) (\id_{\SF(\bt^\vee)} \ot \SF(g_{\bs \bt})\ot  \id_{\SF(\bs^\vee)})(\id_{\SF(\bt^\vee)} \ot \SF(\Cup(\bs)).
\]

\end{theorem}
\begin{proof}
The proof amounts to checking that the defining relations among 
the generators of morphisms are preserved by $\SF$.  

There are three types of relations: 
(a) the usual Brauer relations, (b) two relations satisfied by all coupons depicted by Figure \ref{fig:gvee} and Figure \ref{fig:coupon}, and (c) the relations satisfied by infinitesimal braids as coupons. 

It is immediate from Lemma \ref{lem:duality} that all type (a) relations are preserved by $\SF$. 

By definition, $\SF$ preserves the relation given by Figure \ref{fig:gvee} for all elements of $\CG_1$ and $\BH(a, b)$  for all $a, b\in\CM$.  
Note that \eqref{eq:tau1} and \eqref{eq:tau2} imply that 
$\SF(X(\ba, \bs)=\tau_{\SF(\ba) \SF(\bs)}$ for all objects $\ba, \bs$. 
Using this,  one can easily see that $\SF$ preserves the relation in Figure \ref{fig:coupon},  since 
for any $\Z_2$-graded vector spaces $W$, $W_1$ and $W_2$ and linear map $\Psi: W_1\lra W_2$, 
\[
\tau_{W, W_2}(\id_W\ot \Psi) =(\Psi\ot \id_W) \tau_{W, W_1}, \quad 
\tau_{W_2, W}(\Psi\ot \id_W) =(\id_W\ot\Psi) \tau_{W_1, W}.
\]
This shows that $\SF$ preserves the type (b) relations. 

Recall that \eqref{eq:mu} defines a representation of the Casimir algebra. 
Thus it follows from \eqref{eq:C-4-term} that the four-terms relations among infinitesimal braids are preserved by $\SF$.  

Finally, we consider the left and right skew symmetries of $\BH$. 
For any finite dimensional self dual $\fg$-module $V$, we can define the following linear map 
\beq\label{eq:vee}
&& ^\vee: \End_\C(V)\lra \End_\C(V), \\
&& \phi\mapsto \phi^\vee = ({\omega}_V \ot \id_V)(\phi\ot \tau_{V V})(\check{\omega}_V \ot \id_V). \nonumber
\eeq
We have $^\vee\phi:= (\id_V\ot {\omega}_V)(\tau\ot \phi)(\id_V\ot \check{\omega}_V)=\phi^\vee$ since ${\omega}_V$ is supersymmetric. 
Now look at $\phi$ given by $\mu_V(Y)$ with $Y\in\fg$. Define the linear map $\mu^{\vee}_V: \fg\lra \End_\C(V)$ by $\mu^{\vee}_V(Y)= \mu_V(Y)^{\vee}$. 
Using notation introduced in the proof of  Lemma \ref{lem:duality}, we have for all $v\in V$, 
\[ 
\begin{aligned}
\mu^\vee_V(Y) (v)
&= -  \sum_a  (-1)^{[e_a]([v] + [Y])} \omega_V( e^a, \mu(Y) v) e_a
				&&\text{($\fg$ invariance of $\omega_V$)}\\
&= -  \sum_a  (-1)^{[e_a]}\omega_V( e^a, \mu(Y) v) e_a  							
				&&  \text{(evenness of $\omega_V$)} \\
&= -  \sum_a  \omega_V(\mu(Y) v,  e^a) e_a  				
				&&  \text{(supersymmetry of $\omega_V$)} \\
&= -\mu_V(Y)v.  
\end{aligned}
 \]
Hence $\mu^\vee_V(Y) = -\mu_V(Y)$ for all $Y\in \fg$. Therefore,  for all $M\in\CM$,  we have 
\[
\baln
\SF(\BH^{lt}(V, M)) = (\mu^\vee_V\ot \mu_M)(t) = -\SF(\BH(V, M)), \\
\SF(\BH^{rt}(M, V)) = (\mu_M\ot \mu^\vee_V)(t) = -\SF(\BH(M,V)). 
\ealn
\]
This proves that the type (c) relations are also preserved by $\SF$, completing the proof of the theorem.
\end{proof}

\begin{remark}
There exist results analogous to Theorem \ref{thm:main} between the category of coloured ribbon graphs and the category of representations of any quantum group \cite{RT, T} or quantum supergroup \cite{Z94, Z95}.
\end{remark}

\section{Polar enhancements of the Brauer category}\label{sect:hHB}

We study quotient categories of $\HB(\delta_\CV)$ (i.e., $\HB_{\CG_1}(\delta_\CV)$ with $\CG_1=\emptyset$).  
Examples with non-trivial $\CG_1$ will be studied in Section \ref{sect:EB-G}. 
We further assume that $\CM\supset\CV=\{v\}$, and let $\delta=\delta_v$.  

In this case, we adopt the following {\bf convention}.
\begin{remark}\label{rmk:no-v}
Since $v$ is the only element of $\CV$, we can and will drop the label $v$ from all strings coloured by it. This should not cause any confusion. Similarly we shall also  
write $I^v$, $X_{v, v}$, $\cup^v$ and $\cap_v$ simply as  
$I$, $X$, $\cup$ and $\cap$. 
\end{remark}

\subsection{A quotient category of $\HB(\delta)$}

Consider the morphisms $H, E: v^2 \to v^2$ defined by $E= \cup\circ\cap$ and $H=X-E$. We  represent $H$ diagrammatically by Figure \ref{fig:t-image}.
\begin{figure}[h]
\begin{picture}(150, 40)(30,0)

\put(35, 0){\line(0, 1){40}}

\put(35, 20){\line(1, 0){20}}

\put(55, 0){\line(0, 1){40}}

\put(65, 15){$:=$}

\qbezier(90, 0)(90, 0)(115, 35)
\qbezier(90, 35)(90, 35)(115, 0)

\put(125, 15){$-$}

\qbezier(155, 35)(168, 10)(180, 35)
\qbezier(155, 0)(168, 25)(180, 0)
\end{picture} 
\caption{Morphism $H: v^2 \to v^2$}
\label{fig:t-image}
\end{figure}

\begin{lemma}\label{lem:H-skew-sym}
\begin{enumerate}
\item The morphism $H: v^2\to v^2$ satisfies $X H X = H$. 
\item $H$ has {\em left} and {\em right skew symmetries} 
as depicted in Figure \ref{fig:Brauer-H-skew}. 
\item We have
\[
(H-I\ot I)(H+I\ot I)(H-(1-\delta)I\ot I)=0,  
\]
so that $X$ and $E$ can be expressed as polynomials in $H$ unless $\delta=2$.
\end{enumerate}
\end{lemma}

\begin{figure}[h]
\setlength{\unitlength}{0.25mm}
\begin{picture}(125, 60)(0,0)

\qbezier(40, 35)(50, 55)(60, 40)
\qbezier(40, 25)(50, 5)(60, 20)

\qbezier(40, 35)(30, 5)(20, 0)
\qbezier(40, 25)(30, 55)(20, 60)

\put(60, 20){\line(0, 1){20}}

\put(60, 30){\line(1, 0){20}}
\put(80, 0){\line(0, 1){60}}

\put(100, 27){$=\  - $}
\end{picture}  
\begin{picture}(85, 60)(-15,0)
\put(0, 0){\line(0, 1){60}}
\put(0, 30){\line(1, 0){30}}

\put(30, 0){\line(0, 1){60}}
\put(47, 27){$=\  $}
\end{picture} 
\begin{picture}(45, 60)(-5,0)
\put(0, 0){\line(0, 1){60}}
\put(0, 30){\line(1, 0){20}}

\qbezier(20, 40)(30, 55)(40, 35)
\qbezier(20, 20)(30, 5)(40, 25)
\put(20, 20){\line(0, 1){20}}

\qbezier(40, 35)(50, 5)(60, 0)
\qbezier(40, 25)(50, 55)(60, 60)

\end{picture} 
\caption{left and right skew symmetries of $H$}
\label{fig:Brauer-H-skew}
\end{figure}

\begin{proof}
(1) is clear, because $X^3=X$ and $XE=E X=E$.

(2) The left side of Fig. \ref{fig:Brauer-H-skew} is equal to $(I\ot\cap\ot I)\circ(X\ot H)\circ(I\ot\cup\ot I)$, which by the relations in Theorem \ref{thm:tensor-cat}
is equal to $E-X=-H$. Right skew symmetry of $H$ can be proven similarly. 

(3) We have 
\beq\label{eq:H-squared}
H^2=  I\ot I -(2-\delta) E,  
\eeq
whence $ H^2 - I\ot I=(\delta- 2) E$. 

Now $HE=(1-\delta)E$; thus $(H-(1-\delta)I\ot I)E=0$, which together with the previous equation, implies the cubic relation in (3).

If $\delta\ne 2$,  $E=(\delta-2)\inv(H^2-I\ot I)$ and $X=H+E$ are clearly expressible as polynomials in $H$.  In case $\delta= 2$, we have $H^2= I\ot I$. It is not possible to express $E$ and $X$
as linear combinations of $H$ and $I\ot I$, so they are not polynomials in $H$. 
\end{proof}

Introduce the following elements  of  $\Hom_{\HB(\delta)}(v^3, v^3)$
\[
H_{12}= H\ot I, \quad H_{23}=I\ot H,  \quad 
H_{1 3} = (I\ot X) (H\ot I) (I\ot X). 
\]
We have the following result. 
\begin{lemma} \label{lem:H-4-t}
The following ``four term relations'' hold in $\HB(\delta)$, 
\[
[H_{12}, H_{1 3} + H_{23}]=0, \quad [H_{12}+ H_{1 3},  H_{23}]=0.
\]
\end{lemma}

\begin{proof} 
A direct calculation shows that 
\[
\begin{aligned}
H_{12}(H_{13}+H_{23})=(H_{13}+H_{23})H_{12}\\
\setlength{\unitlength}{0.3mm}
\begin{picture}(180, 50)(0,0)
\put(-20, 15){$=$}
\qbezier(0, 40)(0, 40)(30, 0)
\qbezier(20, 40)(20, 40)(0, 0)
\qbezier(30, 40)(30, 40)(10, 0)

\put(35, 15){$+$}

\qbezier(50, 40)(50, 40)(70, 0)
\qbezier(60, 40)(60, 40)(80, 0)
\qbezier(50, 0)(50, 0)(80, 40)

\put(85, 15){$-$}

\qbezier(100, 40)(115, 10)(130, 40)
\qbezier(110, 0)(118, 30)(125, 0)
\qbezier(100, 0)(100, 0)(115, 40)

\put(135, 15){$-$}

\qbezier(150, 40)(150, 40)(165, 0)
\qbezier(160, 40)(168, 10)(175, 40)
\qbezier(150, 0)(165, 30)(180, 0)

\put(185, 5){, }
\end{picture} 
\end{aligned}
\]
and  
\[
\begin{aligned}
(H_{12} + H_{13}) H_{23}=H_{23}(H_{12} + H_{13})\\
\setlength{\unitlength}{0.3mm}
\begin{picture}(180, 50)(0,0)
\put(-20, 15){$=$}
\qbezier(0, 40)(0, 40)(30, 0)
\qbezier(20, 40)(20, 40)(0, 0)
\qbezier(30, 40)(30, 40)(10, 0)

\put(35, 15){$+$}

\qbezier(50, 40)(50, 40)(70, 0)
\qbezier(60, 40)(60, 40)(80, 0)
\qbezier(50, 0)(50, 0)(80, 40)

\put(85, 15){$-$}

\qbezier(100, 40)(115, 10)(130, 40)
\qbezier(105, 0)(113, 30)(120, 0)
\qbezier(130, 0)(130, 0)(115, 40)

\put(135, 15){$-$}

\qbezier(180, 40)(180, 40)(165, 0)
\qbezier(155, 40)(162, 10)(170, 40)
\qbezier(150, 0)(165, 30)(180, 0)

\put(185, 5){. }
\end{picture} 
\end{aligned}
\]
This proves the stated four-term relations. 
\end{proof}

The fact that $H$ satisfies the defining relations of $\BH(v, v)$ enables us to introduce the following quotient category of $\HB(\delta)$.

\begin{definition}\label{def:mpolar}
Assume that $\CM\supset\CV=\{v\}$, and set $\delta=\delta_\CV$. 
Let $\CJ$ be the tensor ideal generated by $\BH(v, v) - H$ in $\HB(\delta)$. Define the quotient category 
\[
\wh\HB(\delta)=\HB(\delta)/\CJ.
\]
\end{definition}

\begin{remark}
We shall use the same diagrams and symbols of morphisms in $\HB(\delta)$ to denote their images in $\wh\HB(\delta)$. 
\end{remark}

The following facts are clear. 
\begin{lemma} \label{lem:struct}
\begin{enumerate}
\item
The category $\wh\HB(\delta)$ inherits from  $\HB(\delta)$ a monoidal structure. 

\item
The full subcategory of $\wh\HB(\delta)$ with objects $v^r$ for all $r\ge 0$ is isomorphic to the usual Brauer category $\CB(\delta)$  with parameter $\delta$.  
In particular,   if $\CM=\CV=\{v\}$, then $\wh\HB(\delta)$ coincides with $\CB(\delta)$. 
\end{enumerate}
\end{lemma}

Note that the four-term relations in $\Hom_{\wh\HB(\delta)}((a, b, c), (a, b, c))$ with at least one colour in $\CV$ simplify considerably. We shall discuss this in more detail in the next section.  

Denote by $H_{i j}\in \wh\HB(\delta)$ the image of $\BH_{i j}(v^r)$ 
(see Figure \ref{fig:Hijr}) for $1\le i< j\le r$. 
Note that all $H_{i j}$ belong to the Brauer algebra 
$B_r(\delta)\subset \End_{\wh\HB(\delta)}(v^r)$. 
The following algebra homomorphism implicitly appeared in \cite{LZ06}. 
\begin{lemma}\label{lem:T-H-B} 
Retain the notation above. 
There is an algebra homomorphism 
\[
T_r \lra B_r(\delta), \quad t_{i j} \mapsto H_{i j}, \ 1\le i<j\le r, 
\]
which is surjective unless $\delta = 2$. 
\end{lemma}
\begin{proof}
The existence of the algebra homomorphism is an easy consequence of Lemma \ref{lem:H-4-t}. Another way to view it is to regard the map $t_{i j} \mapsto H_{i j}$ 
as the composition of maps 
$t_{i j} \mapsto  \BH_{i j}(v^r) \mapsto H_{i j}$. 
Then the statement follows from Lemma \ref{lem:Murphy}(1). 

The second statement follows from Lemma \ref{lem:H-skew-sym}. 
\end{proof}

\subsection{The multi-polar Brauer category}\label{sect:mpolar}
If $\CM=\CV=\{v\}$,  
then $\wh\HB(\delta)$ reduces to the  usual Bauer category $\CB(\delta)$ \cite{LZ15}.
 
Let us now consider simplest cases of $\wh\HB(\delta)$ with $\CM\supsetneq\CV$.

\subsubsection{Multi-polar Brauer category}
Fix an element $m\in\CM$ which is not contained in $\CV$. 
Let $\MPB(\delta)$ be the full subcategory of $\wh\HB(\delta)$, whose objects are sequences with all entries in $\{m, v\}$. 
One may also think about $\MPB(\delta)$ independently and regard it as $\wh\HB(\delta)$ in the case $\CM= \{m, v\}\supset\CV=\{v\}$. 

We refer to $\MPB(\delta)$ as the {\em multi-polar Brauer category}, for reasons which will  become clear presently.   

As a monoidal category, $\MPB(\delta)$ is generated by the objects 
$m, v$, and morphisms 
\beq
 I, \ X, \ \cup, \  \cap, \ 
\I=I^m,  \ \BH=\BH(m, v), \  \BH(v, m), \ \BH(m, m), 
\eeq
where we note that $\BH(v, v)$ has been identified with $H$. 


There is a more convenient way of representing diagrams in the category $\MPB(\delta)$ by adapting the notion of ``polar tangle diagrams'' \cite{ILZ} to the present context.
We depict arcs coloured by $m$ as thick arcs, and those coloured by $v$ as thin arcs. 
Then the diagram is automatically endowed with information about the colours of its arcs, thus we can drop the letters for colours.  

\begin{figure}[h]
\setlength{\unitlength}{0.30mm}
\begin{picture}(200, 80)(-20,0)

\put(-22, 60){\line(0, 1){20}}
\put(-22, 70){\uwave{\hspace{7mm}}}
\qbezier(-22, 60)(0, 40)(22, 60)
\put(22, 60){\line(0, 1){20}}

\qbezier(10, 80)(55, 50)(100, 80)

{
\linethickness{.75mm}
\put(0, 0){\line(0, 1){20}}
\put(40, 0){\line(0, 1){20}}
\put(80, 0){\line(0, 1){60}}
\qbezier(0, 20)(0, 20)(40, 40)
\qbezier(80, 60)(120, 80)(120, 80)
\qbezier(0, 40)(40, 20)(40, 20)
\put(40, 40){\line(0, 1){40}}
\put(0, 40){\line(0, 1){40}}
}

\put(0, 10){\uwave{\hspace{13mm}}}
\put(40, 50){\uwave{\hspace{13mm}}}

\qbezier(60, 80)(60, 80)(100, 60)
\qbezier(100, 0)(110, 50)(100, 60) 

\qbezier(60, 0)(90, 60)(120, 0)

\end{picture}
\caption{A multi-polar diagram $(m^2, v, m, v^2)\to (v, m, v^2, m, v^2, m)$}
\label{fig:pdiag}
\end{figure}
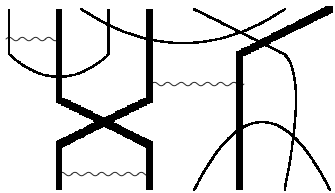

The space of morphisms of $\MPB(\delta)$ is spanned by Brauer diagrams 
coloured by $m$ and $v$ with connectors such that any of them 
connects either two thick arcs, 
or a thick arc and thin arc (but not two thin arcs). 
Figure \ref{fig:pdiag} is an
example of such a diagram $(m^2, v, m, v^2)\to (v, m, v^2, m, v^2, m)$.
We will call such diagrams {\em multi-polar Brauer diagrams}, and hence the name
multi-polar Brauer category for $\MPB(\delta)$. 

We shall use the diagrams shown in Figure \ref{fig:bbH} later.

\begin{figure}[h]
\begin{picture}(80, 40)(0,0)
\put(0, 20){$\BH   =$}
{
\linethickness{1mm}
\put(35, 0){\line(0, 1){40}}
}
\put(35, 20){\uwave{\hspace{7mm}}}

\put(55, 0){\line(0, 1){40}}

\put(65, 5){, }

\end{picture} 
\begin{picture}(120, 40)(-10, 0)
\put(0, 20){$\BH ^T  =$}
{
\linethickness{1mm}
\put(35, 0){\line(0, 1){40}}
}
\put(35, 20){\uwave{\hspace{7mm}}}

\put(55, 8){\line(0, 1){20}}

\qbezier(55, 28)(70, 38)(75, 0)
\qbezier(55, 8)(70, -2)(75, 40)

\put(85, 5){,}
\end{picture}
\begin{picture}(100, 40)(-20, 0)
\put(-20, 20){$(\BH^\ell) ^T  =$}
{
\linethickness{1mm}
\put(35, 0){\line(0, 1){40}}
}
\put(35, 33){\uwave{\hspace{7mm}}}
\put(40, 15){$\vdots$ $\ell$}
\put(35, 12){\uwave{\hspace{7mm}}}

\put(55, 5){\line(0, 1){30}}

\qbezier(55, 35)(70, 40)(75, 0)
\qbezier(55, 5)(70, 0)(75, 40)

\end{picture}
\caption{Diagrams $\BH$, $\BH^T$ and $(\BH^\ell)^T$}
\label{fig:bbH}
\label{fig:H-transpose}
\end{figure}

%

We note in particular that the four-term relations in $\Hom_{\MPB(\delta)}((a, b, c), (a, b, c))$ with all $a=b=c=v$ identically hold because $H$ obeys Lemma \ref{lem:H-4-t}.   
If two of $a, b, c$ are equal to $v$, the four-term relations also simplify considerably. 

\begin{lemma}\label{lem:4-t-id}  The following four-term relation is an identity in $\MPB(\delta)$ 
\[
[\mathbb{H}_{12}(m, v^2)+ \mathbb{H}_{1 3}(m, v^2), \mathbb{H}_{23}(m, v^2)]=0, \quad \forall  m\in\CM.
\]
\end{lemma}
\begin{proof} 
This is a consequence of the (right) skew symmetry $\BH(m, v)^T=-\BH(m, v)$ of $\BH(m, v)$,  which can be depicted by  Figure \ref{fig:H-skew}. 
\begin{figure}[h]
\begin{picture}(100, 40)(0,0)
{
\linethickness{1mm}
\put(35, 0){\line(0, 1){40}}
}
\put(35, 20){\uwave{\hspace{7mm}}}

\put(55, 8){\line(0, 1){20}}

\qbezier(55, 28)(70, 38)(75, 0)
\qbezier(55, 8)(70, -2)(75, 40)

\end{picture}
\begin{picture}(100, 40)(20, 0)
\put(0, 16){$=\ - \ $}
{
\linethickness{1mm}
\put(35, 0){\line(0, 1){40}}
}
\put(35, 20){\uwave{\hspace{7mm}}}

\put(55, 0){\line(0, 1){40}}

\put(65, 5){. }

\end{picture} 
\caption{Skew symmetry of  $\BH(m, v)$}
\label{fig:H-skew}
\end{figure}  

The proof of the lemma is entirely straightforward, but it provides a good opportunity to illustrate the power of the skew symmetry of $\BH$. We therefore spell out the details.  

Note that $\BH_{2 3}(m, v^2) = \I\ot H$. Thus we have
\[
\qquad\qquad\qquad
\begin{picture}(180, 40)(-70, 0)
\put(-100, 20){$\BH _{12}(m, v^2) \BH _{23}(m, v^2) \  =$}
{
\linethickness{1mm}
\put(35, 0){\line(0, 1){40}}
}
\put(35, 30){\uwave{\hspace{5mm}}}

\put(50, 40){\line(0, -1){25}}
\put(65, 40){\line(0, -1){25}}

\qbezier(50, 15)(50, 15)(65, 0)
\qbezier(65, 15)(65, 15)(50, 0)


\put(80, 18){$-$}
\end{picture}
\begin{picture}(110, 40)(30, 0)
{
\linethickness{1mm}
\put(35, 0){\line(0, 1){40}}
}
\put(35, 30){\uwave{\hspace{5mm}}}

\put(50, 40){\line(0, -1){25}}
\put(65, 40){\line(0, -1){25}}

\qbezier(50, 15)(57, 5)(65, 15)

\qbezier(50, 0)(57, 10)(65, 0)



\put(80, 5){,}
\end{picture}
\]

\[
\qquad\qquad\qquad
\begin{picture}(180, 40)(-70, 0)
\put(-100, 20){$\BH _{1 3}(m, v^2) \BH _{23}(m, v^2) \  =$}
{
\linethickness{1mm}
\put(35, 0){\line(0, 1){40}}
}
\put(35, 15){\uwave{\hspace{5mm}}}

\put(50, 25){\line(0, -1){25}}
\put(65, 25){\line(0, -1){25}}

\qbezier(50, 25)(50, 25)(65, 40)
\qbezier(65, 25)(65, 25)(50, 40)


\put(80, 18){$-$}
\end{picture}
\begin{picture}(110, 40)(30, 0)
{
\linethickness{1mm}
\put(35, 0){\line(0, 1){40}}
}
\put(35, 22){\uwave{\hspace{5mm}}}

\put(50, 15){\line(0, 1){10}}
\put(65, 15){\line(0, 1){10}}
\qbezier(50, 25)(50, 25)(65, 40)
\qbezier(50, 40)(50, 40)(65, 25)

\qbezier(50, 15)(57, 5)(65, 15)
\qbezier(50, 0)(57, 10)(65, 0)



\put(80, 5){,}
\end{picture}
\]

%
\[
\qquad\qquad\qquad
\begin{picture}(180, 40)(-70, 0)
\put(-100, 20){$\BH _{23}(m, v^2) \BH _{12}(m, v^2) \  =$}
{
\linethickness{1mm}
\put(35, 0){\line(0, 1){40}}
}
\put(35, 15){\uwave{\hspace{5mm}}}

\put(50, 25){\line(0, -1){25}}
\put(65, 25){\line(0, -1){25}}

\qbezier(50, 25)(50, 25)(65, 40)
\qbezier(65, 25)(65, 25)(50, 40)


\put(80, 18){$-$}
\end{picture}
\begin{picture}(110, 40)(30, 0)
{
\linethickness{1mm}
\put(35, 0){\line(0, 1){40}}
}
\put(35, 15){\uwave{\hspace{5mm}}}

\put(50, 0){\line(0, 1){25}}
\put(65, 0){\line(0, 1){25}}

\qbezier(50, 25)(57, 35)(65, 25)

\qbezier(50, 40)(57, 30)(65, 40)



\put(80, 5){,}
\end{picture}
\]

\[
\qquad\qquad\qquad
\begin{picture}(180, 40)(-70, 0)
\put(-100, 20){$\BH _{23}(m, v^2) \BH _{13}(m, v^2) \  =$}
{
\linethickness{1mm}
\put(35, 0){\line(0, 1){40}}
}
\put(35, 30){\uwave{\hspace{5mm}}}

\put(50, 40){\line(0, -1){25}}
\put(65, 40){\line(0, -1){25}}

\qbezier(50, 15)(50, 15)(65, 0)
\qbezier(65, 15)(65, 15)(50, 0)


\put(80, 18){$-$}
\end{picture}
\begin{picture}(110, 40)(30, 0)
{
\linethickness{1mm}
\put(35, 0){\line(0, 1){40}}
}
\put(35, 18){\uwave{\hspace{5mm}}}

\qbezier(50, 40)(57, 25)(65, 40)
\qbezier(50, 20)(57, 35)(65, 20)

\put(50, 20){\line(0, -1){10}}
\put(65, 20){\line(0, -1){10}}

\qbezier(50, 10)(50, 10)(65, 0)
\qbezier(50, 0)(50, 0)(65, 10)

\put(80, 5){.}
\end{picture}
\]
The skew symmetry of $\BH $ implies
\[
\begin{picture}(100, 40)(30, 0)
{
\linethickness{1mm}
\put(35, 0){\line(0, 1){40}}
}
\put(35, 22){\uwave{\hspace{5mm}}}

\put(50, 15){\line(0, 1){10}}
\put(65, 15){\line(0, 1){10}}
\qbezier(50, 25)(50, 25)(65, 40)
\qbezier(50, 40)(50, 40)(65, 25)

\qbezier(50, 15)(57, 5)(65, 15)
\qbezier(50, 0)(57, 10)(65, 0)

\put(80, 20){$ = \ -$}
\end{picture}
\begin{picture}(50, 40)(50, 0)
{
\linethickness{1mm}
\put(35, 0){\line(0, 1){40}}
}
\put(35, 30){\uwave{\hspace{5mm}}}

\put(50, 40){\line(0, -1){25}}
\put(65, 40){\line(0, -1){25}}

\qbezier(50, 15)(57, 5)(65, 15)

\qbezier(50, 0)(57, 10)(65, 0)

\end{picture}
%
\begin{picture}(110, 40)(5, 0)
\put(-8, 18){$\text{and}$}
{
\linethickness{1mm}
\put(35, 0){\line(0, 1){40}}
}
\put(35, 18){\uwave{\hspace{5mm}}}

\qbezier(50, 40)(57, 25)(65, 40)
\qbezier(50, 20)(57, 35)(65, 20)

\put(50, 20){\line(0, -1){10}}
\put(65, 20){\line(0, -1){10}}

\qbezier(50, 10)(50, 10)(65, 0)
\qbezier(50, 0)(50, 0)(65, 10)

\put(80, 20){$=\ - $}
\end{picture}
\begin{picture}(60, 40)(35, 0)
{
\linethickness{1mm}
\put(35, 0){\line(0, 1){40}}
}
\put(35, 15){\uwave{\hspace{5mm}}}

\put(50, 0){\line(0, 1){25}}
\put(65, 0){\line(0, 1){25}}

\qbezier(50, 25)(57, 35)(65, 25)

\qbezier(50, 40)(57, 30)(65, 40)



\put(70, 5){.}
\end{picture}
\]

\noindent
Using these relations in the formulae above, we obtain the identity
\[
(\BH _{12}(m, v^2) + \BH _{13}(m, v^2) )\BH _{23}(m, v^2)=
\BH _{23}(m, v^2)(\BH _{12}(m, v^2) + \BH _{13}(m, v^2) ).
\]
This completes the proof.  
\end{proof}

The four-term relation 
\beq\label{eq:4t-right}
[\BH_{1 2}(m, v^2), \BH_{1 3}(m, v^2)+ \BH_{2 3}(m, v^2)]=0 
\eeq
now can be represented  diagrammatically by Figure \ref{fig:4-term}.

\begin{figure}[h]
\begin{picture}(50, 70)(0,0)
{
\linethickness{1mm}
\put(10, 0){\line(0, 1){70}}
}

\put(10, 50){\uwave{\hspace{5mm}}}
\put(10, 25){\uwave{\hspace{5mm}}}

\put(25, 70){\line(0, -1){30}}
\put(40, 70){\line(0, -1){30}}

\qbezier(25, 40)(25, 40)(40, 30)
\qbezier(25, 30)(25, 30)(40, 40)

\put(25, 30){\line(0, -1){15}}
\put(40, 30){\line(0, -1){15}}
\qbezier(25, 15)(25, 15)(40, 0)
\qbezier(25, 0)(25, 0)(40, 15)
\put(45, 30){$-$}
\end{picture}
\begin{picture}(50, 70)(0,0)
{
\linethickness{1mm}
\put(10, 0){\line(0, 1){70}}
}

\put(10, 50){\uwave{\hspace{5mm}}}
\put(10, 25){\uwave{\hspace{5mm}}}

\qbezier(25, 55)(25, 55)(40, 70)
\qbezier(25, 70)(25, 70)(40, 55)

\put(25, 55){\line(0, -1){15}}
\put(40, 55){\line(0, -1){15}}

\put(25, 30){\line(0, -1){30}}
\put(40, 30){\line(0, -1){30}}
\qbezier(25, 30)(25, 30)(40, 40)
\qbezier(25, 40)(25, 40)(40, 30)
\put(45, 30){$+$}
\end{picture}
\begin{picture}(50, 70)(0,0)
{
\linethickness{1mm}
\put(10, 0){\line(0, 1){70}}
}

\put(10, 50){\uwave{\hspace{5mm}}}

\put(25, 70){\line(0, -1){40}}
\put(40, 70){\line(0, -1){40}}


\put(25, 30){\line(0, -1){15}}
\put(40, 30){\line(0, -1){15}}
\qbezier(25, 15)(25, 15)(40, 0)
\qbezier(25, 0)(25, 0)(40, 15)
\put(45, 30){$-$}
\end{picture}
\begin{picture}(50, 70)(0,0)
{
\linethickness{1mm}
\put(10, 0){\line(0, 1){70}}
}

\put(10, 25){\uwave{\hspace{5mm}}}

\qbezier(25, 55)(25, 55)(40, 70)
\qbezier(25, 70)(25, 70)(40, 55)

\put(25, 55){\line(0, -1){15}}
\put(40, 55){\line(0, -1){15}}

\put(25, 40){\line(0, -1){40}}
\put(40, 40){\line(0, -1){40}}
\put(45, 30){$+$}
\end{picture}
\begin{picture}(50, 70)(0,0)
{
\linethickness{1mm}
\put(10, 0){\line(0, 1){70}}
}

\put(10, 25){\uwave{\hspace{5mm}}}
\qbezier(25, 60)(33, 45)(40, 60)
\put(25, 60){\line(0, 1){10}}
\put(40, 60){\line(0, 1){10}}


\qbezier(25, 40)(33, 55)(40, 40)
\put(25, 40){\line(0, -1){40}}
\put(40, 40){\line(0, -1){40}}
\put(45, 30){$-$}
\end{picture}
\begin{picture}(80, 70)(0,0)
{
\linethickness{1mm}
\put(10, 0){\line(0, 1){70}}
}

\put(10, 50){\uwave{\hspace{5mm}}}

\qbezier(25, 40)(33, 25)(40, 40)
\put(25, 70){\line(0, -1){30}}
\put(40, 70){\line(0, -1){30}}


\qbezier(25, 20)(33, 35)(40, 20)
\put(25, 20){\line(0, -1){20}}
\put(40, 20){\line(0, -1){20}}
\put(45, 30){$=0$.}
\end{picture}
\caption{A four-term relation}
\label{fig:4-term}
\end{figure}

\subsubsection{Polar Brauer category}\label{ss:polar}

The category $\MPB(\delta)$ contains some interesting subcategories which we identify below. 
\begin{enumerate}[i)]
\item Denote by $\UPB(\delta)$ the full subcategory of $\MPB(\delta)$ with objects 
$(v^r, m, v^s)$ for all $r, s\ge 0$,  and refer to it as the {\em unipolar Brauer category} with parameter $\delta$.
\item Denote by $\PB(\delta)$ the full subcategory of $\MPB(\delta)$ with objects $(m, v^r)$ for all $r\ge 0$, and refer to it as the {\em polar Brauer category} with parameter $\delta$. 
\item The category $\UPB(\delta)$ contains $\PB(\delta)$ as a full subcategory.
\end{enumerate}
%
%
Clearly $\UPB(\delta)$ and $\PB(\delta)$ no longer have monoidal structures.  However,  we have the following result. 
\begin{lemma}\label{lem:mod-cat}
Regard $\CB(\delta)$ as the full subcategory of $\MPB(\delta)$ with objects $v^r$ for all $r\in\N$. Then restrictions of the monoidal structure of $\MPB(\delta)$ lead to the following functors: 
\begin{enumerate}
\item $
\PB(\delta)\times \CB(\delta)\lra \PB(\delta),
$
which makes $\PB(\delta)$ a right module category over the monoidal category $\CB(\delta)$; and 
\item $
\CB(\delta)\times \UPB(\delta)\times \CB(\delta)\lra \UPB(\delta),
$
which makes $\UPB(\delta)$ a bi-module category over the monoidal category $\CB(\delta)$.
\end{enumerate}

\end{lemma}

Note that the morphisms of $\PB(\delta)$ are linear combinations of polar diagrams with a single pole on the left. Such a polar Brauer diagram is shown in Figure \ref{fig:AffD}.  
\begin{figure}[h]
\begin{picture}(120, 190)(-20,0)
{
\linethickness{1mm}
\put(0, 0){\line(0, 1){190}}
}
\put(25, 170){\line(0, 1){20}}
\put(35, 170){\line(0, 1){20}}
\put(55, 170){\line(0, 1){20}}
\put(40, 180){...}

\put(35, 155){\small$D_k$}
\put(20, 150){\line(0, 1){20}}
\put(60, 150){\line(0, 1){20}}
\put(20, 150){\line(1, 0){40}}
\put(20, 170){\line(1, 0){40}}

\put(25, 130){\line(0, 1){20}}
\put(35, 130){\line(0, 1){20}}
\put(55, 130){\line(0, 1){20}}
\put(40, 140){...}

\put(0, 140){\uwave{\hspace{9mm}}}

\put(40, 115){$\vdots$}

\put(0, 100){\uwave{\hspace{9mm}}}

\put(25, 90){\line(0, 1){20}}
\put(35, 90){\line(0, 1){20}}
\put(55, 90){\line(0, 1){20}}
\put(40, 100){...}
\put(35, 75){\small$D_2$}
\put(20, 70){\line(0, 1){20}}
\put(60, 70){\line(0, 1){20}}
\put(20, 70){\line(1, 0){40}}
\put(20, 90){\line(1, 0){40}}

\put(25, 70){\line(0, -1){30}}
\put(35, 70){\line(0, -1){30}}
\put(55, 70){\line(0, -1){30}}
\put(40, 55){...}

\put(0, 60){\uwave{\hspace{9mm}}}

\put(35, 25){\small$D_1$}
\put(20, 20){\line(0, 1){20}}
\put(60, 20){\line(0, 1){20}}
\put(20, 20){\line(1, 0){40}}
\put(20, 40){\line(1, 0){40}}

\put(25, 20){\line(0, -1){20}}
\put(35, 20){\line(0, -1){20}}
\put(55, 20){\line(0, -1){20}}
\put(40, 10){...}

\put(65, 5)

\end{picture} 
\caption{Polar Brauer diagram}
\label{fig:AffD}
\end{figure}
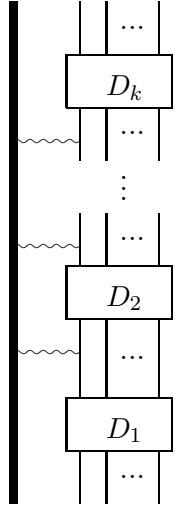
The unique thick arc in each such diagram appears on the left boundary. The 
$D_i$ for all $i=1, \dots, k$ are usual Brauer $(r_{i-1}, r_i)$-diagrams for non-negative integers $r_j$. We refer to the diagram in Figure \ref{fig:AffD} as a {\em polar Brauer $(r_0, r_k)$-diagram}.

Applications to orthosymplectic Lie superalgebras of $\PB(\delta)$ will be investigated later. 

\subsection{The multi-polar enhanced Brauer category}
As we will explain in Remark \ref{rem:funct-F},  while the Brauer category enables us to describe all the morphisms among the tensor powers of the natural module $V$ 
for the orthosymplectic supergroup $\OSp(V; \omega)$, it does not yield all the $\osp(V; \omega)$ morphisms among the tensor powers of the natural module. 
This problem is present even for the orthogonal Lie algebra $\fso_m(\C)$. 
This  problem was resolved in the $\fso_m(\C)$ case in \cite{LZ17-Ecate} by introducing  the enhanced Brauer category described in Section \ref{ex:B-SO}, 
which is isomorphic to the full subcategory of $\fso_m(\C)$-Mod with objects $(\C^m)^{\ot r}$ for all $r\in\N$.  

To study $\osp(V; \omega)$-modules of the type $M_1\ot M_2\ot\dots\ot M_r$ with $M_i$ belonging to some 
given finite set $\CM$ of modules containing $V$,   one will need multi-polar generalisations of enhanced Brauer categories. Such categories are quotient categories of $\HB_{\CG_1}(\delta_\CV)$ with non-trivial $\CG_1$, where $\CV\subset\CM$  is a subset of finite dimensional simple $\osp(V; \omega)$-modules including $V$.

It is entirely straightforward to combine Section \ref{ex:B-SO} and Section \ref{sect:mpolar} 
to create a multi-polar enhanced Brauer category for the orthogonal Lie algebra $\fso_m$. 
We will not give the details here. Instead, we work out a multi-polar enhanced Brauer category 
for $G_2$ in Section \ref{sect:EB-G}, where $\BH(v, v)$ satisfies a quartic polynomial relation,
in analogy with the logarithm of the $R$-matrix acting on the tensor square. 

\;

\subsection{Endomorphism algebras of the multi-polar Brauer category} 

For any usual Brauer diagram $A\in \Hom_{\CB(\delta)}(r, s)$, we write $\A_0= \I \ot A$. In particular, we have $\X_0= \I \ot X$ for $X=X_{v  v}\in \Hom_{\CB(\delta)}(2, 2)$. 
Denote by $\Pi$ and $\amalg$ respectively the polar Brauer $(2, 0)$- and $(0, 2)$-diagrams shown in Figure \ref{fig:A-U}.

\begin{figure}[h]
\begin{picture}(85, 30)(-20, 5)
\put(-30, 10){$\Pi=$}
{
\linethickness{1mm}
\put(0, 0){\line(0, 1){30}}
}
\qbezier(10, 0)(17, 40)(25, 0)
\put(30, 5){,}
\end{picture}
\begin{picture}(30, 30)(-20, 5)
\put(-30, 10){$\amalg=$}
{
\linethickness{1mm}
\put(0, 0){\line(0, 1){30}}
}
\qbezier(10, 30)(17, -10)(25, 30)
\end{picture}
\caption{The polar diagrams  $\Pi$ and $\amalg$}
\label{fig:A-U}
\end{figure}

\noindent
Consider the following elements of $\End_{\MPB(\delta)}(m)$
\beq\label{eq:Z-gener}
 Z_\ell=\Pi (\BH^\ell \ot I)\amalg, \quad  \ell\ge 1, 
\eeq
which are represented graphically in Figure \ref{fig:Z}.
\begin{figure}[h]
\begin{picture}(110, 60)(50, -5)
\put(0, 18){$Z_\ell  =$}
{
\linethickness{1mm}
\put(35, -5){\line(0, 1){55}}
}
\put(35, 35){\uwave{\hspace{7mm}}}
\put(35, 28){\uwave{\hspace{7mm}}}
\put(35, 12){\uwave{\hspace{7mm}}}
\put(44, 13){{\tiny$\vdots$}}

\put(55, 5){\line(0, 1){35}}

\qbezier(55, 40)(70, 50)(73, 22)
\qbezier(55, 5)(70, -5)(73, 22)

\put(95, 15){with $\ell$ connectors.}
\end{picture}
\caption{The central elements $Z_\ell$}
\label{fig:Z}
\end{figure}
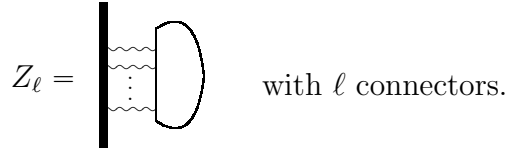

We also have the elements $\wt{Z}_\ell =(\cap\ot\I) (I\ot \BH(m, v)^\ell)(\cup\ot\I)$ for $\ell\in\N$. Graphically $\wt{Z}_\ell$ is represented by the diagram 
obtained by reflecting Figure \ref{fig:Z} in a vertical line. The following then holds. 

\begin{lemma}\label{eq:ZZ}
$Z_\ell =\wt{Z}_\ell$ for all $\ell$. 
\end{lemma}
\begin{proof}
We have $\BH(m, v)= X_{m, v} \BH(v, m) X_{v, m}$ by the symmetry of $\BH(a, b)$. 
Hence 
\[
\wt{Z}_\ell =(\cap\ot\I)(I\ot X_{m, v}) (I\ot \BH^\ell)(X_{v, m}\ot I)(\cup\ot\I).
\]
The right hand side can be expressed as $\Pi (\I\ot X) (\BH \ot I)(\I\ot X) \amalg$ by sliding the thin arc to the right.  This is clearly equal to $\Pi  (\BH \ot I) \amalg=Z_\ell$.
\end{proof}

Thus we confine our attention to $Z_\ell$. 

\begin{lemma}   \label{lem:central} The elements $Z_1, Z_2$ and $Z_3$ have the following properties:
\begin{enumerate}

\item If the characteristic of $\K$ is not $2$, then $Z_1=0$, and $ 2 Z_3 = (2-\delta)Z_2$;

\item The element $Z_2\ot I$ commutes with $\BH(m, v)$, and $I\ot Z_2$ commutes with $\BH(v, m)$. 

\item The element $Z_2$ is central in $\UPB(\delta)$ in the sense that
\[
\forall\A\in\Hom_{\UPB(\delta)}((v^r, m, v^s),  (v^{r'}, m, v^{s'})),  
\]
\[
\baln
(I(v^{r'})\ot Z_2\ot I(v^{s'})) \A= \A (I(v^{r})\ot Z_2\ot I(v^{s'})).
\ealn
\]

\end{enumerate}
\end{lemma}
\begin{proof} 
It is an immediate consequence of the skew symmetry of $\BH $ that 
$Z_1=-Z_1=0$ if $2\neq 0$.  
To prove the second relation in (1), we pre-multiply (vertically) the four-term relation in Figure \ref{fig:4-term} by $\Pi$,  and then apply skew symmetry of $\BH$ several times to obtain the relation depicted in Figure \ref{fig:deg0-relation}.  
\begin{figure}[h]
\begin{picture}(60, 55)(0,0)
{
\linethickness{1mm}
\put(10, 0){\line(0, 1){55}}
}

\put(10, 40){\uwave{\hspace{5mm}}}
\put(10, 25){\uwave{\hspace{5mm}}}

\put(25, 45){\line(0, -1){25}}
\put(40, 45){\line(0, -1){25}}

\qbezier(25, 45)(33, 60)(40, 45)


\put(25, 30){\line(0, -1){15}}
\put(40, 30){\line(0, -1){15}}
\qbezier(25, 15)(25, 15)(40, 0)
\qbezier(25, 0)(25, 0)(40, 15)
\put(50, 23){$-$}
\end{picture}
\begin{picture}(100, 50)(0,0)
{
\linethickness{1mm}
\put(10, 0){\line(0, 1){55}}
}

\put(10, 40){\uwave{\hspace{5mm}}}
\put(10, 25){\uwave{\hspace{5mm}}}

\qbezier(25, 45)(32, 60)(40, 45)

\put(25, 45){\line(0, -1){45}}
\put(40, 45){\line(0, -1){45}}
\put(50, 23){$= \ (\delta-2)$}
\end{picture}
\begin{picture}(50, 50)(0,0)
{
\linethickness{1mm}
\put(10, 0){\line(0, 1){55}}
}

\put(10, 25){\uwave{\hspace{5mm}}}

\qbezier(25, 40)(33, 55)(40, 40)
\put(25, 40){\line(0, -1){40}}
\put(40, 40){\line(0, -1){40}}
\put(45, 5){.}
\end{picture}
\caption{A relation in $\Hom_{\PB(\delta)}(1, 1)$}
\label{fig:deg0-relation}
\end{figure}  
We then post-multiply (vertically) each diagram in  Figure \ref{fig:deg0-relation}  by $(\BH\ot I) \amalg$ and again apply the skew symmetry of $\BH$. This leads directly to 
$
2 Z_3 = (2-\delta) Z_2. 
$

To prove part (2), we note that the
four-term relation \eqref{eq:4t-right} leads to
\[
(\I\ot I\ot \cap)\left([\BH_{1 2}, (\BH_{1 3}+ \BH_{2 3})^\ell]\ot I\right)(\I\ot I\ot \cup)=0
\] 
for all $\ell$. Define
\beq
\BB_\ell= (\I\ot I\ot \cap)\left((\BH_{13} + \BH_{23})^\ell \ot I\right)(\I\ot I\ot \cup).
\eeq
Then the above relation becomes  
\beq\label{eq:Xell}
\BH \BB_\ell - \BB_\ell \BH = 0, \quad \forall \ell. 
\eeq
For $\ell=2$, we have $\BB_2 =Z_2\ot I+ 2(\delta-1)\I\ot I$, and hence 
$\BH (Z_2\ot I)= (Z_2\ot I)\BH$ by \eqref{eq:Xell}, where the relation between $\BB_2$ and $Z_2\ot I$ follows from the formulae 
\[
\baln
&(\I\ot I\ot \cap)(\BH_{02}^2\ot I)(\I\ot I\ot \cup) = Z_2\ot I, \\
&(\I\ot I\ot \cap)((\BH_{02}\BH_{12}+\BH_{12}\BH_{02})\ot I)(\I\ot I\ot \cup) = 0, \\
&(\I\ot I\ot \cap)(\BH_{12}^2\ot I)(\I\ot I\ot \cup) =2(\delta-1)\I\ot I. 
\ealn
\]

Similarly, using the four-term relation
\[
[\BH_{1 2}(v^2, m)+\BH_{1 3}(v^2, m),  \BH_{2 3}(v^2, m)]=0,  
\]
we see that 
\beq\label{eq:Xell-1}
\BH(v, m) \wt\BB_\ell -  \wt\BB_\ell \BH(v, m)=0, \quad \forall \ell, 
\eeq
where 
$
\wt\BB_\ell =  
(\cap\ot I \ot \I)\left(I\ot\left(\BH_{1 2}(v^2, m)+\BH_{1 3}(v^2, m)\right)^\ell \right)(\cup\ot I \ot \I)
$.
For $\ell=2$, we have $\wt\BB_2 =\wt{Z}_2\ot \I+ 2(\delta-1)I\ot \I$. By Lemma \ref{eq:ZZ}, 
$\wt{Z}_2={Z}_2$, hence $\wt\BB_2 ={Z}_2\ot \I+ 2(\delta-1)I\ot \I$. Substituting  this into \eqref{eq:Xell-1}, we obtain   
$\BH(v, m) (I\ot {Z}_2) -  (I\ot {Z}_2) \BH(v, m)=0$.

Now to prove part (3), note that any morphism $\A$ of $\UPB(\delta)$ is a composition of elements of the form $A\ot \BD\ot B$, where $A$ and $B$ are ordinary Brauer diagrams, and $\BD$ is $\I$, $\BH(m, v)$, $\BH(v, m)$, $X_{m, v}$ or $X_{v, m}$.  Part (3) now follows from part (2) and the facts that  
$X_{m, v}(Z_2\ot I)= (I\ot Z_2)X_{m, v}$ and $(Z_2\ot I)X_{v, m}=X_{v, m} (I\ot Z_2)$, which  
are easily obtained by sliding thin arcs over infinitesimal braids.  
\end{proof}

\begin{lemma} \label{lem:HT2} 
Let $\BG_\ell=Z_\ell\ot I  -  \BH^\ell$, and $\Phi=(1-\delta)\I \ot I-\BH$ and let
 $(\BH^{\ell})^T$ be the diagram in Figure \ref{fig:bbH}. 

Then the following relations hold in  $\End_{\MPB(\delta)}((m, v))$ for all $\ell\ge 0$, 
\beq
(\BH^{\ell+1})^T  - (\BH^\ell)^T \Phi - \BG_\ell =0, \label{eq:HT-H-1}\\
(\BH^{\ell+1})^T - \Phi (\BH^{\ell})^T - \BG_{\ell} =0, \label{eq:HT-H-2}
\eeq
Furthermore, 
$\BH^k$ and $(\BH^{\ell})^T$ commute for all $k, \ell$.
\end{lemma}
\begin{proof}
If $\ell=0$, both equations are equivalent to the skew symmetry of $\BH$ in view of the relation $\BG_0=(\delta-1)\I \ot I$. 
For any $\ell\ge 1$, we have the iterated four-term relation Figure \ref{fig:4-term-iterated}.

\begin{figure}[h]
\begin{picture}(50, 70)(0,0)
{
\linethickness{1mm}
\put(10, 0){\line(0, 1){70}}
}
\put(10, 65){\uwave{\hspace{5mm}}}
\put(15, 52){\tiny$\vdots\, \ell$}
\put(10, 50){\uwave{\hspace{5mm}}}
\put(10, 25){\uwave{\hspace{5mm}}}

\put(25, 70){\line(0, -1){30}}
\put(40, 70){\line(0, -1){30}}

\qbezier(25, 40)(25, 40)(40, 30)
\qbezier(25, 30)(25, 30)(40, 40)

\put(25, 30){\line(0, -1){15}}
\put(40, 30){\line(0, -1){15}}
\qbezier(25, 15)(25, 15)(40, 0)
\qbezier(25, 0)(25, 0)(40, 15)
\put(45, 30){$-$}
\end{picture}
\begin{picture}(50, 70)(0,0)
{
\linethickness{1mm}
\put(10, 0){\line(0, 1){70}}
}

\put(10, 50){\uwave{\hspace{5mm}}}
\put(10, 25){\uwave{\hspace{5mm}}}
\put(15, 12){\tiny$\vdots\, \ell$}
\put(10, 10){\uwave{\hspace{5mm}}}

\qbezier(25, 55)(25, 55)(40, 70)
\qbezier(25, 70)(25, 70)(40, 55)

\put(25, 55){\line(0, -1){15}}
\put(40, 55){\line(0, -1){15}}

\put(25, 30){\line(0, -1){30}}
\put(40, 30){\line(0, -1){30}}
\qbezier(25, 30)(25, 30)(40, 40)
\qbezier(25, 40)(25, 40)(40, 30)
\put(45, 30){$+$}
\end{picture}
\begin{picture}(50, 70)(0,0)
{
\linethickness{1mm}
\put(10, 0){\line(0, 1){70}}
}

\put(10, 65){\uwave{\hspace{5mm}}}
\put(15, 50){\tiny$\vdots\, \ell$}
\put(10, 45){\uwave{\hspace{5mm}}}

\put(25, 70){\line(0, -1){40}}
\put(40, 70){\line(0, -1){40}}


\put(25, 30){\line(0, -1){15}}
\put(40, 30){\line(0, -1){15}}
\qbezier(25, 15)(25, 15)(40, 0)
\qbezier(25, 0)(25, 0)(40, 15)
\put(45, 30){$-$}
\end{picture}
\begin{picture}(50, 70)(0,0)
{
\linethickness{1mm}
\put(10, 0){\line(0, 1){70}}
}

\put(10, 25){\uwave{\hspace{5mm}}}
\put(15, 12){\tiny$\vdots\, \ell$}
\put(10, 10){\uwave{\hspace{5mm}}}

\qbezier(25, 55)(25, 55)(40, 70)
\qbezier(25, 70)(25, 70)(40, 55)

\put(25, 55){\line(0, -1){15}}
\put(40, 55){\line(0, -1){15}}

\put(25, 40){\line(0, -1){40}}
\put(40, 40){\line(0, -1){40}}
\put(45, 30){$+$}
\end{picture}
\begin{picture}(50, 70)(0,0)
{
\linethickness{1mm}
\put(10, 0){\line(0, 1){70}}
}

\put(10, 25){\uwave{\hspace{5mm}}}
\put(15, 12){\tiny$\vdots\, \ell$}
\put(10, 10){\uwave{\hspace{5mm}}}

\qbezier(25, 60)(33, 45)(40, 60)
\put(25, 60){\line(0, 1){10}}
\put(40, 60){\line(0, 1){10}}


\qbezier(25, 40)(33, 55)(40, 40)
\put(25, 40){\line(0, -1){40}}
\put(40, 40){\line(0, -1){40}}
\put(45, 30){$-$}
\end{picture}
\begin{picture}(50, 70)(0,0)
{
\linethickness{1mm}
\put(10, 0){\line(0, 1){70}}
}

\put(10, 65){\uwave{\hspace{5mm}}}
\put(15, 52){\tiny$\vdots\, \ell$}
\put(10, 50){\uwave{\hspace{5mm}}}

\qbezier(25, 40)(33, 25)(40, 40)
\put(25, 70){\line(0, -1){30}}
\put(40, 70){\line(0, -1){30}}


\qbezier(25, 20)(33, 35)(40, 20)
\put(25, 20){\line(0, -1){20}}
\put(40, 20){\line(0, -1){20}}
\put(45, 30){$=0$.}
\end{picture}
\caption{Iterated four-term relation}
\label{fig:4-term-iterated}
\end{figure}

Now pre-multiply (vertically) this relation 
 by $\Pi$ and use the skew symmetry of $\BH$, to obtain Figure \ref{fig:iterated-cap}.  
Similarly, by post-multiplying (vertically) the relation Figure \ref{fig:4-term-iterated} by $\amalg$ and using the skew symmetry of $\BH$, we obtain the relation depicted in Figure \ref{fig:iterated-cap-2}. 
\begin{figure}[h]
\begin{picture}(50, 70)(0,0)
{
\linethickness{1mm}
\put(10, 0){\line(0, 1){70}}
}

\qbezier(25, 65)(33, 80)(40, 65)

\put(10, 60){\uwave{\hspace{5mm}}}
\put(15, 47){\tiny$\vdots\, \ell$}
\put(10, 45){\uwave{\hspace{5mm}}}
\put(10, 25){\uwave{\hspace{5mm}}}

\put(25, 65){\line(0, -1){25}}
\put(40, 65){\line(0, -1){25}}

\qbezier(25, 40)(25, 40)(40, 30)
\qbezier(25, 30)(25, 30)(40, 40)

\put(25, 30){\line(0, -1){15}}
\put(40, 30){\line(0, -1){15}}
\qbezier(25, 15)(25, 15)(40, 0)
\qbezier(25, 0)(25, 0)(40, 15)
\put(45, 30){$+$}
\end{picture}
\begin{picture}(50, 70)(0,0)
{
\linethickness{1mm}
\put(10, 0){\line(0, 1){70}}
}

\qbezier(25, 60)(33, 75)(40, 60)

\put(10, 50){\uwave{\hspace{5mm}}}
\put(10, 25){\uwave{\hspace{5mm}}}
\put(15, 12){\tiny$\vdots\, \ell$}
\put(10, 10){\uwave{\hspace{5mm}}}


\put(25, 60){\line(0, -1){60}}
\put(40, 60){\line(0, -1){60}}

\put(25, 30){\line(0, -1){30}}
\put(40, 30){\line(0, -1){30}}
\put(45, 30){$+$}
\end{picture}
\begin{picture}(50, 70)(0,0)
{
\linethickness{1mm}
\put(10, 0){\line(0, 1){70}}
}

\qbezier(25, 60)(33, 75)(40, 60)

\put(10, 55){\uwave{\hspace{5mm}}}
\put(15, 38){\tiny$\vdots\, \ell$}
\put(10, 30){\uwave{\hspace{5mm}}}

\put(25, 60){\line(0, -1){30}}
\put(40, 60){\line(0, -1){30}}


\put(25, 30){\line(0, -1){15}}
\put(40, 30){\line(0, -1){15}}
\qbezier(25, 15)(25, 15)(40, 0)
\qbezier(25, 0)(25, 0)(40, 15)
\put(45, 30){$-$}
\end{picture}
\begin{picture}(90, 70)(-40,0)
\put(-30, 30){$(1-\delta)$}
{
\linethickness{1mm}
\put(10, 0){\line(0, 1){70}}
}

\qbezier(25, 60)(33, 75)(40, 60)

\put(10, 55){\uwave{\hspace{5mm}}}
\put(15, 38){\tiny$\vdots\, \ell$}
\put(10, 30){\uwave{\hspace{5mm}}}


\put(25, 60){\line(0, -1){60}}
\put(40, 60){\line(0, -1){60}}

\put(45, 30){$-$}
\end{picture}
\begin{picture}(50, 70)(0,0)
{
\linethickness{1mm}
\put(10, 0){\line(0, 1){70}}
}

\qbezier(25, 60)(33, 75)(40, 60)

\put(10, 60){\uwave{\hspace{5mm}}}
\put(15, 47){\tiny$\vdots\, \ell$}
\put(10, 45){\uwave{\hspace{5mm}}}

\qbezier(25, 40)(33, 25)(40, 40)
\put(25, 60){\line(0, -1){20}}
\put(40, 60){\line(0, -1){20}}


\qbezier(25, 20)(33, 35)(40, 20)
\put(25, 20){\line(0, -1){20}}
\put(40, 20){\line(0, -1){20}}
\put(45, 30){$=0$.}
\end{picture}
\caption{A relation in $\Hom_{\PB(\delta)}(2, 0)$}
\label{fig:iterated-cap}
\end{figure}  
\begin{figure}[h]
\begin{picture}(50, 70)(0,0)
\put(-5, 30){$-$}
{
\linethickness{1mm}
\put(10, 0){\line(0, 1){70}}
}
\put(10, 65){\uwave{\hspace{5mm}}}
\put(15, 52){\tiny$\vdots\, \ell$}
\put(10, 50){\uwave{\hspace{5mm}}}
\put(10, 25){\uwave{\hspace{5mm}}}

\put(25, 70){\line(0, -1){40}}
\put(40, 70){\line(0, -1){40}}


\put(25, 30){\line(0, -1){20}}
\put(40, 30){\line(0, -1){20}}

\qbezier(25, 10)(32,-5)(40, 10)

\put(45, 30){$-$}
\end{picture}
\begin{picture}(50, 70)(0,0)
{
\linethickness{1mm}
\put(10, 0){\line(0, 1){70}}
}

\put(10, 50){\uwave{\hspace{5mm}}}
\put(10, 28){\uwave{\hspace{5mm}}}
\put(15, 15){\tiny$\vdots\, \ell$}
\put(10, 13){\uwave{\hspace{5mm}}}

\qbezier(25, 55)(25, 55)(40, 70)
\qbezier(25, 70)(25, 70)(40, 55)

\put(25, 55){\line(0, -1){15}}
\put(40, 55){\line(0, -1){15}}

\put(25, 30){\line(0, -1){22}}
\put(40, 30){\line(0, -1){22}}
\qbezier(25, 30)(25, 30)(40, 40)
\qbezier(25, 40)(25, 40)(40, 30)

\qbezier(25, 8)(32,-7)(40, 8)

\put(45, 30){$+$}
\end{picture}
\begin{picture}(50, 70)(0,0)
{
\linethickness{1mm}
\put(10, 0){\line(0, 1){70}}
}

\put(10, 65){\uwave{\hspace{5mm}}}
\put(15, 50){\tiny$\vdots\, \ell$}
\put(10, 45){\uwave{\hspace{5mm}}}

\put(25, 70){\line(0, -1){40}}
\put(40, 70){\line(0, -1){40}}


\put(25, 30){\line(0, -1){20}}
\put(40, 30){\line(0, -1){20}}

\qbezier(25, 10)(32,-5)(40, 10)

\put(45, 30){$-$}
\end{picture}
\begin{picture}(50, 70)(0,0)
{
\linethickness{1mm}
\put(10, 0){\line(0, 1){70}}
}

\put(10, 35){\uwave{\hspace{5mm}}}
\put(15, 22){\tiny$\vdots\, \ell$}
\put(10, 20){\uwave{\hspace{5mm}}}

\qbezier(25, 55)(25, 55)(40, 70)
\qbezier(25, 70)(25, 70)(40, 55)

\put(25, 55){\line(0, -1){15}}
\put(40, 55){\line(0, -1){15}}

\put(25, 40){\line(0, -1){30}}
\put(40, 40){\line(0, -1){30}}

\qbezier(25, 10)(32,-5)(40, 10)

\put(45, 30){$+$}
\end{picture}
\begin{picture}(50, 70)(0,0)
{
\linethickness{1mm}
\put(10, 0){\line(0, 1){70}}
}

\put(10, 35){\uwave{\hspace{5mm}}}
\put(15, 22){\tiny$\vdots\, \ell$}
\put(10, 20){\uwave{\hspace{5mm}}}

\qbezier(25, 60)(33, 45)(40, 60)
\put(25, 60){\line(0, 1){10}}
\put(40, 60){\line(0, 1){10}}


\qbezier(25, 40)(33, 55)(40, 40)
\put(25, 40){\line(0, -1){30}}
\put(40, 40){\line(0, -1){30}}

\qbezier(25, 10)(32,-5)(40, 10)

\put(45, 30){$-$}
\end{picture}
\begin{picture}(50, 70)(0,0)
{
\linethickness{1mm}
\put(10, 0){\line(0, 1){70}}
}

\put(10, 65){\uwave{\hspace{5mm}}}
\put(15, 52){\tiny$\vdots\, \ell$}
\put(10, 50){\uwave{\hspace{5mm}}}

\qbezier(25, 40)(33, 25)(40, 40)
\put(25, 70){\line(0, -1){30}}
\put(40, 70){\line(0, -1){30}}


\qbezier(25, 20)(33, 35)(40, 20)
\put(25, 20){\line(0, -1){10}}
\put(40, 20){\line(0, -1){10}}

\qbezier(25, 10)(32,-5)(40, 10)

\put(45, 30){$=0$.}
\end{picture}
\caption{A relation in $\Hom_{\PB(\delta)}(0, 2)$}
\label{fig:iterated-cap-2}
\end{figure}  

To prove \eqref{eq:HT-H-1}, we post-multiply the relation Figure \ref{fig:iterated-cap} by $\X_0$, and then pull the bottom right end point to the top. We obtain 
\[
(\BH^{\ell+1})^T +  (\BH^\ell)^T \BH +  \BH^\ell - (1-\delta) (\BH^\ell)^T - Z_\ell\ot I =0, 
\]
which leads to equation \eqref{eq:HT-H-1}. 

Equation \eqref{eq:HT-H-2} can be proven similarly. 
We pre-multiply the relation by $\X_0$, and then pull the top right end point to the bottom. We obtain the relation 
\[
-(\BH^{\ell+1})^T - \BH (\BH^{\ell})^T + (\BH^{\ell})^T - \BH^{\ell}  + Z_\ell\ot I - \delta (\BH^{\ell})^T=0,
\]
which leads to equation \eqref{eq:HT-H-2}. 

Taking the difference between equations \eqref{eq:HT-H-1} and \eqref{eq:HT-H-2}, we see that $\BH$ commutes with all $(\BH^{\ell})^T$. This implies the final statement of the lemma. 
\end{proof}

\begin{remark} If the characteristic of $\K$ is not $2$, then 
\be\label{eq:quadratic}
(\BH^2)^T = \BH^2 +(\delta-2)\BH; 
\ee
this equation may be represented pictorially as
\[
\begin{picture}(60, 45)(0,0)
{
\linethickness{1mm}
\put(10, 0){\line(0, 1){45}}
}

\put(10, 30){\uwave{\hspace{5mm}}}
\put(10, 15){\uwave{\hspace{5mm}}}

\put(25, 35){\line(0, -1){30}}

\qbezier(25, 35)(35, 45)(40, 0)
\qbezier(25, 5)(35, -5)(40, 40)


\put(50, 18){$-$}
\end{picture}
\begin{picture}(85, 45)(0,0)
{
\linethickness{1mm}
\put(10, 0){\line(0, 1){45}}
}

\put(10, 30){\uwave{\hspace{5mm}}}
\put(10, 15){\uwave{\hspace{5mm}}}


\put(25, 45){\line(0, -1){45}}
\put(35, 18){$= \ (\delta-2)$}
\end{picture}
\begin{picture}(50, 45)(0,0)
{
\linethickness{1mm}
\put(10, 0){\line(0, 1){45}}
}

\put(10, 25){\uwave{\hspace{5mm}}}

\put(25, 45){\line(0, -1){45}}
\put(30, 5){.}
\end{picture}
\]
This follows easily from Figure \ref{fig:deg0-relation}.
\end{remark}

\begin{corollary} \label{cor:HT} 
Retain the notation of Lemma \ref{lem:HT2} .  
The following relations hold in $\End_{\MPB(\delta)}((m, v))$ for all $\ell\ge 0$: 
\beq
(\BH^{\ell+1})^T =\sum_{i=1}^\ell \BG_i \Phi ^{\ell -i}  - \BH \Phi^\ell, \quad
(\BH^{\ell+1})^T =\sum_{i=1}^\ell  \Phi ^{\ell -i} \BG_i - \Phi^\ell  \BH, 
\eeq
which in particular imply 
\beq\label{eq:H-Z-comm}
\sum_{i=1}^{\ell-1}  [Z_i\ot I, \Phi^{\ell-i}]=0, \quad \forall \ell. 
\eeq
\end{corollary}
\begin{proof}
By iterating equation \eqref{eq:HT-H-1},  we obtain 
\[
\baln
(\BH^{\ell+1})^T 
&= \sum_{i=1}^\ell \BG_i \Phi ^{\ell -i} +\BH^T \Phi^\ell=\sum_{i=1}^\ell \BG_i \Phi ^{\ell -i}  - \BH \Phi^\ell, 
\ealn
\]
where the last equality is obtained by using the skew symmetry of $\BH$. This proves the first relation. 
The second relation is a consequence of equation \eqref{eq:HT-H-2}. 
By taking the difference between the two relations, we obtain \eqref{eq:H-Z-comm}.
This completes the proof.
\end{proof}

\begin{theorem} \label{thm:Z-odd} 

Assume that $2\ne 0$ in $\K$  (thus $Z_1=0$). 
Then 
\begin{enumerate}
\item The elements $Z_\ell$ are central in $\MPB(\delta)$ for $\ell\geq 1$, in the sense that for any objects $\bs, \bs', \bt, \bt'$,
and all $\A\in\Hom_{\MPB(\delta)}((\bt, m, \bt'),  (\bs, m, \bs'))$,  
\[
\baln
(I(\bt)\ot Z_\ell\ot I(\bt')) \A= \A (I(\bs)\ot Z_\ell\ot I(\bs'));
\ealn
\]
\item The elements $Z_{2j+1}$ with $j\ge 1$ belong to the subalgebra generated by $\I$ and $Z_{2\ell}$ with $\ell\ge 1$.
\end{enumerate}
\end{theorem}
\begin{proof}
To prove part (1), we need only show that $[\BH, Z_\ell\ot I]=0$ for all $\ell$. We will do this by induction on $\ell$  using \eqref{eq:H-Z-comm}. 
Let us first replace $\ell$ by $\ell+1$ and re-write \eqref{eq:H-Z-comm} as
\[
\sum_{i=1}^{\ell}  [\Phi^{\ell+1-i},  Z_i\ot I]=0, \quad \forall \ell.
\]
Recall that $Z_1=0$, and we have already shown that $Z_2$ (and hence $Z_3$) commutes with $\BH$. Thus they also commute with $\Phi= (1-\delta)\I \ot I-\BH$. 
Assume that $Z_i$ commutes with $\Phi$ for all $i<\ell$. Then the above relation gives
$[\Phi,  Z_\ell\ot I]=0$, and hence $[\BH,  Z_\ell\ot I]=0$. 

Now we prove part (2). 
In view of the expressions for $\Phi$ and $\BG_i$ in terms of $\BH$ and $Z_j$ from Lemma \ref{lem:HT2}, we can re-write  the first relation in Corollary \ref{cor:HT} as 
\beq\label{eq:diff-HT}
(\BH^{\ell+1})^T  = (1)^{\ell+1} \BH^{\ell +1}+\sum_{i=0}^\ell f_i(Z_2, Z_3, \dots, Z_\ell) \BH^i, 
\eeq
where the $f_i(Z_2, Z_3, \dots, Z_\ell)$ are linear combinations of $1, Z_2, Z_3, \dots, Z_\ell$.
Taking the tensor product of both sides of \eqref{eq:diff-HT} with $I$, we obtain a relation in $\Hom_{\PB(\delta)}(3, 3)$. Pre-multiplying the resulting relation by $\Pi$ and then post-multiplying by $\amalg$, and using $\Pi Z_{\ell+1}\left((\BH^{\ell+1})^T \ot I\right) \amalg = Z_{\ell+1}$,  we obtain 
\[
\left(1-(1)^{\ell+1}\right) Z_{\ell+1}= \sum_{i=0}^\ell f_i(Z_2, Z_3, \dots, Z_\ell) Z_i.
\]

Let $\ell=2j$, then this relation allows us to express $Z_{2j+1}$ in terms of $Z_k$ for $k\le 2j$. An induction on $j$, starting from $j=0$ with $Z_1=0$,  completes the proof of 
the theorem. 
%
\end{proof}

Note that by slightly extending the proof of Theorem \ref{thm:Z-odd}(2),  we also obtain the relation $
\sum_{i=0}^{2\ell+1} f_i(Z_2, Z_3, \dots, Z_{2\ell +1}) Z_i=0$ for all $\ell$. 

\;

\;

\begin{lemma}   \label{lem:tech} 
The elements 
$
Z(k_1, k_2, \dots, k_p)=\Pi \BH^{k_1} \X_0\BH^{k_2} \X_0\dots  \X_0 \BH^{k_p}\amalg$,  
for $p\ge 1$ and $k_i\ge 1$, 
can be expressed, using composition, in terms of $\I$ and $Z_{\ell}$ for $\ell\ge 2$. 
\end{lemma}
\begin{proof} 
Denote $\wh{\bf H}(r, s)=\Hom_{\MPB(\delta)}((m, v^r),  (m, v^s))$. There is a filtration 
\beq
F\wh{\bf H}(r, s)_0\subset F\wh{\bf H}(r, s)_1\subset F\wh{\bf H}(r, s)_2\subset\dots,
\eeq
where $F\wh{\bf H}(r, s)_k$ is spanned by multi-polar diagrams with no more than $k$ connectors. Each $\wh{\bf H}(r, r)$ is a filtered algebra with this filtration.

Now $Z(k_1, k_2, \dots, k_p)$  belongs to  $F\wh{\bf H}(0, 0)_k$ with $k=\sum_i k_i$. We prove the lemma by induction on $k$. If $k=1$, the statement is clear. 
Note that $F\wh{\bf H}(0, 0)_k$ is spanned by diagrams of the type we are considering, viz. $\{Z(k_1,k_2, \dots ,k_p)\mid \sum_ik_i\leq k\}$.

If $k_1=1$, it follows the skew symmetry of $\BH$ that $Z(k_1, k_2, \dots, k_p)=-Z(k_1+k_2, k_3, \dots, k_p)$, and there is a similar relation if $k_p=1$. Thus we may assume that $k_1, k_p\ge 2$. 

We now describe two reductions, which apply respectively to the cases $k_2=1$ or $k_2>1$.
\begin{enumerate}
\item If $k_2=1$, by repeated application of the four-term relation modulo $F\wh{\bf H}(0, 0)_{k-1}$, we obtain 
\[
Z(k_1, k_2, \dots, k_p)=-Z(k_1+k_2+ k_3, k_4, \dots, k_p) \mod F\wh{\bf H}(0, 0)_{k-1}.
\] 
\item If $k_2\ge 2$, we write $Z(k_1, k_2, \dots, k_p) = \Pi \BH^{k_1} \X_0\BH\X_0 \X_0\BH^{k_2-1} \X_0\dots  \X_0 \BH^{k_p}\amalg$. Then applying the reduction (1), we obtain 
\[
\begin{aligned}
Z(k_1, k_2, \dots, k_p) &= - \Pi \BH^{k_1+1} \X_0\BH^{k_2-1} \X_0\dots  \X_0 \BH^{k_p}\amalg \mod F\wh{\bf H}(0, 0)_{k-1}\\
&=- Z(k_1+1, k_2-1, \dots, k_p) \mod F\wh{\bf H}(0, 0)_{k-1}. 
\end{aligned}
\]
If $k_2-1\ge 2$, repeat this process, eventually reaching 
\[
\begin{aligned}
Z(k_1, k_2, \dots, k_p) 
&=(-1)^{k_2-1} Z(k_1+k_2-1, 1, \dots, k_p) \mod F\wh{\bf H}(0, 0)_{k-1}, 
\end{aligned}
\]
and we are in the situation of (1). 
\end{enumerate}
The two reductions above enable us to reduce $Z(k_1, k_2, \dots, k_p)$ to $Z(k-1, 1)$ or $-Z(k-1, k) $ modulo $F\wh{\bf H}(0, 0)_{k-1}$. These in turn are equal to $\pm Z_k\mod F\wh{\bf H}(0, 0)_{k-1}$
(for some $k$) by the skew symmetry of $\BH$. Using induction on $k$,
we can express $Z(k_1, k_2, \dots, k_p)$ as a polynomial in $Z_\ell$ for $\ell\le k$ with the coefficient of $Z_k$ being $\pm 1$. 

This completes the proof of the lemma. 
\end{proof}

\begin{theorem}\label{thm:Hom01}
\begin{enumerate}
\item
The endomorphism algebra $\End_{\MPB(\delta)}(m)$ is commutative, and is generated by the elements $\Z_{2\ell}$ for $\ell\ge 1$. 
\item All elements $Z\in \End_{\MPB(\delta)}(m)$ are central in $\UPB(\delta)$ in the sense that for any pair of objects $(v^r, m, v^s),  (v^{r'}, m, v^{s'})\in\UPB(\delta)$,
and morphism $\A\in\Hom_{\UPB(\delta)}((v^r, m, v^s),  (v^{r'}, m, v^{s'}))$,  
\[
\baln
(I(v^{r'})\ot Z\ot I(v^{s'})) \A= \A (I(v^{r})\ot Z\ot I(v^{s'})).
\ealn
\]
\item The endomorphism algebra $\End_{\MPB(\delta)}(\bs)$ for $\bs=(m, v)$ (resp. $\bs =(v, m)$) is generated by $\BH(m, v)$ (resp. $\BH(v, m)$) and the elements $Z_{2\ell}I(\bs)$ for all $\ell$, 
and thus  is commutative. 
\item The elements $Z_{2\ell}$ for $\ell\ge 1$ are algebraically independent. 
\end{enumerate}
\end{theorem}
\begin{proof}
We adopt notation from the proof of Lemma \ref{lem:tech}. In particular, $\wh{\bf H}(0, 0)=\End_{\MPB(\delta)}(m)$ and $\wh{\bf H}(1, 1)=\End_{\MPB(\delta)}((m, v))$.

Denote by $Z(0, 0)$ the subalgebra of $\wh{\bf H}(0, 0)$ generated by the elements $\Z_{2\ell}$ for all $\ell$, which are central by Theorem \ref{thm:Z-odd}. 
Then all $Z_\ell\in Z(0, 0)$ by Theorem \ref{thm:Z-odd}(2). By the above lemma, 
$Z(k_1, k_2, \dots, k_p)\in Z(0, 0)$, thus is central for all $(k_1,\dots,k_p)$.

Using this fact, it follows
that $\wh{\bf H}(1, 1)$ is generated by the elements $(\BH^\ell)^T$, $\BH^k$,  
and $Z(k_1, k_2, \dots, k_p)\ot I$ for all $k_1, \dots, k_p$ and  $k, \ell, p$. 
Corollary \ref{cor:HT} states that the elements $(\BH^\ell)^T$ can all be expressed in terms of the elements 
$\BH^k$  and $Z(k_1, k_2, \dots, k_p)\ot I$. This implies part (3). 

As the elements 
$\Pi (\BD\ot I)\amalg$ for $\BD\in\wh{\bf H}(1, 1)$ span $\wh{\bf H}(0, 0)$, we have  
$\wh{\bf H}(0, 0)=Z(0, 0)$ by part (3). Now part (2) is clear. 

We will give a proof of part (4) at the end of Section \ref{sect:centre-construct} by using 
a special case of the functor constructed in Theorem \ref{thm:a-funct}.
\end{proof}
\begin{remark}
One can also deduce from Corollary \cite[XX.3.2. (a)]{K} the commutativity of $\End_{\MPB(\delta)}(m)$. We mention that there is an analysis of the endomorphism algebra of the object $0$ of the category $\AB(\delta)$ in \cite{RSo}, 
which corresponds to the endomorphism algebra in $PB(\delta)$ of 
the object $m$,  by the analysis in Section \ref{sect:PB-AB}. Similar analysis in the case of the $\rm{Spin}$ and $\rm{Pin}$ groups may be found in \cite{McS}.
\end{remark}

\subsection{Endomorphism algebras of the polar Brauer category}

When dealing with the polar Brauer category $\PB(\delta)$, we shall adopt the following convention. 

\begin{remark}\label{rmk:no-m}
Denote the object $(m, v^r)$ of $\PB(\delta)$ by $r$ for any $r\in\N$. 
Thus we will write $\Hom_{\PB(\delta)}((m, v^r), (m, v^s))$ and $\End_{\PB(\delta)}((m, v^r)))$ respectively as $\Hom_{\PB(\delta)}(r, s)$ and $\End_{\PB(\delta)}(r)$ for any $r, s\in\N$. 
\end{remark}

\subsubsection{Endomorphism algebras of the polar Brauer category}

We consider the endomorphism algebra $\End_{\PB(\delta)}(r)$. It contains the subalgebra generated by the elements $Z_\ell\ot I_r$,  where $Z_\ell$ are defined in Lemma  \ref{lem:central}.  

Recall that the usual Brauer category $\CB(\delta)$ is a full subcategory of $\MPB(\delta)$ 
with objects $v^r$ for $r\in\N$. We have a canonical embedding $\iota$ of
the $\K$-spaces of morphisms of $\CB(\delta)$ in that of $\PB(\delta)$,  
which takes a Brauer diagram $A$ to its image $\iota(A)=\A_0=\I\ot A$. 
Let $s_i, e_i$ for $i=1, 2, \dots, r-1$ be the standard generators of the Brauer algebra $B_r(\delta)=\End_{\MPB(\delta)}((v^r))$ given in Figure \ref{fig:siei}. 
\begin{figure}[h]
\setlength{\unitlength}{0.3mm}
\begin{picture}(150, 70)(-20,-5)
\put(-20, 28){$s_i\  =$}
\put(20, 0){\line(0, 1){60}}
\put(25, 30){...}
\put(40, 0){\line(0, 1){60}}

\qbezier(60, 0)(60, 0)(80, 60)
\qbezier(80, 0)(80, 0)(60, 60)

\put(100, 0){\line(0, 1){60}}
\put(120, 0){\line(0, 1){60}}
\put(105, 30){...}
\put(56, -10){\small$i$}
\put(72, -10){\small{$i$+1}}
\put(130, 0){, }
\end{picture} 
\hspace{.5cm}
\begin{picture}(150, 70)(-20,-5)
\put(-20, 28){$e_i\  =$}
\put(20, 0){\line(0, 1){60}}
\put(25, 30){...}
\put(40, 0){\line(0, 1){60}}

\qbezier(60, 60)(70, 10)(80, 60)
\qbezier(60, 0)(70, 50)(80, 0)

\put(100, 0){\line(0, 1){60}}
\put(120, 0){\line(0, 1){60}}
\put(105, 30){...}
\put(56, -10){\small$i$}
\put(72, -10){\small{$i$+1}}
\put(130, 0){. }
\end{picture}
\caption{Generators of $B_r(\delta)$}
\label{fig:siei}
\end{figure}
Define 
$
H_i = s_i - e_i,$ fior $i=1,2, \dots, r-1,  
$
which are represented graphically by Figure \ref{fig:Hi},  
\begin{figure}[h]
\begin{picture}(150, 60)(-20,0)
\put(-20, 28){$H_i\  =$}
\put(20, 0){\line(0, 1){60}}
\put(25, 30){...}
\put(125,30){,}
\put(40, 0){\line(0, 1){60}}

\put(60, 0){\line(0, 1){60}}
\put(60, 30){\line(1, 0){20}}
\put(80, 0){\line(0, 1){60}}


\put(100, 0){\line(0, 1){60}}
\put(120, 0){\line(0, 1){60}}
\put(105, 30){...}
\put(56, -10){\small$i$}
\put(72, -10){\small{$i$+1}}
\end{picture} 
\caption{Picture for $H_i=s_i-e_i$}
\label{fig:Hi}
\end{figure}
%
%
where the right hand side is the linear combination of two Brauer diagrams in Figure \ref{fig:siei}.

We have the elements $S_i=\I\ot s_i$, 
$E_i=\I\ot e_i$ and $\I\ot H_i$ in $\End_{\PB(\delta)}(r)$. 
Let $\BH_{i j}(r) = \BH_{i+1, j+1}(m, v^r)$ for $0\le i<j\le r$, where $\BH_{i+1, j+1}(m, v^r)$ are defined by Figure \ref{fig:Hijr}.   
Denote by $PB_r(\delta)$ the $\K$-subalgebra of $\End_{\PB(\delta)}(r)$ 
generated by 
\[
\{ \BH_{0 j}(r)\mid   j=1, 2, \dots, r\}\bigcup \{S_k, \  E_k  \mid k=1, 2, \dots, r-1  \}, 
\] 
where we note that the elements $Z_\ell\ot I_r$ are not in $PB_r(\delta)$. 
We have the following generalisation of Theorem 
\ref{thm:TtoC}.

\begin{theorem}
Recall the definition of the infinitesimal braid algebra $T_r$ given in \S\ref{ss:tr}. The  map 
$
t_{i j} \mapsto \BH _{i-1, j-1}(r)$, for all $1\le i<j\le r+1, 
$
extends  to a unique algebra homomorphism 
$
\Phi_r: T_{r+1}\lra PB_r(\delta).
$
This homomorphism  is surjective if $\delta- 2$ is a unit in $\K$. 
\end{theorem}
\begin{proof}  
The first statement follows from Lemma \ref{lem:Murphy}(1), and the second statement is easily deduced from Lemma \ref{lem:T-H-B}. 
\end{proof}

To better understand the structure of $PB_r(\delta)$, we define the elements 
\beq\label{eq:vartheta}
\vartheta_j(r)= \sum_{0\le a< j} \BH_{a j}(r), \quad j=1, 2, \dots, r, 
\eeq
and note that they have the following properties.
\begin{theorem} \label{lem:JM}\label{thm:JM}
The following relations hold in $PB_r(\delta)$ for all valid indices $i, j, k$:
\beq
&& E_1(\vartheta_1(r))^\ell E_1= Z_\ell E_1,  \label{eq:JM-1}\\
 &&\vartheta_i(r) \vartheta_j(r)=  \vartheta_j(r) \vartheta_i(r), \  \text{ for all  }\    i, j,  \label{eq:JM-2}\\
 &&S_k  \vartheta_j(r) -  \vartheta_j(r) S_k  =0, \ \text{ if $j\not\in\{k, k+1\}$}, \label{eq:JM-3}\\
 &&E_k  \vartheta_j(r) -  \vartheta_j(r) E_k  =0, \ \text{ if $j\not\in\{k, k+1\}$},  \label{eq:JM-4}\\
 && S_k  \vartheta_{k}(r)- \vartheta_{k+1}  S_k  = E_k - \I_r,  \label{eq:JM-5}\\
 && \vartheta_k(r) S_k  -S_k  \vartheta_{k+1}(r) = E_k - \I_r,  \label{eq:JM-6}\\
 &&E_k \left(\vartheta_k(r) + \vartheta_{k+1}(r) \right) = (1-\delta) E_k , \label{eq:JM-7}\\
&&\left(\vartheta_k(r) + \vartheta_{k+1}(r) \right) E_k = (1-\delta)E_k, \label{eq:JM-8}
  \eeq
  where $\I_r=I(m, v^r)$. 
\end{theorem}
\begin{proof}
The relations \eqref{eq:JM-1}, \eqref{eq:JM-3} and \eqref{eq:JM-4} are quite clear; the relations \eqref{eq:JM-5} and \eqref{eq:JM-6} follow from the obvious fact that 
\[
S_k  \vartheta_{k} S_k  = \vartheta_{k+1} - \BH_{k, k+1}(r); 
\]
and \eqref{eq:JM-2}follows from Lemma \ref{lem:Murphy}(2).

The relations \eqref{eq:JM-7} and \eqref{eq:JM-8}  can be derived from 
the skew symmetry of $\BH$ and of $H$, as well as the relations
\[
\cap\circ H = (1-\delta) \cap, \quad H\circ \cup = (1-\delta) \cup.
\]
To prove them, we define the  elements $\Pi_i(r)= \I\ot I^{\ot (i-1)}\ot\cap\ot I^{\ot (r-i-1)}$ and 
$\amalg_i(r)= \I\ot I^{\ot (i-1)}\ot\cup\ot I^{\ot (r-i-1)}$ for $i=1, 2, \dots, r-1$.
Note that $\Pi_i(r) \in \Hom_{\PB(\delta)}(r, r-2)$ and 
$\amalg_i(r) \in \Hom_{\PB(\delta)}(r-2, r)$. 
Then \eqref{eq:JM-7} and \eqref{eq:JM-8} are implied by the following relations for any $i<r$:
\beq
&\Pi_i(r)\left(\vartheta_i(r) + \vartheta_{i+1}(r) \right) = (1-\delta) \Pi_i(r), \label{eq:cap-theta}\\
&\left(\vartheta_i(r) + \vartheta_{i+1}(r) \right) \amalg_i(r) = (1-\delta)\amalg_i(r). \label{eq:cup-theta}
\eeq
We  verify the $i=1$ case of \eqref{eq:cap-theta} by the following computation:
\[
\baln
\Pi_1(r) \left(\vartheta_1(r) +\vartheta_2(r)\right) =\Pi_1(r) \BH_{01}(r)+ \Pi_1(r) \BH_{02}(r) + \I \ot \cap H\ot I^{\ot{r-2}},
\ealn
\]
where the first two terms on the right side cancel because of skew symmetry of $\BH$ and $H$, 
and the third term is equal to $(1-\delta) \Pi_1(r)$.  An easy induction on $i$ proves the general case.  
The relation \eqref{eq:cup-theta}  can be proved in the same way. 

This completes the proof of the theorem.
\end{proof}

Recall the Nazarov--Wenzl algebra of degree $r$ as presented in \cite[Definition 2.1]{ES}  with generators and relations $(VW.1)$ -- $(VW.8)$, where $(VW.3)$ and $(VW.4)$ involve an infinite family of  arbitrary parameters $w_k$ for all non-negative integers $k$.  It emerges in our context
as a quotient algebra of $PB_r(\delta)$ by identifying certain linear combinations of $Z_\ell$ with the scalars $w_k$.

\;
\;

\;

\;
%
\subsubsection{Relationship to the affine Brauer category of \cite{RSo}}
\label{sect:PB-AB}
%
%
Our category $\PB(\delta)$ is not a monoidal category, but rather a module category over the usual Brauer category $\CB(\delta)$ which is a monoidal category.  That is, we have the bi-functor 
$\ot: \PB(\delta) \times \CB(\delta)\lra  \PB(\delta)$. 
Some categories of dotted Brauer diagrams were introduced in \cite{RSo, Betal}, which are monoidal categories in contrast, but are very closely related to ours nonetheless. 

The work \cite{Betal} deals with a super category, which is essentially the same as the affine Brauer category $\AB(\delta)$ of \cite{RSo},  but with various signs arising through a $\Z_2$-grading. 
For the sake of precision, let us compare our category with the latter.  The  category $\AB(\delta)$ has objects $r\in\N$. 
As a monoidal category,  its morphisms are generated by the standard generators \cite{LZ12} of $\CB(\delta)$ together with the dotted diagram
$\begin{picture}(10, 10)(-5,0)
\put(0, 0){\line(0, 1){10}}
\put(-2, 3){\tiny$\bullet$}
\end{picture}.$
However, when $\AB(\delta)$ is only considered as a (right) module category of the Brauer category $\CB(\delta)$,  one also needs the generators 
$
\begin{picture}(80, 20)(-30,5)
\put(-28, 5) {$Y_{\ell+1} =$} 
\put(10, 0){\line(0, 1){20}}
\put(17, 12){$\dots$}
\put(20, -2){$\ell$}
\put(35, 0){\line(0, 1){20}}
\put(45, 0){\line(0, 1){20}}
\put(43, 8){\tiny$\bullet$}
\end{picture}
$
for $\ell =0, 1, \dots.$

\medskip

The following result is a straightforward consequence of Theorem \ref{thm:JM}
\begin{theorem}\label{thm:RSo} Retain the notation of Theorem \ref{thm:JM}. 
There exists an isomorphism of categories $\mathfrak{H}: \PB(\delta)\lra\AB(\delta)$, which preserves $\CB(\delta)$-module category structures.
 The functor $\mathfrak{H}$ sends any object $(m, v^r)\in \PB(\delta)$ to $r\in\AB(\delta)$, and   
\beq\label{eq:iso-cat}
&\vartheta_\ell(\ell)+\frac{1-\delta}2 \I_\ell\mapsto Y_\ell, \quad \forall \ell =1, 2, \dots, \\
&\I \ot B\mapsto B, \quad \text{for any morphism $B$ of $\CB(\delta)$}. 
\eeq
\end{theorem}
The functor $\mathfrak{H}$ and its inverse map the defining relations of $\PB(\delta)$ and those of $\AB(\delta)$ to one another.   For example,  relation \cite[(1.17)]{RSo} is a consequence of equation \eqref{eq:JM-7} via $\mathfrak{H}$, and relation \cite[(1.18)]{RSo} follows from the skew symmetry of $\BH$, or more explicitly, from the relation
$\prod_{23} \Theta_2(m, v^3) \amalg_{12}=\prod_{12} \Theta_2(m, v^3) \amalg_{23}=-\Theta_1(m, v)$,
where 
 $
\begin{picture}(60, 10)(-15,0)
\put(-15, 2){$\amalg_{12}=$}
{
\linethickness{.7mm}
\put(20, 0){\line(0, 1){10}}
}
\qbezier(25, 10)(30, -5)(35, 10)
\put(40, 0){\line(0, 1){10}}
\end{picture}, 
$
$
\begin{picture}(60, 10)(-15,0)
\put(-15, 2){$\prod_{12}=$}
{
\linethickness{.7mm}
\put(20, 0){\line(0, 1){10}}
}
\qbezier(25, 0)(30, 15)(35, 0)
\put(40, 0){\line(0, 1){10}}
\end{picture}, 
$
$
\begin{picture}(60, 10)(-15,0)
\put(-15, 2){$\amalg_{23}=$}
{
\linethickness{.7mm}
\put(20, 0){\line(0, 1){10}}
}
\put(25, 0){\line(0, 1){10}}
\qbezier(30, 10)(35, -5)(40, 10)
\end{picture}, 
$
and 
$
\begin{picture}(60, 10)(-15,0)
\put(-15, 2){$\prod_{23}=$}
{
\linethickness{.7mm}
\put(20, 0){\line(0, 1){10}}
}
\put(25, 0){\line(0, 1){10}}
\qbezier(30, 0)(35, 15)(40, 0)
\end{picture}. 
$

\smallskip

Note that the morphisms 
$\begin{picture}(20, 10)(-5,0)
\put(0, 0){\line(0, 1){10}}
\put(-2, 3){\tiny$\bullet$}
\put(10, 0){\line(0, 1){10}}
\end{picture}$ and 
$\begin{picture}(20, 10)(-5,0)
\put(0, 0){\line(0, 1){10}}
\put(10, 0){\line(0, 1){10}}
\put(8, 3){\tiny$\bullet$}
\end{picture}$ 
commute in $\AB(\delta)$,  as their compositions in different orders were both taken to be equal to the tensor product of $\begin{picture}(10, 10)(-5,0)
\put(0, 0){\line(0, 1){10}}
\put(-2, 3){\tiny$\bullet$}
\end{picture}$ with itself, that is,  
$\begin{picture}(20, 10)(-5,0)
\put(0, 0){\line(0, 1){10}}
\put(-2, 3){\tiny$\bullet$}
\put(10, 0){\line(0, 1){10}}
\put(8, 3){\tiny$\bullet$}
\end{picture}$.
This commutation relation in $\AB(\delta)$ corresponds to the four-term relation in $\PB(\delta)$ via $\mathfrak{H}$. 

However,  when $\AB(\delta)$ was applied in \cite{RSo} to  study the category $\fg$-Mod of $\fg$-modules for the orthogonal or symplectic Lie algebra $\fg$, there was an extra layer of complication in that it was necessary to go through the endofunctor category $\mathcal{E}nd(\text{$\fg$-Mod})$ instead of working with $\fg$-Mod directly. This will be explained in Remark \ref{rmk:compare}.
Similar remarks apply to the analysis in \cite{McS}.

\subsection{The multi-polar Temperley-Lieb category}\label{sect:TL}
We now introduce a quotient category of $\MPB(\delta)$, which may be regarded as a multi-polar analogue of the Temperley-Lieb category (cf. \cite{GL03, ILZ}).

The required category will be a quotient $\MPB(\delta)/\CJ$ for some tensor ideal $\CJ$ generated by a morphism $\Theta\in \End_{\MPB(\delta)}(v^2)$.
Any such $\Theta$ is of the form $\Theta=a X_{v v} + b I(v^2) + c E$ with $E=\cup^v\circ \cap_v$,  for some $a, b, c\in\K$ with $a\ne 0$.
%
%
To obtain an interesting quotient category, we stipulate that the images of the quasi idempotents $I(v^2)- X_{v v}$ and $E$ 
(which are orthogonal to each other) are non-zero in the quotient.
Now 
\[
\baln
&(I(v^2)- X_{v v})\Theta=(b-a) (I(v^2)- X_{v v})=0, \\
&E\Theta=(a+b+\delta c) E=0 
\ealn
\] 
in the quotient category, from which we conclude that
$
a-b=0, \text{ and } a+b+\delta c=0. 
$
These relations imply that $b=a$ and 
$
2 a + \delta c=0.  
$
Hence 
$\delta$ is a unit, and $c= - \frac{2a}{\delta}$.
%
%
Choose $a=1$, so that the morphism is given by $\Theta= X_{v v} + I(v^2) - \frac{2}{\delta} E$, 
which is depicted graphically in Figure \ref{fig:TL}. Note that $\Theta$ is also a quasi idempotent.

\begin{figure}[h]
\begin{picture}(150, 35)(0,0)
\put(-35, 15){$\Theta\ = \ $}

\qbezier(0, 0)(0, 0)(15, 35)
\qbezier(0, 35)(0, 35)(15, 0)

\put(35, 15){$+ $}

\put(65, 0){\line(0, 1){35}}
\put(75, 0){\line(0, 1){35}}
\put(90, 15){$-\ \frac{2}{\delta}$}

\qbezier(120, 35)(130, 10)(140, 35)
\qbezier(120, 0)(130, 25)(140, 0)

\end{picture} 
\caption{Temperley-Lieb condition}
\label{fig:TL}
\end{figure}

\begin{definition} \label{def:TL}
Assume that $0\ne \delta\in\K$. Let $\langle \Theta\rangle$ be the tensor ideal generated by the element $\Theta$.  The quotient category $\MTL(\delta)=\MPB(\delta)/\langle \Theta\rangle$
will be referred to as the {\em multi-polar Temperley-Lieb category}.
\end{definition}

As noted above, the relation $\Theta=0$ implies that the vector space  of morphisms of $\MTL(\delta)$ is spanned by diagrams analogous to multi-polar 
diagrams without any crossings of strings all labelled by $v$. Such diagrams will be called {\em multi-polar Temperley-Lieb diagrams}.  

Henceforth we assume that $0\ne \delta\in\K$.

Note that the composition of diagrams now obeys new rules. In particular, we have the following result. 

\begin{lemma} \label{lem:TL-H} The four-term relation in $\End((m, v^2))$ reduces to the following relation in $\End_{\MTL(\delta)}((m, v))$
\beq
\BH^2 +\frac{\delta-2}{2} \BH - \frac{1}{\delta} Z_2\ot I =0.
\eeq
This can be graphically represented by 
\beq\label{eq:TL-H}
\begin{array}{c}
\begin{picture}(100, 40)(0, 0)
{
\linethickness{1mm}
\put(15, 0){\line(0, 1){40}}
}
\put(15, 28){\uwave{\hspace{7mm}}}
\put(15, 12){\uwave{\hspace{7mm}}}

\put(35, 0){\line(0, 1){40}}

\put(45, 16){ $+ \ \frac{\delta-2}2$ }
\end{picture} 
\begin{picture}(100, 40)(25, 0)
{
\linethickness{1mm}
\put(15, 0){\line(0, 1){40}}
}
\put(15, 20){\uwave{\hspace{7mm}}}

\put(35, 0){\line(0, 1){40}}

\put(45, 16){ $ - \ \frac{Z_2}\delta$ }
\end{picture} 
\begin{picture}(60, 40)(55, 0)
{
\linethickness{1mm}
\put(15, 0){\line(0, 1){40}}
}

\put(35, 0){\line(0, 1){40}}

\put(45, 16){ $=\ 0$ }
\put(75, 5){.}
\end{picture} 
\end{array}
\eeq
\end{lemma}
\begin{proof}
Eliminating the crossings of thin arcs in all the diagrams in the four-term relation given in Figure \ref{fig:4-term} by using the relation $\Theta=0$, we obtain the relation 
\beq \label{eq:TL-H-1}
\begin{array}{c}
\begin{picture}(75, 40)(30, 0)
{
\linethickness{1mm}
\put(35, 0){\line(0, 1){40}}
}
\put(35, 35){\uwave{\hspace{5mm}}}
\put(35, 25){\uwave{\hspace{5mm}}}

\put(50, 40){\line(0, -1){25}}
\put(65, 40){\line(0, -1){25}}

\qbezier(50, 15)(57, 5)(65, 15)

\qbezier(50, 0)(57, 10)(65, 0)

\put(80, 15){$-$}
\end{picture}
\begin{picture}(85, 40)(30, 0)
{
\linethickness{1mm}
\put(35, 0){\line(0, 1){40}}
}
\put(35, 18){\uwave{\hspace{5mm}}}
\put(35, 8){\uwave{\hspace{5mm}}}

\put(50, 0){\line(0, 1){25}}
\put(65, 0){\line(0, 1){25}}

\qbezier(50, 25)(57, 35)(65, 25)

\qbezier(50, 40)(57, 30)(65, 40)



\put(80, 15){$-$}
\end{picture}
\begin{picture}(85, 40)(30, 0)
\put(12, 15){$\frac{2-\delta}{2}$}
{
\linethickness{1mm}
\put(35, 0){\line(0, 1){40}}
}
\put(35, 30){\uwave{\hspace{5mm}}}

\put(50, 40){\line(0, -1){25}}
\put(65, 40){\line(0, -1){25}}

\qbezier(50, 15)(57, 5)(65, 15)

\qbezier(50, 0)(57, 10)(65, 0)



\put(80, 15){$+$}
\end{picture}
\begin{picture}(80, 40)(30, 0)
\put(12, 15){$\frac{2-\delta}{2}$}
{
\linethickness{1mm}
\put(35, 0){\line(0, 1){40}}
}
\put(35, 15){\uwave{\hspace{5mm}}}

\put(50, 0){\line(0, 1){25}}
\put(65, 0){\line(0, 1){25}}

\qbezier(50, 25)(57, 35)(65, 25)

\qbezier(50, 40)(57, 30)(65, 40)

\put(80, 15){$=\ 0$}
\put(105, 10){.}
\end{picture}
\end{array}
\eeq

Post-multiplying \eqref{eq:TL-H-1} by $\cup$, we obtain a relation involving three $(0, 2)$-diagrams. 
Bringing the right end point of the thin arc in each of the diagram down to the bottom, we arrive at \eqref{eq:TL-H}.  Finally, given \eqref{eq:TL-H}, 
the relation \eqref{eq:TL-H-1} holds identically in $\MTL(\delta)$.
\end{proof}

\begin{remark}
Note that the above shows that if $\delta-2=0$, then $\BH^2=\frac{1}{\delta}Z_2\ot I$, and many of our results become degenerate. We therefore
assume henceforth that $\delta\neq 0, 2$.
\end{remark}

The following results are immediate from Lemma \ref{lem:TL-H}  and Lemma \ref{lem:central}. 
\begin{corollary} \label{lem:TL-4-t} The multi-polar Temperley-Lieb category has the following properties.
\begin{enumerate}[(i)]
\item $(Z_2\ot I) \BH = \BH (Z_2\ot I)$, and hence  ${\mathbb D}(Z_2\ot I_r)= (Z_2\ot I_s){\mathbb D}$ for any ${\mathbb D}\in\Hom_{\MTL(\delta)}((m, v^r), (m, v^s))$.

\item $\End_{\MTL(\delta)}(m)=\K[Z_2]$, the polynomial algebra in $Z_2$;  
 
\item $\End_{\MTL(\delta)}((m, v))=(\K[Z_2]\ot I) \oplus (\K[Z_2]\ot I)\BH$, which is  a commutative $\K$-algebra.
\end{enumerate}
\end{corollary}
\begin{proof}
Part (1) is a consequence of Lemma \ref{lem:central}(2). The rest of the lemma follows from \eqref{eq:TL-H} and part (1). 
\end{proof}

Let $\ATL(\delta)$ be the full subcategory of $\MTL(\delta)$ with objects $(m, v^r)$ for all $r\in\N$; we refer to $\ATL(\delta)$ as the {\em polar Temperley-Lieb category}. 
%

When considering $\ATL(\delta)$, we shall adopt the convention of Remark \ref{rmk:no-m} by writing the object $(m, v^r)$ as $r$. 
We shall also write $f(Z_2) {\mathbb D}$  for both ${\mathbb D}(f(Z_2)\ot I_r)$ and $(f(Z_2)\ot I_s){\mathbb D}$ for any $f(Z_2)\in \K[Z_2]$ and ${\mathbb D}\in \Hom_{\ATL(\delta)}(r, s)$. This should not cause any confusion because of parts (i) and (ii) of Corollary \ref{lem:TL-4-t}.

Now the space of homomorphisms $\Hom(\ATL(\delta))$ of $\ATL(\delta)$ may be thought of as a $\K[Z_2]$-module,  which is clearly free.
Note in particular that for any $N$, a polar Temperley-Lieb $(0, 2N)$-diagram is the product of some power of $Z_2$ and a diagram of the form shown in Figure \ref{fig:ATL2N},   
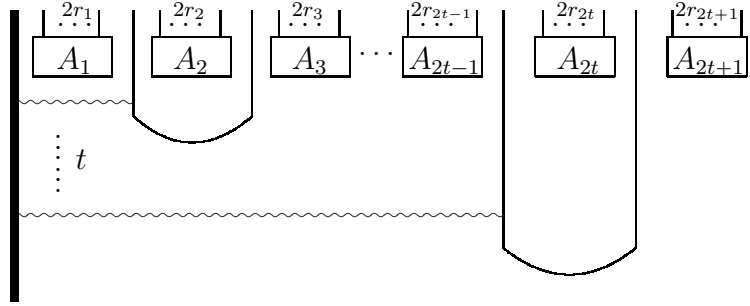
\begin{figure}[h]
\begin{picture}(280, 120)(10,-60)
{
\linethickness{1mm}
\put(10, -60){\line(0, 1){110}}
}

\put(10, 18){\uwave{\hspace{15mm}}}
\put(55, 10){\line(0, 1){40}}
\put(100, 10){\line(0, 1){40}}

\qbezier(55, 10)(77, -10)(100, 10)

\put(17, 25){\line(1, 0){30}}
\put(17, 40){\line(1, 0){30}}
\put(17, 25){\line(0, 1){15}}
\put(47, 25){\line(0, 1){15}}
\put(21, 40){\line(0, 1){10}}
\put(42, 40){\line(0, 1){10}}
\put(26, 28){$A_1$}
\put(26, 44){$\dots$}
\put(28, 48){\tiny$2r_1$}

\put(62, 25){\line(1, 0){30}}
\put(62, 40){\line(1, 0){30}}
\put(62, 25){\line(0, 1){15}}
\put(92, 25){\line(0, 1){15}}
\put(65, 40){\line(0, 1){10}}
\put(88, 40){\line(0, 1){10}}
\put(70, 28){$A_2$}
\put(68, 44){$\dots$}
\put(70, 48){\tiny$2r_2$}

\put(107, 25){\line(1, 0){30}}
\put(107, 40){\line(1, 0){30}}
\put(107, 25){\line(0, 1){15}}
\put(137, 25){\line(0, 1){15}}
\put(110, 40){\line(0, 1){10}}
\put(133, 40){\line(0, 1){10}}
\put(115, 28){$A_3$}
\put(113, 44){$\dots$}
\put(115, 48){\tiny$2r_3$}

\put(140, 33){$\dots$}
\put(157, 25){\line(1, 0){30}}
\put(157, 40){\line(1, 0){30}}
\put(157, 25){\line(0, 1){15}}
\put(187, 25){\line(0, 1){15}}
\put(159, 40){\line(0, 1){10}}
\put(185, 40){\line(0, 1){10}}
\put(159, 28){$A_{2t-1}$}
\put(163, 44){$\dots$}
\put(160, 48){\tiny$2r_{\scriptscriptstyle{2t-1}}$}

\put(207, 25){\line(1, 0){30}}
\put(207, 40){\line(1, 0){30}}
\put(207, 25){\line(0, 1){15}}
\put(237, 25){\line(0, 1){15}}
\put(210, 40){\line(0, 1){10}}
\put(233, 40){\line(0, 1){10}}
\put(215, 28){$A_{2t}$}
\put(213, 44){$\dots$}
\put(215, 48){\tiny$2r_{2t}$}

\put(257, 25){\line(1, 0){30}}
\put(257, 40){\line(1, 0){30}}
\put(257, 25){\line(0, 1){15}}
\put(287, 25){\line(0, 1){15}}
\put(259, 40){\line(0, 1){10}}
\put(285, 40){\line(0, 1){10}}
\put(259, 28){$A_{2t+1}$}
\put(263, 44){$\dots$}
\put(260, 48){\tiny$2r_{2t+1}$}


\put(20, -10){$\begin{array}{c c}
\vdots \vspace{-2.5mm}\\
\vdots
\end{array} t $}

\put(10, -25){\uwave{\hspace{65mm}}}
\put(195, -40){\line(0, 1){90}}
\put(245, -40){\line(0, 1){90}}

\qbezier(195, -40)(220, -60)(245, -40)
\end{picture}
\caption{Standard polar Temperley-Lieb diagram}
\label{fig:ATL2N}
\end{figure}
where there are $t$ connectors with $0\le t\le N$, and each $A_j$ is a usual Brauer $(0,2r_j)$-diagram with no intersections,
such that $\sum_j r_j=N- t$.  We call such a polar Temperley-Lieb diagram {\em standard}.

We have the following result.

\begin{theorem} \label{thm:ATL-dim}
For any non-negative integers $N$ and $r$ such that $r\le 2N$, the rank of $\Hom_{\ATL(\delta)}(r, 2N-r)$ over $\K[Z_2]$ is equal to   
$\begin{pmatrix}2N \\ N \end{pmatrix}$.
\end{theorem}

\begin{proof}  
We adapt some key ideas of the proof of \cite[Lemma 3.24]{ILZ} to the present context. 
Let $\wh{K}$ be the fraction field of $\K[Z_2]$. Let $W(r, s)=\wh{K}\ot_{\K[Z_2]}\Hom_{\ATL(\delta)}(r, s)$ for all $r, s$, and denote $W(2N)=W(0, 2N)$. 
The standard polar Temperley-Lieb $(0, 2N)$-diagrams form a basis of $W(0, 2N)$. 
For each $i\le N$, the standard $(0, 2N)$-diagrams with at most $i\le N$ connectors span a
$\wh{K}$-subspace $F_iW(2N)$ of $W(2N)$, and thus we have a
filtration of $W(2N)$ given as follows.
\beq\label{eq:TL-filt}
W(2N)=F_N W(2N)\supset F_{N-1}W(2N)  \supset \dots \supset  F_0 W(2N)\supset 0.
\eeq

Note that for any $n$, the subspace $E_n^0$ of $W(n, n)$  spanned by polar Temperley -Lieb $(n, n)$-diagrams without connectors forms a $\wh{K}$-algebra isomorphic to the Temperley-Lieb algebra ${\rm TL}_n(\delta)$ of degree $n$ with parameter $\delta$; in particular, $E^0_{2N}={\rm TL}_{2N}(\delta)$. 

Now ${\rm TL}_{2N}(\delta)$ naturally acts on $W(2N)$ by composition of morphisms in $\ATL(\delta)$, and each $F_iW(2N)$ forms a ${\rm TL}_{2N}(\delta)$-submodule. Thus \eqref{eq:TL-filt} is a filtration of ${\rm TL}_{2N}(\delta)$-modules. It follows from \eqref{eq:TL-H} and skew symmetry of $\BH$ that $W_{2i}(2N):=\frac{F_i W(2N)}{F_{i-1} W(2N)}$ is a simple ${\rm TL}_{2N}(q)$-module for each $i$. Recall that the dimension of this simple ${\rm TL}_{2N}(\delta)$-module is given by $\dim_{\wh{K}}(W_{2i}(2N))= \begin{pmatrix}2N \\ N-i \end{pmatrix}-\begin{pmatrix}2N \\ N-i-1 \end{pmatrix}$. Therefore,  
$\dim_{\wh{K}}(W(2N))=\sum_i \dim_{\wh{K}}(W_{2i}(2N)) = \begin{pmatrix}2N \\ N \end{pmatrix}$.

However $\dim_{\wh{K}}(W(r, 2N-r))=\dim_{\wh{K}}(W(2N))$ for all $r$ with $0\leq r\leq 2N$, since pulling of strings shows that $W(a, b)$ and $W(0, a+b)$ 
are isomorphic as vector spaces over $\wh{K}$ for all $a, b$. This completes the proof of the Theorem.
\end{proof}

Since $Z_2$ is central in $\ATL(\delta)$, we may take a quotient of the category by setting $Z_2$ to some fixed parameter. 
\begin{definition}  \label{def:TLB}
Fix $\lambda\in \K$. Denote by $\TLBC(\delta, \lambda)$ the quotient category of $\ATL(\delta)$ obtained by setting 
\[
\frac{Z_2}{\delta}=-\lambda\left(\frac{\delta-2}{2}-\lambda\right), 
\]
and call it a {\em Temperley-Lieb category of type $B$}. 
\end{definition}
We continue to denote the image of $\BH$ in $\TLBC(\delta, \lambda)$ by the same symbol $\BH$. We have the following result. 
\begin{lemma}
The following relation holds in $\TLBC(\delta, \lambda)$.
\[
(\BH + \lambda)\left(\BH + \frac{\delta-2}{2}-\lambda\right) =0.
\]
\end{lemma}
\begin{proof}
This follows from Lemma \ref{lem:TL-H}, which asserts that $\BH$ satisfies the quadratic relation $\BH^2+(\frac{\delta-2}{2})\BH-\frac{Z_2}{\delta}=0$ over $\K$. 
\end{proof}

There is a close relationship between the category $\TLBC(-2, \lambda)$ and a category of infinite dimensional $\fsp_2$-representations, which will be studied in detail in Section \ref{sect:sp2}.

\section{Applications to $G_2$}\label{sect:EB-G}

We consider an example of $\HB_{\CG_1}(\delta_\CV)$ with non-trivial $\CG_1$. 
The categories discussed in this section are suitable for studying representations 
of the Lie algebra $G_2$. 

Let us first recall some basic facts about $G_2$. 

\bigskip
\noindent
\subsection{Some basic facts concerning $G_2$}  \label{sect:so-G2}

 Let $V=\C^7$ be equipped with a non-degenerate symmetric bilinear form $\omega$, and we take $\fso_7(\C)=\fso(V; \omega)$. Now $\omega$ gives rise to  non-zero $\fso_7(\C)$-invariant maps $\hat{C}: V\ot V\lra \C$ and $\check{C}: \C\lra V\wedge V$, where $\hat{C}(v\ot  v')=\omega(v, v')$ for all $v, v'\in V$, and $\check{C} \hat{C}$ is a quasi idempotent such that $(\check{C} \hat{C})^2=7 \check{C} \hat{C}$.  The Lie algebra $\fso(V; \omega)$ can be equivalently described as the subalgebra of $\gl_7(\C)$ annihilating $\check{C}(1)$.

There is a map $\check{C}_3: \C\lra V \wedge V\wedge V$ such that $G_2$ is the subalgeba of $\fso_7(\C)$ annihilating $\check{C}_3(1)$. Explicitly, if we
 assume that $\omega(e_i, e_j)=\delta_{i j}$ for the elements of
the standard basis $\{e_i\mid 1\le i\le 7\}$ of $V$, then $\check{C}(1)=\sum_{i=1}^7 e_i \ot e_i$.
The element $\check{C}_3(1)$ is then given by
\[
\baln
\check{C}_3(1)&=e_1 \wedge e_2 \wedge e_3 - e_1 \wedge (e_4\wedge e_5 + e_6\wedge e_7)\\
&- e_2 \wedge (e_4  \wedge  e_6  - e_5\wedge e_7) - e_3\wedge (e_4 \wedge  e_7 + e_5 \wedge e_6). 
\ealn
\]
%
%
%
It is straightforward that 
\beq\label{eq:norm}
 (\id_V\ot \omega\ot\id_V)(\id_V^{\ot2} \omega\ot  \id_V^{\ot 2})(\check{C}_3(1) \ot  \check{C}_3(1))=3 \check{C}(1).
\eeq

The Lie algebra $G_2$ is $14$-dimensional and is of rank $2$. 
Denote by $\alpha_1$ and $\alpha_2$ the two simple roots of $G_2$ with $\alpha_1$ short and $\alpha_2$ long. Denote by $\lambda_1$ and $\lambda_2$ the fundamental weights, i.e., a basis of the weight space satisfying 
$
\frac{2(\alpha_j, \lambda_i)}{(\alpha_j, \alpha_j)}=\delta_{i j}
$ 
for $i, j=1, 2$. In fact both fundamentals weights are roots, with 
$\lambda_1 = 2\alpha_1+\alpha_2$ and $\lambda_2= 3\alpha_1 +2\alpha_2$.  
Note that $\lambda_2$ is the maximal root, thus the simple module $L_{\lambda_2}$ is isomorphic to $G_2$ under the adjoint action. The simple module $L_{\lambda_1}$ is the natural module $V=\C^7$. It is self-dual, and is strongly multiplicity free in the sense of \cite{LZ06}.  

Consider $V\ot V$ by first decomposing it into the direct sum of the symmetric and skew symmetric submodules, and then decomposing the two submodules into direct sums of simple modules. We have  
\[
V\ot V = S^2(V) \oplus \wedge^2(V), \quad S^2(V)=L_{2\lambda_2}\oplus L_0, \quad 
\wedge^2(V) = L_{\lambda_2} \oplus L_{\lambda_1}.
\]
The eigenvalues of the Casimir operator in the relevant simple modules are respectively given by 
\[
\chi_0(C) =0, \quad \chi_{\lambda_1}(C) = 12, \quad \chi_{\lambda_2} (C) = 24, \quad \chi_{2\lambda_1}(C) = 28, 
\]
and hence the eigenvalues of the tempered Casimir operator on $V\ot V$ are given by 
\[
\chi_0(t) =-12, \quad \chi_{\lambda_1}(t) = -6, \quad \chi_{\lambda_2} (t) = 0, \quad \chi_{2\lambda_1}(t) = 2.  
\]

Denote by $P_\zeta$ the idempotent mapping $V\ot V$ onto the simple submodule $L_\zeta$ for $\zeta = 0, \lambda_1, \lambda_2$ and $2\lambda_1$. Let $\tau=\tau_{V, V}$,  $e=\check{\omega}_V{\omega}_V$, and $\SH=(\mu\ot\mu)(t)$. Then 
\beq
&1=P_{2\lambda_1} + P_0 + P_{\lambda_1 } + P_{\lambda_2}, \\
&\tau=P_{2\lambda_1} + P_0 - P_{\lambda_1 } - P_{\lambda_2}, \\
&e= 7 P_0, \\
&
\SH=  2 P_{2\lambda_1} -12 P_0 - 6 P_{\lambda_1 }. 
\eeq
Hence $\SH$ satisfies the quartic polynomial equation  
\[
\SH(\SH-2) (\SH+6)(\SH+12)=0. 
\]
Furthermore, we can express $\SH$ as 
\beq
&\SH =1 +\tau - 2 e - 6 P_{\lambda_1 }. \label{eq:spectral}
\eeq

\begin{lemma}
The following relation holds in $\End_{G_2}(V\ot V)$,  
\[
6 ( ^\vee \ot \id) P_{\lambda_1 } + 6 P_{\lambda_1}= -2\tau + e  + 1, 
\]
where the map $^\vee$ is defined by \eqref{eq:vee}.
\end{lemma}
\begin{proof}
We have shown that $
(\id\ot   ^\vee) \SH	=- \SH.
$
Using 
$
\SH=1 +\tau - 2 e - 6 P_{\lambda_1 },
$
we obtain 
$
-(1 +\tau - 2 e - 6 P_{\lambda_1 }) = (\id\ot   ^\vee) \SH	
= 1 + e - 2 \tau - 6 (\id\ot   ^\vee)P_{\lambda_1 }.
$
Thus
\[
6 (\id\ot ^\vee)P_{\lambda_1 } +6 P_{\lambda_1} -2 + e  + \tau=0. 
\]
Multiplying this by $\tau$ from the left, and using $\tau P_{\lambda_1}=- P_{\lambda_1}$, we obtain 
\beq\label{eq:P1tran}
 6( ^\vee \ot\id) P_{\lambda_1 }+ 6 P_{\lambda_1}= -2\tau + e  + 1.
\eeq
This completes the proof. 
\end{proof}

\subsection{Coloured Brauer categories with infinitesimal braids for $G_2$}\label{sect:B-G2}
We now consider the special case of $\HB_{\CG_1}(\delta_\CV)$ with 
 $\CM\supset \CV=\{v\}$. 

\subsubsection{$G_2$ categories}

Maintain the convention of Remark \ref{rmk:no-v}, i.e., drop the label $v$ from all arcs of colour $v$, and 
write $I^v$, $X_{v, v}$, $\cup^v$ and $\cap_v$ simply as  
$I$, $X$, $\cup$ and $\cap$.  We also write $\delta=\delta_\CV$. 

Assume that $|\CG_1|$ consists of the single element $\check\omega_3\in \Hom(\emptyset, v^3)$, which is totally skew symmetric in the sense that 
\beq\label{eq:skew-omega}
(X\ot I) \check\omega_3= (I\ot X) \check\omega_3 =- \check\omega_3.
\eeq
Call this category $\HB_{\check\omega_3}(\delta)$. 

The morphism $\check\omega_3$  will be depicted graphically by 
\[
\begin{picture}(80, 30)(-25, 0)
\put(-25, 10){$\check\omega_3=$}

\put(20, 0){\line(0, 1){30}}

\qbezier(10, 30)(10, 5)(20, 0)
\qbezier(30, 30)(30, 5)(20, 0)

\put(30, 0){.}
\end{picture}
\]
Let $\Upsilon=(I\ot I \ot  \cap) (\check\omega_3\ot I)$, $\hat\Upsilon=(I\ot \cap)(\Upsilon\ot I)$ and $\omega_3=\cap(\hat\Upsilon\ot I)$. Then it follows the total skew symmetry of $\check\omega_3$ given by \eqref{eq:skew-omega} that 
\beq\label{eq:skew-omega-2}
\check\omega_3= (I\ot \Upsilon)\cup, \quad  \hat\Upsilon=(\cap\ot I)(I\ot\Upsilon), \quad
\omega_3=\cap(I\ot\hat\Upsilon).
\eeq

Let  
$\CK=\Upsilon\hat\Upsilon, 
\text{and}\  
\begin{picture}(50, 10)(-3, 0)
\put(0, 0){$E=$}
\qbezier(30, 0)(35, 8)(40, 0)
\qbezier(30, 10)(35, 2)(40, 10)
\put(45, 0){.}
\end{picture}
$
We want the morphism $\check\omega_3$ to mimic the map $\sqrt{2}\check{C}_3$ in the $G_2$ context, and thus $\CK$ mimics some scalar multiple of  $P_{\lambda_1}$.  The category to be defined should have analogues of relations \eqref{eq:norm} and \eqref{eq:spectral}. Therefore, we introduce the following definition. 

\begin{definition}
Let $\CJ$ be the tensor ideal in $\HB_{\check\omega_3}(\delta)$ generated by the morphisms
\beq
\hat\Upsilon\Upsilon - 6 I, \quad  \BH(v, v) - X +2 E - I\ot I+ \Upsilon\hat\Upsilon. 
\eeq
Denote by $\wh\HB_{\check\omega_3}(\delta)$ the quotient category $\HB_{\check\omega_3}(\delta)/\CJ$, and refer to is as the {\em coloured Brauer category with infinitesimal braids} for $G_2$-representations.
\end{definition}

We have the following result. 
\begin{lemma}\label{lem:constraint}
If $I\ne 0$ in $\wh\HB_{\check\omega_3}(\delta)$, then $\delta=7$. 
\end{lemma}
We will prove this later after we have developed the necessary diagrammatics. 

As usual, we shall still use the same symbols and corresponding pictures for diagrams in $\HB_{\check\omega_3}(\delta)$ to represent their images in $\wh\HB_{\check\omega_3}(\delta)$.  In view of \eqref{eq:skew-omega-2}, we can and will simply represent the morphisms $\Upsilon$, $\hat\Upsilon$,  and $\omega_3$ by the following diagrams respectively. 
\[
\begin{picture}(65, 30)(-25, -7)
\put(-25, 10){$\Upsilon\  \  =$}

\put(20, 15){\line(0, -1){15}}

\qbezier(10, 30)(10, 30)(20, 15)
\qbezier(30, 30)(30, 30)(20, 15)

\put(30, 0){,}
\end{picture}
\quad
\begin{picture}(80, 40)(-25,-10)
\put(-25, 10){$\hat\Upsilon=$}

\put(20, 30){\line(0, -1){20}}

\qbezier(10, -5)(10, -5)(20, 10)
\qbezier(30, -5)(30, -5)(20, 10)

\put(35, 0){,}
\end{picture}
\begin{picture}(60, 40)(-25,-10)
\put(-25, 10){$\omega_3=$}

\put(20, -5){\line(0, 1){30}}

\qbezier(10, -5)(10, 20)(20, 25)
\qbezier(30, -5)(30, 20)(20, 25)

\put(35, 0){.}
\end{picture}
\]

\noindent
{\bf Additional relations in $\wh\HB_{\check\omega_3}(\delta)$}. The category $\wh\HB_{\check\omega_3}(\delta)$ inherites all the relations from $\HB_{\check\omega_3}(\delta)$. It also has the additional defining relations arising from the generators of $\CJ$. They can be pictorially represented by 
Figure \ref{fig:Upsilon-1} and Figure \ref{fig:Upsilon-2}.

\begin{figure}[h]
\begin{picture}(80, 40)(0, 0)

\put(10, 0){\line(0, 1){10}}

\qbezier(0, 20)(10, 10)(10, 10)
\qbezier(20, 20)(10, 10)(10, 10)

\qbezier(0, 20)(10, 30)(10, 30)
\qbezier(20, 20)(10, 30)(10, 30)

\put(10, 30){\line(0, 1){10}}

\put(28, 18){$- \  \, 6$}
\put(60, 0){\line(0, 1){40}}

\put(70, 18){$= \ 0$}
\end{picture}
\caption{Relation (1) in $\wh\HB_{\check\omega_3}(\delta)$}
\label{fig:Upsilon-1}
\end{figure}

\begin{figure}[h]
\begin{picture}(280, 40)(0, 0)

\put(0, 0){\line(0, 1){40}}

\put(0, 20){\uwave{\hspace{7mm}}}
\put(20, 0){\line(0, 1){40}}

\put(30, 16){$- $}

\qbezier(50, 40)(70, 0)(70, 0)
\qbezier(70, 40)(50, 0)(50, 0)

\put(82, 16){$+ $}
\put(105, 16){$2 $}
\qbezier(120, 40)(130, 5)(140, 40)
\qbezier(120, 0)(130, 35)(140, 0)

\put(150, 16){$- $}

\put(175, 0){\line(0, 1){40}}
\put(195, 0){\line(0, 1){40}} 

\put(208, 16){$+$}

\qbezier(225, 40)(235, 25)(235, 25)
\qbezier(245, 40)(235, 25)(235, 25)

\qbezier(225, 0)(235, 15)(235, 15)
\qbezier(245, 0)(235, 15)(235, 15)

\put(235, 15){\line(0, 1){10}} 

\put(255, 16){$= \ 0$}

\end{picture}
\caption{Relation (2) in $\wh\HB_{\check\omega_3}(\delta)$}
\label{fig:Upsilon-2}
\end{figure}


Lemma \ref{lem:reduction} below provides a powerful tool for manipulating diagrams 
in $\wh\HB_{\check\omega_3}(\delta)$.  
\begin{lemma} \label{lem:reduction}
The following relation holds in $\wh\HB_{\check\omega_3}(\delta)$.
\[
\begin{picture}(270, 40)(-75, 0)

\qbezier(-75, 40) (-75, 40) (-65, 20)
\qbezier(-35, 40) (-35, 40) (-45, 20)
\put(-65, 20){\line(1, 0){20}}

\qbezier(-75, 0) (-75, 0) (-65, 20)
\qbezier(-35, 0) (-35, 0) (-45, 20)



\put(-25, 16){$+$}

\qbezier(-10, 40)(0, 25)(0, 25)
\qbezier(10, 40)(0, 25)(0, 25)

\qbezier(-10, 0)(0, 15)(0, 15)
\qbezier(10, 0)(0, 15)(0, 15)

\put(0, 15){\line(0, 1){10}} 

\put(18, 16){$+ \  \ \  2$}

\qbezier(50, 40)(70, 0)(70, 0)
\qbezier(70, 40)(50, 0)(50, 0)

\put(85, 16){$- $}
\qbezier(110, 40)(120, 5)(130, 40)
\qbezier(110, 0)(120, 35)(130, 0)

\put(140, 16){$- $}

\put(160, 0){\line(0, 1){40}}
\put(180, 0){\line(0, 1){40}} 

\put(190, 16){$= \ 0 $ .}

\end{picture}
\]
\end{lemma}
\begin{proof}
For any $\B\in \Hom(v^2, v^2)$, let $\B^T=  (I\ot \cap\ot I)(\B\ot X)(I\ot \cup\ot I)$. 
Denote by $\BL$ the left hand side of the relation of Figure \ref{fig:Upsilon-2}, and let $\CK=\Upsilon\hat\Upsilon$. Write $\CH=\BH(v,v)$. Then 
\[
\BL^T = \CH^T -  E + 2 X - I\ot I +\CK^T= -\CH -  E + 2 X - I\ot I +\CK^T=0, 
\]
where the second equality uses the skew symmetry of $\CH$. Using Figure \ref{fig:Upsilon-2}, we obtain 
\[
X+ E - 2 I\ot I + \CK+\CK^T=0. 
\]
Multiplying both sides of this relation by $X$ on the left, and then using the fact that $X\CK=- \CK$, we obtain 
\beq\label{eq:simplify}
\CK - X \CK^T + 2 X  - E  - I\ot I =0. 
\eeq
Pictorially, $X \CK^T$ is given by 
\[
\setlength{\unitlength}{0.3mm}
\begin{picture}(60, 60)(0, 0)
\qbezier(0, 60)(30, 50)(30, 50)
\qbezier(0, 50)(30, 60)(30, 60)

\put(0, 50){\line(0, -1){10}} 
\put(30, 50){\line(0, -1){10}} %

\qbezier(0, 40)(10, 30)(10, 30)
\qbezier(20, 40)(10, 30)(10, 30) %

\put(10, 30){\line(0, -1){10}} 

\qbezier(0, 5)(10, 20)(10, 20)
\qbezier(20, 10)(10, 20)(10, 20)%

\qbezier(20, 10)(32, 8)(30, 40)%
\qbezier(20, 40)(35, 42)(35, 5) %

\put(40, 30){$= \ - $}
\end{picture}
\quad
\begin{picture}(80, 60)(-5, -10)

\put(-5, 27){\line(0, 1){18}}

\put(15, 27){\line(0, 1){18}} 

\qbezier(15, 27)(25, 13)(25, 13)
\qbezier(15, 27)(5, 13)(5, 13)
\qbezier(-5, 27)(5, 13)(5, 13)
\put(5, -5){\line(0, 1){18}} 
\put(25, -5){\line(0, 1){18}} 
\put(35, 20){$=\ - $}

\qbezier(75, -5)(75, -5)(85, 22)
\qbezier(115, -5)(115, -5)(105, 22)
\put(85, 22){\line(1, 0){20}} 

\qbezier(75, 50)(75, 50)(85, 22)
\qbezier(115, 50)(115, 50)(105, 22)

\put(120, 0){,} 
\end{picture}
\]
where the first equality is obtained by applying \eqref{eq:skew-omega} and the skew symmetry of $\Upsilon$. The second equality is simply obtained by straightening
the bent arcs on the left side.   Using this relation in the pictorial description of equation \eqref{eq:simplify}, we obtain the lemma. 
\end{proof}

Lemma \ref{lem:reduction} enables one to eliminate triangles, squares, pentagons, etc. built from $\Upsilon$ and its relatives.  For example, one easily deduces from $(I\ot\cap)(\BL\ot I)(I\ot\cup)=0$ and $\BL\Upsilon=0$ the following relations:
\[
\setlength{\unitlength}{0.3mm}
\begin{picture}(150, 40)(-10, 0)
\qbezier(-10, 40)(0, 25)(0, 25)
\qbezier(10, 40)(0, 25)(0, 25)

\qbezier(-10, 0)(0, 15)(0, 15)
\qbezier(10, 0)(0, 15)(0, 15)

\put(0, 15){\line(0, 1){10}} 

\qbezier(10, 0)(20, 0)(20, 20)
\qbezier(10, 40)(20, 40)(20, 20)

\put(30, 15){$+ \  (1-\delta)$} 
\put(95, 0){\line(0, 1){40}} 
\put(105, 15){$=  \ 0$} 
\put(135, 10){; }
\end{picture}
\]

\[
\setlength{\unitlength}{0.3mm}
\begin{picture}(150, 55)(-5, -25)

\put(5, 10){\line(0, 1){20}} 
\put(25, 10){\line(0, 1){20}}
\put(5, 10){\line(1, 0){20}} 

\qbezier(5, 10)(5, 10)(15, -5)
\qbezier(25, 10)(25, 10)(15, -5)

\put(15, -5){\line(0, -1){18}}
 
\put(40, 0){$+\ \  3 $}

\qbezier(70, 30)(80, 5)(80, 5)
\qbezier(90, 30)(80, 5)(80, 5)
\put(80, 5){\line(0, -1){25}}


\put(100, 0){$=\ \ 0$}
\put(140, -10){.}
\end{picture}
\]

We are now in a position to prove that the parameter $\delta$ must be equal to $7$.  
\begin{proof}[Proof of Lemma \ref{lem:constraint}]  

Comparing the first relation above with Figure \ref{fig:Upsilon-1} taking into account that 
\[
\setlength{\unitlength}{0.30mm}
\begin{picture}(50, 30)(-10, 0)
\qbezier(-10, 40)(0, 25)(0, 25)
\qbezier(10, 40)(0, 25)(0, 25)

\qbezier(-10, 0)(0, 15)(0, 15)
\qbezier(10, 0)(0, 15)(0, 15)

\put(0, 15){\line(0, 1){10}} 

\qbezier(10, 0)(20, 0)(20, 20)
\qbezier(10, 40)(20, 40)(20, 20)
\put(35, 15){$=$}
\end{picture}
\quad
\setlength{\unitlength}{0.35mm}
\begin{picture}(50, 30)(-5, 2)

\put(10, 0){\line(0, 1){10}}

\qbezier(0, 20)(10, 10)(10, 10)
\qbezier(20, 20)(10, 10)(10, 10)

\qbezier(0, 20)(10, 30)(10, 30)
\qbezier(20, 20)(10, 30)(10, 30)

\put(10, 30){\line(0, 1){10}}


\put(28, 5){,}
\end{picture}
\]
we obtain $(1-\delta) I+ 6 I=0$. Since we assume that $I\ne 0$, it follows that $\delta=7$. 
\end{proof}

\subsubsection{Recovery of the diagram category $\mathscr{T}_\Gamma$ of  \cite{BE}}
Let us consider $\wh\HB_{\check\omega_3}(7)$ in the simplest case with $\CM=\CV=\{v\}$, where we denote the category by $\CB_{\check\omega_3}(7)$. When $\CM\supsetneq\CV=\{v\}$, the category $\CB_{\check\omega_3}(7)$ appears as a full subcategory of $\wh\HB_{\check\omega_3}(7)$. 
We will show that $\CB_{\check\omega_3}(7)$ recovers the monoidal diagram category $\mathscr{T}_\Gamma$ in \cite[\S 4.1]{BE} (also see \cite{Ros96, Ros04}), 
which are closely related to the $G_2$ spiders of Kuperberg \cite{Kg}. The category $\mathscr{T}_\Gamma$ was called a ``3-tangle category'' in \cite{BE}. 

\begin{remark}\label{rmk:Benkart}
The objects of the diagram category $\mathscr{T}_\Gamma$ \cite[\S 4.1]{BE} are the integers $r\in\N$, and the morphisms are generated by the standard Brauer generators (i.e., $I$, $X$, $\cup$, $\cap$) and an additional generator which formally looks the same as the diagram of $\Upsilon$. Besides the usual relations satisfied by the Brauer generators (the loop removal relation is called $\gamma_0$ in op. cit.), 
there are two further
 relations $\gamma_1$ and $\gamma_2$ in op. cit., which  involve the extra 
 generator (with $\Gamma=\{\gamma_0, \gamma_1, \gamma_2\}$).  The relation $\gamma_1$ is nothing but the skew symmetry property of the extra generator. The relation  $\gamma_2$ is described  below. 
\end{remark}

\begin{theorem}\label{thm:G2-iso}
 As monoidal category,  $\CB_{\check\omega_3}(7)$ is isomorphic to
the diagram category $\mathscr{T}_\Gamma$ in \cite[\S 4]{BE}. 
\end{theorem}
\begin{proof}[Outline of the proof]    
Note that we can use 
$\Upsilon$ as a generator for $\CB_{\check\omega_3}(7)$ 
instead of $\check\omega_3$. 
If we identify our $\Upsilon$ with $\sqrt{-1}$ times the morphism of $\mathscr{T}_\Gamma$ represented by the same picture in \cite[\S 4.1]{BE} (and hence $\CK$ is replaced by the negative of the morphism represented by the same diagram in op. cit.), 
Lemma \ref{lem:reduction} becomes the defining relation $\gamma_2$ of the category in op. cit..  Note that Lemma \ref{lem:reduction} is equivalent to the defining relation of $\CB_{\check\omega_3}(7)$ arising from the second generator of the tensor ideal $\CJ$. 

For $\delta=7$, the relation $\hat\Upsilon\Upsilon - 6 I=0$ arising from the first generator of 
$\CJ$ is a consequence of Lemma \ref{lem:reduction} and the skew symmetry of $\Upsilon$, as we have shown in the proof of Lemma \ref{lem:constraint}. This fact was also proven in \cite{BE, Ros04}. 

These observations together with Remark \ref{rmk:Benkart} led to an isomorphism between the monoidal categories $\CB_{\check\omega_3}(7)$ and $\mathscr{T}_\Gamma$ of \cite[\S 4]{BE}. 
\end{proof}

The following result is now straightforward. 
\begin{theorem} 
Retain the notation of Section \ref{sect:so-G2}; in particular, 
denote by $V$ the $7$-dimensional simple $G_2$-module. 
Let  $\CT(V)$ be the full subcategory of the category of $G_2$-modules with objects $V^{\ot r}$ for all $r\in \N$.
There exists an isomorphism of monoidal categories  
$
\CF: \CB_{\check\omega_3}(7)\lra \CT(V),
$
which maps an object $r$ to $V^{\ot r}$, and maps the generators of the space of morphisms as indicated below
\[
\begin{aligned}
&\CF\left(\setlength{\unitlength}{0.1mm}
\begin{picture}(10, 20)(0,0)
\put(5, 0){\line(0, 1){30}}
\end{picture}\right) = \id_V, \quad
&\CF\left(
\setlength{\unitlength}{0.1mm}
\begin{picture}(28, 23)(0,5)
\qbezier(0, 0)(0, 0)(25, 35)
\qbezier(0, 35)(0, 35)(25, 0)
\end{picture}\right) = \tau, \\
&\CF\left(\setlength{\unitlength}{0.1mm}
\begin{picture}(28, 15)(0,0)
\qbezier(0, 0)(13, 35)(25, 0)
\end{picture}\right)= \hat{C}, \quad
&\CF\left(\setlength{\unitlength}{0.1mm}
\begin{picture}(28, 15)(0,15)
\qbezier(0, 35)(13, 0)(25, 35)
\end{picture}\right)=\check{C},   \\
%
&\CF\left(
\setlength{\unitlength}{0.1mm}
\begin{picture}(35, 30)(5, 2)

\put(20, 0){\line(0, 1){30}}

\qbezier(10, 30)(10, 5)(20, 0)
\qbezier(30, 30)(30, 5)(20, 0)

\end{picture}
\right)=\check{C}_3.
\end{aligned}
\]
\end{theorem}
\begin{proof}
It is clear that 
$
\CF: \CB_{\check\omega_3}(7)\lra \CT(V)
$ defined above is a monoidal functor.  Replacing $\CB_{\check\omega_3}(7)$ by the category $\mathscr{T}_\Gamma$ of \cite[\S 4]{BE}, we also obtain a monoidal functor, which was shown to be an isomorphism (see [Theorem 4.14.] in \cite{BE}).  It  thus follows from Theorem \ref{thm:G2-iso} that $\CF$ is an isomorphism of monoidal categories. 
\end{proof}

\begin{remark}
The full nature of the functor $\CF$ was proved in \cite{LZ06} by a specialisation argument. Subsequently, in \cite{Mor11}, this result was 
re-proved using Kuperberg's spiders. Our ``multipolar'' approach subsumes both approaches to this problem, and yields more.
\end{remark}

\subsubsection{A multi-polar enhanced Brauer category for $G_2$}

The next simplest case of the category $\wh\HB_{\check\omega_3}(7)$ is where
$\CM=\{m, v\}\supset\CV=\{v\}$; it will be denoted
by $\MPB_{\check\omega_3}(7)$. 
In this case, we represent diagrams in $\MPB_{\check\omega_3}(7)$ by ``multi-polar enhanced Brauer diagrams'' in a manner similar to Section \ref{sect:mpolar}. We will call $\MPB_{\check\omega_3}(7)$ the {\em multi-polar enhanced Brauer category for $G_2$}. It contains a full subcategory $\PB_{\check\omega_3}(7)$ with objects $(m, v^r)$ for all $r\in \N$, which will be referred to as the {\em polar enhanced Brauer category for $G_2$}. 

We intend to develop the structure of $\MPB_{\check\omega_3}(7)$ and $\PB_{\check\omega_3}(7)$ in future work, 
where we will also explore further applications of these categories to $G_2$-representations. 

\;

{\color{red}  CHECKED TO HERE}

\;

\section{Applications to representation theory}\label{sect:centre}
Henceforth we take $\K$ to be the field $\C$ of complex numbers.
\subsection{The orthosymplectic Lie superalgebra $\osp(V; \omega)$}\label{sect:osp}

Fix a $\Z_2$-graded vector space $V=V_{\bar 0}\oplus  V_{\bar 1}$ with 
$(\dim(V_{\bar 0})| \dim(V_{\bar 1}))= (m|\ell)$. 
We take as given a homogeneous basis $\{e^a\mid a=1, 2, \dots, m+\ell\}$ of $V$, 
where $e^a$ is even if $a\le m$, and is odd if $a>m$. 
For $a=1, 2, \dots, m+\ell$, we write $[a]$ for the parity of $a$, so that $[a]=0$ if $a\le m$, and $[a]=1$ if $a>m$.
Denote by $E_a^b$  the matrix units relative to this basis; then $E_a^b\in\End_\C(V)$ satisfies 
\[
E^c_b (e^a)= \delta^a_b e^c, \quad \forall a, b, c =1, 2, \dots, m+\ell, 
\]
where $\delta^a_b$ is the usual Kronecker delta.
The canonical permutation $\tau:=\tau_{V, V}: V\ot V\lra V\ot V$ may then be expressed as 
\beq\label{eq:tau}
\tau = \sum_{a, b=1}^{m+\ell} (-1)^{[b]}E_a^b\ot E^b_a.
\eeq

Assume that $V$ is equipped with a non-degenerate supersymmetric bilinear form 
$\omega: V\times V \lra \C$. This requires that $\dim(V_{\bar 1})=\ell=2 n$ be even. 
 The form $\omega$ may be  represented by an invertible matrix $g^{-1}=(g^{a b})$, where $g^{a b}=\omega(e^a, e^b)$ for all $a, b$. 
 We write $g=(g_{a b})$.  Note that $g$ has the following symmetry properties: $g_{a b}\ne 0$ only if $[a]=[b]$, and
  $g_{a b} = (-1)^{[a]} g_{b a} = (-1)^{[b]} g_{b a}$. The elements $e_b= \sum_c g_{b c} e_c$ satisfy 
\[
\omega(e_b, e^a) = \sum_c g_{b c} \omega(e^c, e^a)  = \sum_c g_{b c} g^{c a}  = \delta^a_b
\]
Thus $\{e_a\mid a=1, 2, \dots, m+2n\}$ is a basis of $V$, dual to the standard basis.  We have $e^a = \sum_b g^{a b} e_b$. 

Let $E_{a b}=\sum_c g_{a c} E^c _b$, then $E_{a b}(e^c)=e_a\delta_b^c$. 
Introduce the following elements of $\End_\C(V)$:
\beq\label{eq:J}
J_{a b} = E_{a b} - (-1)^{[a][b]} E_{b a}, \quad a, b=1, 2, \dots, m+2n, 
\eeq
where clearly $J_{a b}=-(-1)^{[a][b]}J_{b a}$. We also let
\[
\begin{aligned}
J^a_b :=\sum_{a'} g^{a a'} J_{a' b}	&=E^a_b -  (-1)^{[a][b]}  \sum_{b', a'} g^{a a'} g_{b b'} E^{b'}_{a'}.
\end{aligned}
\]
The elements $J_{a b}$ satisfy the following commutation relations \cite{JG, Z08}
\beq
\begin{aligned}
{} [J_{a b}, J_{c d}] &= g_{c b} J_{a d} + (-1)^{[a]([b]+[c])}g_{d a}  J_{b c} \\
&- (-1)^{[c][d]} g_{d b} J_{a c} - (-1)^{[a][b]} g_{c a} J_{b d},  
\end{aligned}
\eeq
where $[X, Y]=XY - (-1)^{[X][Y]} Y X$ is the $\Z_2$-graded commutator of $X, Y$.  
Clearly $J_{i j}$, $J_{i, m+k}$ and  $J_{m+k, m+\ell}$ with $1\le i< j\le m$, $1\le k\le \ell \le 2n$ are linearly independent, and hence they span 
a Lie superalgebra $\fg$ of dimension $\frac{1}{2}m(m-1) + \frac{1}{2}2n(2n+1) + 2m n$. 

The  Lie superalgebra $\fg$ spanned by the elements $J_{a b}$ is the orthosymplectic Lie superalgebra $\osp(V; \omega)$. 

It is instructive to directly verify that the bilinear form $\omega$ is indeed invariant under the transformations $J_{ab}$. 
Note that 
\beq\label{eq:act}
\begin{aligned}
J_{a b}(e^c)&=e_a \delta^c_b - (-1)^{[a][b]} e_b \delta_a^c.
\end{aligned}
\eeq
Using this relation we obtain 
\[
\begin{aligned}
\omega(J_{a b}(e^c), e^d) &= \delta_a^d \delta^c_b - (-1)^{[a][b]} \delta_b^d\delta_a^c, \\
\omega(e^c, J_{a b}(e^d)) 
&= -(-1)^{([a]+[b])[c]} \left(\delta^d_a \delta^c_b- (-1)^{[a][b]} \delta^d_b \delta^c_a\right),
\end{aligned}
\]
implying that
\[
 \omega(J_{a b}(e^c), e^d) + (-1)^{([a]+[b])[c]} \omega(e^c, J_{a b}(e^d))=0, \quad \forall a, b, c, d.  
\]
This proves the invariance of $\omega$ with respect to the
Lie superalgebra $\fg$ spanned by the elements $J_{a b}$. It is not hard to show that the linear span of the $J_{ab}$ is
precisely the set of all linear transformations of $V$ which preserve $\omega$. We therefore write $\fg=\osp(V,\omega)$.

%
%
%
%
Denote by $X_{a b}$ (resp. $X^a_b$) the image of $J_{a b}$ (resp. $J^a_b$) in $\U(\fg)$ under the canonical embedding $\fg\hookrightarrow \U(\fg)$.  We obtain 
the quadratic Casimir of $\U(\fg)$ by ``tensor contraction''.    
\beq\label{eq:C-formula}
C=\frac{1}{2}\sum_{a, b=1} ^{m+2n}(-1)^{[b]} X^a_b X^b_a.  
\eeq
Now the tempered Casimir element $t=\frac{1}{2}\left(\Delta(C)- C\ot 1 - 1 \ot C\right)$ (see \eqref{eq:deftr})  is equal to 
\beq\label{eq:t-formula}
t=\frac{1}{2}\sum_{a, b=1} ^{m+2n} X^a_b \ot (-1)^{[b]} X^b_a. 
\eeq

Denote the natural representation of $\fg$ on $V$ by $\mu$. Then $\mu(X^a_b)=J^a_b$ for all $a, b$. 
\begin{lemma}  The following formula holds in $\U(\fg)\ot\End(V)$.
\beq
(\id_{\U(\fg)}\ot\mu)(t)&=&\sum_{a, b} X^a_b\ot (-1)^{[b]} E^b_a. \label{eq:t-HH}
\eeq
\end{lemma}
\begin{proof}
Write $\U=\U(\fg)$.
Using  $X^b_a=\sum_c g^{b c} X_{c b}$, we can express $2(\id_\U\ot \mu)(t)$ as 
\[
\begin{aligned}
&2(\id_\U\ot \mu)(t) = \sum_{a, b} X^b_a\ot (-1)^{[a]}J^a_b\\
&= \sum_{a, b} X^b_a\ot (-1)^{[a]}E^a_b 
-  \sum_{a, b, c} g^{b c} X_{c a}\ot(-1)^{[a][b]}  \sum_{b', a'} g^{a a'} g_{b b'} E^{b'}_{a'}. 
\end{aligned}
\]
Note that the first term is equal to the right hand side of \eqref{eq:t-HH}.  We now compute
the second term. 
Using symmetry properties of $g$, we can re-write the second term as 
\[
- \sum_{a, b, c} X_{c a}\ot (-1)^{[a][c] +[a]+[c]}  g^{a b}  E^{c}_{b}
=
- \sum_{a, b, c} g^{ba}  (-1)^{[a][c]} X_{c a}\ot (-1)^{[c]}   E^{c}_{b}.
\]
Using $X_{c a}=-(-1)^{[a][c]}X_{a c}$, we can re-write the right side of the above identity as 
\[
\begin{aligned}
\sum_{a, b, c} g^{b a} X_{a c}\ot (-1)^{[c]}   E^{c}_{b}
= \sum_{b, c}X^b_{c}\ot (-1)^{[c]}   E^{c}_{b}, 
\end{aligned}
\]
which is equal to the right hand side of \eqref{eq:t-HH}. This completes the proof of the Lemma. 
\end{proof}

The bilinear form $\omega$ gives rise to a $\fg$-module homomorphism 
\beq
\wh{C}: V\ot V\lra \C, \quad \wh{C}(v\ot v')=\omega(v, v'), \ \forall v, v'\in V. 
\eeq
Further, since $\omega$ is non-degenerate, we have a linear isomorphism $j: V\ot V \lra \End_\C(V)$ of $\Z_2$-graded vector spaces, which is defined, for any $v\ot v'\in V\ot V$, by 
\[
j(v\ot v')(u)=\omega(v', u) v, \quad \forall u\in V. 
\]
The map $j$ is a $\fg$-module homomorphism if $\End_\C(V)$ is endowed with the natural $\fg$-module structure in which an element $A\in \fg$ acts on an arbitrary element $f\in \End_\C(V)$ by 
$[A, f]=A f - (-1)^{[X][f]} f A$. 
Note that $c_0:=j^{-1}(\id_V)=\sum_{a=1}^{m+2n} e^a\ot e_a$,  is a $\fg$-invariant in $V\ot V$. We therefore have the following $\fg$-morphisms.
\be
\begin{aligned}
&\check{C}: \C\lra V\ot V, \quad c\mapsto c j^{-1}(\id_V);\\
&\hat{C}: V\ot V\lr \C, \quad v\ot v'\mapsto \omega(v,v')\text{ for }v,v'\in V.\\
\end{aligned}
\ee

The following result is well known (cf. \cite{LZ17}). 
\begin{lemma} \label{lem:X-cup-cap}
Let $\fg=\osp(V;\omega)$. Denote the identity map on $V$ by $\id$.
\begin{enumerate}
\item The element $c_0$ belongs to
$(V\otimes V)^\fg$ and satisfies $\tau(c_0)= c_0$. It is independent of the basis.

\item The maps $\tau$, $\check{C}$ and $\hat{C}$ are all $\fg$-equivariant and satisfy
\begin{eqnarray}
&\tau \check{C} =  \check{C}, \quad
\hat{C} \tau =  \hat{C}, \label{eq:Ctau}\\
&\hat{C}\check{C}=\sdim(V),  \quad (\hat{C}\ot\id)(\id\ot\check{C})=\id=(\id\ot\hat{C})(\check{C}\ot\id), \label{eq:C-C}\\
&(\hat{C}\ot\id)\circ (\id\ot \tau)= (\id\ot\hat{C})\circ (\tau\ot\id), \label{eq:C-tau-1}\\
&(\tau\ot \id)\circ(\id\ot \check{C})=(\id\ot \tau)\circ(\check{C}\ot \id). \label{eq:C-tau-2}
\end{eqnarray}
\end{enumerate}
\end{lemma}
Define the endomorphism $e$ of $V\ot V$ as follows. 
\beq\label{eq:e}
e=\check{C}\circ \wh{C} : V\ot V \lra V\ot V. 
\eeq
Then $\tau e = e\tau = e$ by \eqref{eq:Ctau}, and $e^2 = \sdim(V) e$ by the first relation of \eqref{eq:C-C}.

The following observation will play a crucial role later. 
\begin{theorem}\label{thm:H-formula} The action of the tempered Casimir operator  $t$  on $V\ot V$ is given by
\beq
\SH:=(\mu\ot\mu)(t)= \tau - e. 
\eeq
It satisfies the cubic relation 
\beq
(\SH- 1) (\SH + 1) (\SH - (1-\sdim(V)))=0. 
\eeq
\end{theorem}
\begin{proof}
We have $(\mu\ot\mu)(t)= \frac{1}{2}\sum_{a, b} J^a_b \ot (-1)^{[b]}J^b_a$. 
Expressing the $J$'s in terms of $E$'s using \eqref{eq:J}, we obtain from \eqref{eq:t-HH}
\[
(\mu\ot\mu)(t)= \sum_{a, b} J^a_b \ot (-1)^{[b]}E^b_a  =T_1 - T_2, 
\]
with $T_1= \sum (-1)^{[b]} E^a_b\ot  E^b_a$,  and 
$
T_2= \sum (-1)^{[b]} (-1)^{[a]([a]+[b])}   g_{b b''} E^{b''}_{a''} g^{a'' a} \ot E^b_a.
$
We have seen $T_1=\tau$. For any $e^c\ot e^d$, an explicit calculation leads to  
\[
\begin{aligned}
T_2(e^c\ot e^d)
					 &= g^{c d} \check{C}(1)= \check{C}\circ \wh{C}(e^c\ot e^d), 
\end{aligned}
\]
that is,  $T_2=e$.  This proves that $\mathscr{H}=\tau - e$. 

Write $\delta_V=\sdim(V)$. Using $\mathscr{H}=\tau - e$, we easily obtain
\[
\baln
\SH^2= 1 - (2-\delta_V) e, \quad 
\SH^3= \tau -e - (1-\delta_V) (2-\delta_V)  e.
\ealn
\]
These relations together lead to $\SH^3= \SH - (1-\delta_V)(1-\SH^2)$, 
which implies the second claim of the theorem. 
\end{proof}

\begin{remark}
Note the similarity between the above result and the purely diagrammatic Lemma \ref{lem:H-skew-sym}.
This will be exploited below.
\end{remark}

Recall that if $\sdim(V)\ne 0$, then $V\ot V$ is semi-simple with three simple distinct submodules. 
In this case the following statements hold. We omit the proofs, which are straightforward. 
\begin{lemma} \label{lem:H-tau-e} Write $\delta_V=\sdim(V)$. 
\begin{enumerate}
\item If $\delta_V\ne 2$, then 
$e= \frac{(\SH- 1) (\SH + 1)}{2-\delta_V}$ and 
$\tau = \SH + \frac{(\SH- 1) (\SH + 1)}{2-\delta_V}. 
$

\item If $\delta_V=2$, then $(\SH - 1)(\SH+1) =0$, but it is not possible to express $\tau$ and $e$ in terms of $\SH$ in this case.

\item If $\delta_V=0$, then 
$\frac{(\SH + 1) (\SH - (1-\delta_V))}{2\delta_V}, $  
$\frac{(\SH- 1)  (\SH - (1-\delta_V))}{2(2-\delta_V)}$ and
$-\frac{(\SH- 1) (\SH + 1)}{\delta_V (2-\delta_V)}$
form a complete set of mutually orthogonal idempotents, which project $V\ot V$ to its three simple submodules respectively.
\end{enumerate}
\end{lemma}

\begin{remark}
As $\fg$-module $V\ot V$ is semi-simple when $\delta_V=2$, and the idempotents mapping to the simple submodules (there are three) can be easily constructed from $\tau$ and $e$. However, $V\ot V$ is not semi-simple when $\delta_V=0$.
\end{remark}

Finally we compute the eigenvalue of  the Casimir element $C$ in $V$.
\begin{lemma}
The eigenvalue of $C$ in $V$ is equal to
\beq\label{eq:eigen-C}
\chi_V(C) = \sdim(V)-1.
\eeq
\end{lemma}
\begin{proof}
We have  
$
\chi_V(C)\id_V=\frac{1}{2}\sum_{a, b} J^a_b (-1)^{[b]}J^b_a=\sum_{a, b} J^a_b (-1)^{[b]}E^b_a.
$
Expressing $J$'s in terms of $E$'s, we obtain
$\chi_V(C)\id_V
=Q_1 - Q_2$, with
\[
\begin{aligned}
Q_1= \sum (-1)^{[b]} E^a_b E^b_a, \quad
Q_2= \sum (-1)^{[b]} (-1)^{[a]([a]+[b])}   g_{b b''} E^{b''}_{a''} g^{a'' a} E^b_a. 
\end{aligned}
\]
Using properties of the matrix $g$ and carefully keeping track of the sign factors, 
one shows that 
$
Q_1 = \sdim(V) \id_V$  and $Q_2 = \id_V$. This proves 
\eqref{eq:eigen-C}. 
\end{proof}

\subsection{Representations of the multi-polar Brauer category}\label{sect:PB-rep}
Retain the notation of the last section. In particular, $\fg=\osp(V; \omega)$. 

We will also refer to the orthosymplectic supergroup $G=\OSp(V;\omega)$, most efficiently regarded as a group scheme, as well as 
its subgroup 
$G_0={\rm SO}(V_{\bar 0}; \omega|_{\bar 0})\times {\rm Sp}(V_{\bar 1}; \omega|_{\bar 1})$, where $\omega|_{\bar 0}$ and $\omega|_{\bar 1}$ are the restrictions of the bilinear form $\omega$ to the even and odd subspaces of $V$ respectively. It is conceptually appealing to consider the supergroup as a group scheme, however, for the purpose of this paper, we will follow \cite{DM}, 
identifying $G$ with the Harish-Chandra pair $(G_0, \U(\fg)))$, where $G_0$ acts on  $\U(\fg)$ by conjugation.

Denote by $\CT(V)$ the full subcategory of $\U(\fg)$-modules with objects $V^{\ot r}$ for all $r\in\N$, where $V^{\ot 0}=\C$.  Note that $\CT(V)$ is a tensor category.  
The following theorem is well known \cite{LZ17, DLZ}. 
\begin{theorem} \label{thm:funct}
There exists a tensor functor 
$
\CF: \CB(\sdim)\lra \CT(V),
$
which maps an object $r$ to $V^{\ot r}$, and maps the generators of the space of morphisms as shown below.
\[
\begin{aligned}
&\CF\left(\setlength{\unitlength}{0.1mm}
\begin{picture}(10, 20)(0,0)
\put(5, 0){\line(0, 1){30}}
\end{picture}\right) = \id_V, \quad
&\CF\left(
\setlength{\unitlength}{0.1mm}
\begin{picture}(28, 23)(0,5)
\qbezier(0, 0)(0, 0)(25, 35)
\qbezier(0, 35)(0, 35)(25, 0)
\end{picture}\right) = \tau, \\
&\CF\left(\setlength{\unitlength}{0.1mm}
\begin{picture}(28, 15)(0,0)
\qbezier(0, 0)(13, 35)(25, 0)
\end{picture}\right)= \hat{C}, \quad
&\CF\left(\setlength{\unitlength}{0.1mm}
\begin{picture}(28, 15)(0,15)
\qbezier(0, 35)(13, 0)(25, 35)
\end{picture}\right)=\check{C}.
\end{aligned}
\]
\end{theorem}

\begin{remark}\label{rem:funct-F}
\begin{enumerate}
\item 
In \cite{LZ17} we constructed a full tensor functor from $\CB(\sdim)$ to the full subcategory of representations of $G={\rm OSp}(V; \omega)$ with objects $V^{\ot r}$ for all $r\in\N$. Note that 
\[\Hom_G(V^{\ot r},  V^{\ot s})= \left\{ \varphi\in \Hom_{\U(\fg)}(V^{\ot r},  V^{\ot s})\mid \eta.\varphi = \varphi.\eta, \forall \eta\in G_0\right\}.\]

\item The functor in Theorem \ref{thm:funct} is not full in general. There is a $1$-dimensional $\U(\fg)$-submodule in $V^{\ot r_c}$ for $r_c= m(2n+1)$ spanned by the super Pfaffian $\wt\Omega$ \cite{LZ17-Pf}, which satisfies $\eta.\wt\Omega= \sdet(\eta) \wt\Omega$ for all $\eta\in G_0$, and hence is not in the image of any morphism in $\CF\left(\Hom_{\CB(\sdim)}(0, r_c)\right)=\Hom_G(\C, V^{\ot r_c})$.   This problem already appears in the invariant theory of $\mathfrak{so}_n$, and has been addressed in \cite{LZ17-Ecate} by introducing an ``enhanced Brauer category''.
\end{enumerate}
\end{remark}

\;

\;


Now fix a finite set $\CM\supset\CV$ of $\U(\fg)$-modules where $\CV=\{V\}$.

For any $\U(\fg)$-module $M$, denote by $\mu_M: \U(\fg)\lra \End_\C(M)$ the corresponding representation. 
Let $\CT(\CM, V)$ be the full subcategory of $\U(\fg)$-modules with objects 
$M_1^{\ot s_1} \ot V^{\ot r_1} \ot \dots \ot M_k^{\ot s_k}\ot V^{\ot r_k}$ 
for $k, r_i, s_i\in\N$ and $M_i\in\CM\setminus \{V\}$. 
Then $\CT(\CM, V)$ has a monoidal structure defined by the usual tensor product 
of $\U(\fg)$-modules and homomorphisms. If $V\in\CM$ is fixed, we write $\CT(\CM)$ for $\CT(\CM,V)$.

The monoidal functor constructed in Theorem \ref{thm:main} in the present case reduces to
\beq\label{eq:SFM}
\SF: \HB(\sdim)\lra \CT(\CM),  
\eeq
with the following properties
\[
\baln
&\SF(I^M)=\id_{M}, \quad \SF(X_{M M'}) = \tau_{M M'},  \\
&\SF(\BH(M, M'))=(\mu_{M}\ot\mu_{M'})(t),  \quad \forall M, M'\in\CM,  \\
&  \SF(\cup^V)=\CF\left(\setlength{\unitlength}{0.1mm}
\begin{picture}(28, 15)(0,0)
\qbezier(0, 0)(13, 35)(25, 0)
\end{picture}\right), \quad
\quad \SF(\cap_V)=\CF\left(\setlength{\unitlength}{0.1mm}
\begin{picture}(28, 15)(0,15)
\qbezier(0, 35)(13, 0)(25, 35)
\end{picture}\right), 
\ealn
\]
where $\CF: \CB(\sdim)\lra \CT(V)$ is defined in Theorem \ref{thm:funct}. 
Note in particular that for the elements $\vartheta_j(r)$ defined by \eqref{eq:vartheta}, we have 
\[
\CF(\vartheta_j(r))=\left(\mu_{M\ot V^{\ot j}}\ot \mu\right)(t) = \left(\mu_M\ot \mu^{\ot (j+1)}\right)(\Delta^{(j-1)}\ot\id)(t). 
\]

We have the following result.
\begin{theorem}\label{thm:a-funct} Retain the notation above. 
The monoidal functor \eqref{eq:SFM} factors through $\wh\HB(\sdim)$, that is, 
there is a unique monoidal functor $\wh\SF: \wh\HB(\sdim)$ $\lra \CT(\CM)$, which renders the following diagram of monoidal categories and monoidal functors commutative, 
\beq\label{eq:tensor-funct}
\begin{tikzcd}
\HB(\sdim)  \arrow[d] 
\arrow[r, "\SF"] & \CT(\CM)\\
\wh\HB(\sdim), \arrow[ur, "\wh\SF"'] & 
\end{tikzcd}
\eeq
where the downward arrow is the quotient functor in Definition \ref{def:mpolar}. 
\end{theorem}
\begin{proof}
Recall that $\wh\HB(\sdim)=\HB(\sdim)/\CJ$ with the tensor ideal $\CJ$ generated by $\BH(V, V)- H$,  where $H= X(V, V) -\cup^V\circ\cap_V$. Thus it suffices to show  that $\SF(\BH(V, V)) = \CF(H)$ in order to prove the theorem. Clearly $\CF(H)= \tau - e$. By Theorem \ref{thm:H-formula}, we have
\beq
\SF(\BH(V, V)) = (\mu\ot \mu)(t) =   \tau - e. 
\eeq
This proves the theorem. 
\end{proof}

If $\CM=\{M, V\}$ where $M\ne V$ is an arbitrary $\U(\fg)$-module, the category $\wh\HB(\sdim)$ is the multi-polar Brauer category $\MPB(\sdim)$. 
Recall that the polar Brauer category $\PB(\sdim)$ is the full subcategory of $\MPB(\sdim)$ 
with  objects $(M, \underbrace{V, \dots, V}_r)$ for $r\in\N$. 
The monoidal functor $\wh\SF: \MPB(\sdim)\lra \CT(M, V)$ of Theorem \ref{thm:a-funct}
maps $\PB(\sdim)$ to the full subcategory category $\CT_M(V)$ of $\CT(\CM)$ whose objects are the
modules $M\ot V^{\ot r}$ for all $r\in\N$. 
We denote the restriction  to $\PB(\sdim)$ of the functor $\wh\SF$ by 
\beq\label{eq:FM}
\CF_M: \PB(\sdim)\lra \CT_M(V). 
\eeq

Note that the usual tensor product bi-functor $\ot: \CT_M(V)\times \CT(V)\lra \CT_M(V)$ equips $\CT_M(V)$ with the structure of a right module category over the monoidal category $\CT(V)$. 
Recall from Lemma \ref{lem:mod-cat} that the monoidal structure of $\MPB(\delta)$ leads to
a bi-functor 
$
\ot: \PB(\delta)\times \CB(\delta)\lra \PB(\delta),
$
which makes $\PB(\delta)$ a right module category over the monoidal category $\CB(\delta)$.  

\begin{lemma}\label{lem:mod-T}
The following diagram commutes.
\[
\begin{tikzcd}
\PB(\sdim)\times \CB(\sdim)  \arrow[d, "\ot"' pos=0.43]
					\arrow[r, "\CF_M\times\CF"] & \CT_M(V)\times \CT(V) \arrow[d, "\ot"]\\
\PB(\sdim) \arrow[r, "\CF_M"'] & \CT_M(V). 
\end{tikzcd}
\]
\end{lemma}

\begin{remark}\label{rmk:compare}
If $\fg$ is the orthogonal or symplectic Lie algebra,
we can compare our functor \eqref{eq:FM} with the monoidal functor  
$\Psi: \AB(\sdim)\lra \mathcal{E}nd(\text{$\fg$-Mod})$ constructed in \cite{RSo}, 
where $\mathcal{E}nd(\text{$\fg$-Mod})$ is the endofunctor category of $\fg$-Mod. 
For any morphism $\BB$ in $\AB(\sdim)$, the component $\Psi(\BB)_{\wt{M}}$ 
of the natural transformation $\Psi(\BB)$ 
on an arbitrary $\fg$-module $\wt{M}$ is a morphism in the category $\CT_{\wt{M}}(V)$.  
In particular, $\Psi_{\wt{M}}\left(\begin{picture}(10, 10)(-5,2)
\put(0, 0){\line(0, 1){10}}
\put(-2, 3){\tiny$\bullet$}
\end{picture}\right)=(\mu_{\wt{M}}\ot \mu)(t)$, 
and $\Psi_{\wt{M}}\left(\begin{picture}(40, 10)(-35,2)
\put(-30, 0){\line(0, 1){10}}
\put(-25, 5){\tiny$\dots$}
\put(-23, -1){\tiny$r$}
\put(-10, 0){\line(0, 1){10}}
\put(0, 0){\line(0, 1){10}}
\put(-2, 3){\tiny$\bullet$}
\end{picture}\right)=(\mu_{\wt{M}\ot V^{\ot r}}\ot \mu)(t)$.
Thus in order to extract from \cite{RSo} the functor $\CF_M: \PB(\delta)\lr\CT_M(V)\subset$  $\fg$-Mod, which is the ultimate object of interest, 
one needs to consider $\Psi_{\wt{M}}$ with $\wt{M}=M\ot V^{\ot r}$ for all $r$.
\end{remark}

\section{Applications to universal enveloping superalgebras}\label{sect:Ug}

We apply the diagrammatics of $\PB(\delta)$ to study the universal enveloping superalgebra $\U(\osp(V; \omega))$ of the orthosymplectic Lie superalgebra. We will develop explicit descriptions of 
the centre of $\U(\osp(V; \omega))$, and give a categorical derivation of characteristic identities \cite[\S 4.10]{B}, \cite{BG, G} for the orthogonal and symplectic Lie algebras. 

A characteristic identity for a Lie algebra  $\fg$ is in the spirit of the Cayley-Hamilton theorem but for a specific type of square matrix with entries in $\U(\fg)$. It says that the matrix satisfies a monic polynomial equation over the centre of $\U(\fg)$. Generalisations of characteristic identities to quantum groups \cite{GZB}  are also known. 

This section has an intimate connection with, and provides a categorical approach to, ``quantum family algebras'' \cite{Ki1, Ki2} (see Remark \ref{rmk:family} below).

We retain the notation of the last section. Since we shall deal with $\PB(\delta)$ exclusively in this section, we also adopt the following {\bf convention}.  
Let $PB_r(\delta)=\End_{\PB(\delta)}(r)$ and $ E_{M, r}(V)= \End_{\fg}(M\ot V^{\ot r})$. Then the restriction of the functor
$\CF_M$ to $PB_r(\sdim)$ leads to the following representation of the endomorphism algebra 
\beq\label{eq:F-prime-r}
\CF_M|_r:  PB_r(\sdim)\lra E_{M, r}(V).
\eeq

\subsection{Construction of the centre of $\U(\osp(V;\omega))$}\label{sect:centre-construct}

We use the notation of Section \ref{sect:osp},  and identify $G=\OSp(V; \omega)$ with the Harish-Chandra pair $(G_0, \U(\fg))$.  We note that any $\eta\in G_0$ is of the form $\eta = \begin{pmatrix}\alpha & 0\\ 0 & \beta\end{pmatrix}$ with $\alpha\in {\rm O}(V_{\bar 0}; \omega_{\bar 0})$ and $\beta \in {\rm Sp}(V_{\bar 1}; \omega_{\bar1})$. Hence the Berezinian $\sdet(\eta)$ is equal  to $\det(\alpha)\det(\beta)^{-1}=\det(\alpha)$, which is $\pm 1$. 

Both $G$ and $\fg$ act on $\fg$ through their respective adjoint actions, and these actions extend to  actions on
the universal enveloping superalgebra $\U(\fg)$. For any $X\in\fg$ and $g\in G_0$, we denote their actions on any $u\in\U(\fg)$ by  $ X.u$ and $g.u$ respectively. 
The algebraic structure of $\U(\fg)$ is preserved by these actions, in the sense that, for any $u, v\in\U(\fg)$, 
\beq
X.(u v) = (X.u) v + (-1)^{[X][u]} u (X.v), \quad g.(u v) =  (g.u) (g.v). 
\eeq

Denote by $Z(\U(\fg))$ the centre of $\U(\fg)$, 
which coincides with the subalgebra $\U(\fg)^\fg = \{z\in\U(\fg)\mid X.z=0, \forall X\in\fg\}$ of $\fg$-invariants. 
We also have the subalgebra of $G$-invariants defined by 
$\U(\fg)^G = \{z\in\U(\fg)^\fg\mid \eta.z=z, \forall \eta\in G_0\}$.  
The relationship between $G$-invariants and $\fg$-invariants was investigated in \cite{LZ17-Pf}. 
Consider the following subspace  of $\U(\fg)^\fg$. 
\[
\U(\fg)^{G, \sdet} = \{z\in\U(\fg)^\fg\mid \eta.z=\sdet(\eta) z, \ \forall \eta\in G_0\}.
\]
Then for all  $\alpha, \beta\in \U(\fg)^{G, \sdet}$, we have 
$\alpha\beta\in \U(\fg)^G$. Furthermore, 
\[
\U(\fg)^\fg=\U(\fg)^G\bigoplus \U(\fg)^{G, \sdet}, \quad \text{if $\{\eta\in G_0\mid \sdet(\eta)=-1\}\ne\emptyset$}.
\]

\begin{lemma}\label{lem:odd}
If $\dim(V_{\bar 0})$ is odd or $0$, then $Z(\U(\fg))=\U(\fg)^G$. 
\begin{proof} 
If $\dim(V_{\bar 0})=0$, we have $G=G_0={\rm Sp}_{2n}$, and hence $\sdet(\eta)=1$ for all $\eta\in G_0$. In this case,  $Z(\U(\fg))=\U(\fg)^\fg=\U(\fg)^G$.

If $\eta_0\in \{\eta\in G_0\mid \sdet(\eta)=-1\}\ne\emptyset$,  then 
$
\U(\fg)^{G, \sdet} = \{z\in \U(\fg)^\fg \mid \eta_0.z=- z\}, 
$
independently of the choice of such an $\eta_0$. Assume that $\dim(V_{\bar 0})$ is odd.  
Let $\eta_0\in G_0$ be the element $\eta_0=-id_V$. 
Then clearly $\sdet(g_0)=-1$, and $\eta_0. X=X$ for all $X\in\fg$.  Hence $
\U(\fg)^{G, \sdet} =0$, and thus $Z(\U(\fg))= \U(\fg)^{G}$. 
\end{proof}
\end{lemma}

Since $G$ acts naturally on $V$, and hence on  $V^{\ot r}$ for any $r$, it follows that $G$ acts 
on $\Hom_\C(V^{\ot r}, V^{\ot s})$  ($r, s\in \N$ fixed).  If $g\in G_0$ and $X\in \fg$, their action on $\varphi\in \Hom_\C(V^{\ot r}, V^{\ot s})$
 ($r, s\in \N$ fixed) is given by 
\[
\baln
& (X.\varphi)({\bf v})= X.(\varphi({\bf v}))- (-1)^{[X][\varphi]} \varphi(X.{\bf v}), \\
& (g.\varphi)({\bf v})= g. (\varphi( g^{-1}.{\bf v})), \quad \forall {\bf v}\in V^{\ot r}. 
\ealn
\]
This $G$-action extends to $\U(\fg)\ot \Hom_\C(V^{\ot r}, V^{\ot s})$ as follows. For any $X\in\fg$, $g\in G_0$, 
$u\in \U(\fg)$, and $\varphi\in \Hom_\C(V^{\ot r}, V^{\ot s})$, 
\beq\label{eq:act-fam}
X.(u\ot \varphi) = X.u \ot \varphi + (-1)^{[X][u]} u\ot X.\varphi, \quad
g.(u\ot \varphi) = g. u \ot g.\varphi. 
\eeq
Now it is evident that $\U(\fg)\ot \Hom_\C(V^{\ot r}, V^{\ot s})\ot \Hom_\C(V^{\ot r'}, V^{\ot r'})$ also naturally acquires a $G$-module structure.  

For $u, u'\in\U(\fg)$, $\varphi\in \Hom_\C(V^{\ot r}, V^{\ot r'})$ and $\psi \in \Hom_\C(V^{\ot r'}, V^{\ot s})$, 
 we have the composition $(u'\ot\psi)(u\ot \varphi)=(-1)^{[u][\psi]}u' u \ot \psi\varphi$. One checks easily that
the $G$-action defined by \eqref{eq:act-fam} respects this composition.


Let us write $\U=\U(\fg)$, and consider it as a left $\U$-module under multiplication. The functor $\CF_M$ for $M=\U$ then becomes 
\beq
&&\CF_\U: \PB(\sdim)\lra \CT_\U(V) ,
\eeq  
where $\CF_\U(\I )=1\in\U$ and $\CF_\U(\BH)=(\id_\U\ot \mu)(t)$.

Clearly $\CF_\U(\BH)$ is $G$-invariant with respect to the action \eqref{eq:act-fam}, 
as is $\CF(D)$ for any morphism $D$ of $\CB(\sdim)$.   
As $\I , \BH$ and all $D$ generate the space of morphisms of $\PB(\sdim)$, 
\beq\label{eq:Hom-G-inv}
\CF_\U(\Hom_{\PB(\sdim)}(r, s)) \subseteq \Hom_\C(\U\ot V^{\ot r}, \U\ot V^{\ot s})^G. 
\eeq  

We are now in a position to prove the following crucial result.
\begin{theorem}\label{eq:key-centre}
The identity element together with $\{\CF_\U(Z_{2\ell})\mid \ell\ge 1\}$  generate $\U(\fg)^G$. Further, 
if $\dim(V_{\bar 0})$ is odd or $0$, these elements generate $Z(\U(\fg))$. 
\end{theorem}
\begin{proof} 
We begin by laying some groundwork for the proof. 
For any $G$-module $M$,  let $T(M)=\oplus_{r\ge 0} M^{\ot r}$ be the tensor algebra on $M$. 
For each $r$, there is the usual semi-simple $\Sym_r$ action on $M^{\ot r}$, which commutes with the $G$-action. 
Let $\Sigma_{\pm 1}=\sum_{\sigma\in\Sym_r}(\mp 1)^{\ell(\sigma)}\sigma$, where $\ell(\sigma)$ is the length of $\sigma$. 
Then $\Sigma_{-1}$ (resp. $\Sigma_{+1}$) is the Young symmetriser associated with the standard Young tableau of shape 
$(r)$ (resp. $(1^r)$). Write $S_{gr}^r (M)=\Sigma_{-1}M^{\ot r}$ and $\wedge_{gr}^r M= \Sigma_{+1} M^{\ot r}$;
these are $G$-module direct summands of $M^{\ot r}$.  Write $S_{gr}(M)=\oplus_r S_{gr}^r M$ and $\wedge_{gr} (M)= \oplus_r \wedge_{gr}^r M$.
These are $G$-module direct summands of the tensor algebra $T(M)=\oplus_r M^{\ot r}$.   

Now consider the natural $G$-module $V$, and the following $G$-module  isomorphism
\[
\begin{aligned}
j:\   &  \wedge^2_{gr}V\stackrel{\sim}\lra \fg, & e_a\wedge e_b \mapsto X_{a b} \quad \forall a, b,
\end{aligned}
\]
where $e_a\wedge e_b =e_a\ot e_b- (-1)^{[a][b]}e_b\ot e_a$, and the $X_{a b}$ are the images in $\U(\fg)$ of the basis elements $J_{ab}$ of $\fg$. 
We extend the map $j$ to a $\U(\fg)$-module isomorphism
$
\Upsilon:  S_{gr}(\wedge^2_{gr} V) \stackrel{\sim}\lra \U(\fg)
$
in the following way. 

For any $r\ge 2$, the symmetric group
$\Sym_r$ acts on $\U(\fg)^{\ot r}$ by permuting the factors with appropriate sign. 
Explicitly, the simple reflections $s_i$ of $\Sym_r$ act by 
$\id_{\U}^{\ot(i-1)}\ot\tau_{\U, \U}\ot \id_{\U}^{\ot(r-i-1)}$, 
where $\tau_{\U, \U}$ is defined by \eqref{eq:tau}. 
As above, let $\Sigma_{-1}\in \C\Sym_r$ be the Young symmetriser associated with the shape $(r)$.
Further, let $\mathfrak{M}_r: \U(\fg)^{\ot r}\lra \U(\fg)$ be the multiplication defined by 
$u_1\ot u_2\ot \dots \ot u_r\mapsto u_1u_2\dots u_r$, for any $u_1, \dots, u_r\in\U(\fg)$. 
Now define $\Upsilon$ by requiring that it map any element $Y_1 Y_2\dots Y_r\in S_{gr}^r(\wedge^2_{gr} V)$ 
to $\frac{1}{r!}{\mathfrak M}_r\Sigma_{-1}\left(j(Y_1)\ot j(Y_2)\ot \dots\ot  j(Y_r)\right)$ in $\U(\fg)$ for any $r$. 

We now prove the theorem.  Clearly, $(S_{gr}(\wedge^2_{gr} V))^G=\sum_{k\ge 0}(S^k_{gr}(\wedge^2_{gr} V))^G$ is $\Z_+$-graded with 
$
(S_{gr}^0(\wedge^2_{gr} V))^G =\C$ and $(S_{gr}^1(\wedge^2_{gr} V))^G=0.
$
It follows from the first fundamental theorem of invariant theory for $G$ \cite{DLZ, LZ17} that the elements 
\[
\wh{T}[\ell]= \sum_{a_1,\dots, a_{\ell-1},  b} (-1)^{\sum_{i=1}^{\ell-1} [a_i]} j^{-1}(X^{a_1}_b)  j^{-1}( X^{a_2}_{a_1})  j^{-1}(X^{a_3}_{a_2})\dots j^{-1}(X^b_{a_{\ell-1}}), \quad \ell\ge 2,
\]
together with the identity generate $(S_{gr}(\wedge^2_{gr} V))^G$. 
A simple calculation yields that 
\beq\label{eq:FZ}
\CF_\U(Z_\ell)= \sum_{a_1,\dots, a_{\ell-1},  b=1}^{m+2n} (-1)^{\sum_{i=1}^{\ell-1} [a_i]} X^{a_1}_b  X^{a_2}_{a_1} X^{a_3}_{a_2}\dots X^b_{a_{\ell-1}}. 
\eeq
Thus for any $\ell$, we have 
\beq\label{eq:Upsilon-1}
&\Upsilon^{-1}(\CF_\U(Z_\ell))=\wh{T}[\ell]\mod (S_{gr}^{\ell-1}(\wedge^2_{gr} V))^G. 
\eeq
In particular, $\Upsilon^{-1}(\CF_\U(Z_2))=\wh{T}[2]$,  and by Lemma \ref{lem:central}, 
$\Upsilon^{-1}(\CF_\U(Z_3))=\wh{T}[3]$.
It is also evident that 
\beq\label{eq:Upsilon-2}
\Upsilon^{-1}(\CF_\U(Z_k Z_\ell))=\wh{T}[k] \wh{T}[\ell]\mod (S_{gr}^{k+\ell-1}(\wedge^2_{gr} V))^G. 
\eeq
Using the relations \eqref{eq:Upsilon-1} and \eqref{eq:Upsilon-2},  one now shows by induction on degree in $(S_{gr}(\wedge^2_{gr} V))^G$ that  the identity and
$\Upsilon^{-1}(\CF_\U(Z_\ell))$ for all $\ell\ge 2$ generate  $(S_{gr}(\wedge^2_{gr} V))^G$.
Since $\Upsilon^{-1}(\U(\fg)^G)=(S_{gr}(\wedge^2_{gr} V))^G$, the identity and the elements   $\CF_\U(Z_\ell)$ for $\ell\ge 2$ together generate $\U(\fg)^G$. 
As the elements $Z_{2\ell}$ for all $\ell\ge 1$ generate $\Hom_{\PB(\sdim)}(0, 0)$ by Theorem \ref{thm:Hom01}(2), 
we obtain the first statement of the theorem. 

The second statement of the theorem now follows from Lemma \ref{lem:odd}.  This completes the proof of the Theorem.  
\end{proof}

In summary, we have 

\begin{corollary}\label{cor:part-trace}  
For any $r, s$, the functor $\CF_\U$ sends $\Hom_{\PB(\sdim)}(r, s)$ to the subspace $\left(\U\ot_\C \Hom_\C\left(V^{\ot r}, V^{\ot s}\right)\right)^G$ of  $G$-invariants. In particular,  
\beq\label{eq:G-inv-U}
\CF_\U(\End_{\PB(\sdim)}(0))=\U(\fg)^G. 
\eeq
Further,  if $\dim(V_{\bar 0})$ is odd or $0$, then 
$Z(\U(\fg))=\CF_\U(\End_{\PB(\sdim)}(0))$.
\end{corollary}

\begin{proof}
The first statement follows from \eqref{eq:Hom-G-inv}, and \eqref{eq:G-inv-U} follows from Theorem \ref{eq:key-centre}. 
If $\dim(V_{\bar 0})$ is odd or $0$, then $\U(\fg)^\fg=\U(\fg)^G$ by Lemma \ref{lem:odd}. 
This completes the proof of the Corollary.
\end{proof}

\begin{remark}\label{rmk:family}
The quantum family algebra associated with a simple $\U$-module $L$ 
is defined by $\CR_L:=\left(\U\ot \End_\C(L)\right)^{G}$ \cite{Ki1, Ki2, Ko}. 
Relevant to the study in this section is the quantum family algebra $\CR_V$.  
Clearly $\im(\CF_\U|_1)\subseteq \CR_V$ by Corollary \ref{cor:part-trace}, and
it will be very interesting to determine whether $\im(\CF_\U|_1)= \CR_V$. 
A central issue in the study of this quantum family algebra is to understand properties of 
$E=(\id_\U\ot \mu)(t)$,  
including a characteristic identity for $E$ in $\CR_V$ \cite{Rn}.
In this section, we provide a diagram categorical approach to these questions. 
\end{remark}

\subsection{Proof of Theorem \ref{thm:Hom01} (4)}
To close this section, we give the promised completion the proof of Theorem \ref{thm:Hom01}. 
Let us recall some well known facts about central elements of the universal enveloping algebra $\U(\fso_m)$ of the orthogonal Lie algebra $\fso_m$.
\begin{remark}\label{rmk:so-Casimir}
By taking $V=\C^m$ to be purely even and equipped with the bilinear form $\omega(e^i, e^j)=\delta_{i j}$ for all $i, j$, we have the orthogonal Lie algebra $\fso_m$, which is spanned by $J_{i j} = E^i_j-E^j_i$ for $i, j=1, 2, \dots, m$, where $E_j^i$ are the $m\times m$ matrix units.  
Denote by  $X_{i j}$ the image  of $J_{i j}$ in $\U(\fso_m)$. Then it follows immediately from the first fundamental theorem of invariant theory for the orthogonal group
 $O_m(\C)$ that as a subalgebra of $\U(\fso_m)$, $\U(\fso_m)^{O_m(\C)}$ is generated by 
\[
I(r)=\sum_{i_1, i_2, \dots, i_r} X_{i_1 i_2} X_{i_2 i_3} \dots X_{i_r, i_1},  \quad r\ge 2.
\]
It is well-known that the central elements $I(2j)$ for $j=1, 2, \dots, \begin{bmatrix}\frac{m}{2}\end{bmatrix}$ are algebraically independent  and generate $\U(\fso_m)^{O_m(\C)}$. 
These facts are easily proved following the line of reasoning suggested in \cite[\S11]{Bin}.

It is also straightforward to verify that $I(2k+1)$ may be explicitly written as a polynomial in the $I(2j)$. 
This also follows from the more general statement in Theorem \ref{thm:Z-odd}(2) (which is proved independently) by using  the functor $\CF_\U$ for $\fso_m$.

\end{remark}

\begin{proof}[Proof of Theorem \ref{thm:Hom01}(4)]
Retain the notation in Remark \ref{rmk:so-Casimir}. 
Let $\fg=\fso_m$ for any $m\ge 3$, so that $\U(\fg)=\U(\fso_m)$. Using \eqref{eq:FZ},  we obtain
\[
\CF_{\U}(Z_{2j})=I(2j), \quad \forall j\ge 1. 
\]
Now the $I(2j)$ for $j=1, 2, \dots, \left[\frac{m}{2}\right]$ are algebraically independent. Thus the elements 
$Z_{2j}$ for all $j=1, 2, \dots, \left[\frac{m}{2}\right]$ must be algebraically independent, and this is true for any $m\ge 3$. 
This proves Theorem  \ref{thm:Hom01}(4).
\end{proof}

\subsection{Characteristic identities}\label{sect:char-id}

In general, we have $\ker(\CF_\U)$ $\ne 0$, and non-zero elements of $\ker(\CF_\U|_r)$ lead to interesting identities in $\left(\U\ot\End_\C(V^{\ot r})\right)^G$. Bracken and Green \cite[\S 4.10]{B}, 
\cite{BG, G} discovered certain `characteristic identities' for Lie algebras in $\left(\U\ot\End_\C(V)\right)^G$, which may be interpreted this way.   
Such characteristic identities have been widely studied in the mathematical physics literature (see \cite{IWD} for a brief review and further references).

Throughout this section, $V$ is either $\C^{m|0}$ or $\C^{0|2n}$, and thus $\fg$ is $\mathfrak{so}_m$ or $\mathfrak{sp}_{2n}$.  Write $d=\dim(V)$.

\subsubsection{Characteristic identities for the orthogonal and symplectic Lie algebras}
We begin by stating precisely what is meant by the term ``characteristic identity'' \cite{B, BG, G} for $\fg$ (also see \cite{Go}). 
Note first that the algebra $\U(\fg)\ot\End_\C(V)$ may be thought of as $\rm{Mat}_{\dim V}(\U)$, the algebra of matrices of size $\dim V$
with entries in $\U$. It is therefore a module over $Z(\U(\fg))$. A {\bf characteristic identity} for the element $A\in\rm{Mat}_{\dim V}(\U)$ is a monic polynomial 
$Q(\alpha)\in Z(\U(\fg))[\alpha]$, where $\alpha$ is an indeterminate, such that $Q(A)=0$. A specific instance of this is as follows.

Let $E$ be the element of  $\left(\U(\fg)\ot\End_\C(V)\right)^G$  
given by $E=(\id_{\U(\fg)}\ot\mu)(t)$, where, as usual, $\mu$ denotes the $\fg$-action on $V$ and $t$ is the tempered Casimir element. 
Then $E$ satisfies a characteristic identity if there are explicitly known central elements $Q_i\in Z(\U(\fg))$ such that 
\beq\label{eq:char-2}
E^d + Q_1E^{d-1} + Q_2E^{d-2}+\dots+Q_d=0. 
\eeq
Here the $Q_i$ in \eqref{eq:char-2} are known explicitly  in the sense that their images under the Harish-Chandra isomorphism are explicitly given, 
as explained in the proof of Theorem \ref{thm:char-id}. 

\begin{lemma}\label{lem:even-coef}
The central elements $Q_i$ appearing in any characteristic identity \eqref{eq:char-2} all belong to $\U(\fg)^G$. 
\end{lemma}
\begin{proof}
By Lemma \ref{lem:odd}, it is always true that $Q_i\in \U(\fg)^G$ for all $i$ if $\fg=\mathfrak{sp}_{2n}$ or $\mathfrak{so}_m$ for odd $m$. 
If $\fg=\mathfrak{so}_m$ with $m$ even, we let $\eta_0\in {\rm O}_m$ be such that $\det(\eta_0)=-1$. We can express $Q_i$ as $Q_i=Q_{i, +}+Q_{i, -}$ 
with $Q_{i, \pm}\in  Z(\U(\fg))$ such that $\eta_0.Q_{i, \pm}=\pm Q_{i, \pm}$, that is, $Q_{i, +}\in \U(\fg)^G$, and $Q_{i, -} \in \U(\fg)^{G, \sdet}$.
Since $\eta_0. E= E$ and $\eta_0.Q_i=Q_{i, +}-Q_{i, -}$, by combining \eqref{eq:char-2} with
\[
E^d + \eta_0.Q_1 E^{d-1} + \eta_0.Q_2 E^{d-2}+\dots+\eta_0.Q_d =0,  
\]
we obtain 
$
E^d + Q_{1, +}E^{d-1} + Q_{2, +}E^{d-2}+\dots+Q_{d, +}=0.
$
This proves the lemma. 
\end{proof}

\subsubsection{Categorification of characteristic identities}

The following theorem gives a categorical derivation of characteristic identities for  $\mathfrak{so}_m$ and $\mathfrak{sp}_{2n}$. It could be thought of as a diagrammatic, or abstract characteristic identity.

\begin{theorem} \label{thm:char-id} Maintain the above notation, and write $d=\dim(V)$. There exist elements $\wh{Q}_i\in \End_{\PB(\sdim)}(0)$, for $i=1, 2, \dots, d=\dim(V)$, such that the element 
\[
\wh\Q:=\BH^d + \wh{\Q}_1\BH^{d-1} + \wh{\Q}_2\BH^{d-2} +\dots + \wh{\Q}_{d}
\]
belongs to $\ker(\CF_\U|_1)$, where $\wh{\Q}_i = \wh{Q}_i\ot I$.
Furthermore, the characteristic identity \eqref{eq:char-2} may simply be stated in the present context as  
$
\CF_\U(\wh\Q)=0. 
$
\end{theorem}

\begin{proof} This proof is partly based on \cite{Go}. 
Denote by $M_\lambda$ the Verma module with highest weight $\lambda$, and by $L_\lambda$ the corresponding 
simple quotient module. Then $M_\lambda\ot V= \sum_{\lambda_i}M_{\lambda+\lambda_i}$, 
where the $\lambda_i$ for $i=1, ,2, \dots, d$ are the weights of a basis of weight vectors of $V$. Now 
$E$ has an eigenvalue 
\[
e_i(\lambda)=\frac{1}{2}\left((\lambda+\lambda_i+2\rho, \lambda+\lambda_i)- (\lambda+2\rho, \lambda)-\chi_V(C)\right),
\]
in each submodule $M_{\lambda+\lambda_i}$, where $\chi_V(C)$ is the eigenvalue of $C$ in $V$.  Let 
\[
R_\lambda=\prod_{i=1}^{d}\left(E- e_i(\lambda)\right). 
\]
Then $(\mu_{M_\lambda}\ot\id)(R_\lambda)=0$ as an element of $\End_{\U}(M_\lambda\ot V)$. This implies 
\beq\label{eq:prod-form}
(\mu_{L_\lambda}\ot\id)(R_\lambda)=0, \quad \forall \lambda. 
\eeq

Now  expand $R_\lambda$ as a polynomial in $E$, obtaining  
\beq
R_\lambda=E^d + P_1(\lambda) E^{d-1} + P_2 (\lambda) E^{d-2}+\dots+ P_d (\lambda), 
\eeq
where the $P_i(\lambda)$ are polynomials in the functions $e_j(\lambda)$ of $\lambda$, 
which are essentially the elementary symmetric functions in the $e_i(\lambda)$, with suitable signs; e.g., $P_1(\lambda)=-\sum_i e_i(\lambda)$, 
$P_2(\lambda)=\sum_{i\ne j} e_i(\lambda)e_j(\lambda)$,  etc. 
Write $E(\lambda)=(\mu_{L_\lambda}\ot\id)(E)$. Then \eqref{eq:prod-form} is equivalent to 
\beq\label{eq:char-1}
E(\lambda)^d + P_1(\lambda) E(\lambda)^{d-1} + P_2 (\lambda) E(\lambda)^{d-2}+\dots+ P_d (\lambda) =0, \quad \forall \lambda.  
\eeq

Using the fact that the set of weights of $V$ is invariant 
with respect to the usual action of the Weyl group $W$ of $\fg$, one easily shows, as in \cite{Go} that the 
$P_i(\lambda)$ are all invariant under the dot-action of $W$. Hence by Harish-Chandra's isomorphism, there exist elements $Q_i\in Z(\U)$ such that 
$\mu_{M_\lambda}(Q_i)=P_i(\lambda)$ for all $i$.  We define 
\beq
R=E^d + Q_1 E^{d-1} + Q_2 E^{d-2}+\dots+ Q_d,  
\eeq
which is equal to the left hand side of \eqref{eq:char-2}.  Now equation \eqref{eq:char-1} becomes 
\beq
(\mu_{L_\lambda}\ot\id)(R)=0, \quad \forall \lambda.
\eeq

Note that $R$ is a square matrix with entries in $\U(\fg)$. The above equation states that all entries of $R$ vanish in the irreducible representations $\mu_{L_\lambda}$ for all $\lambda$. This implies that $R=0$, which is equation \eqref{eq:char-2}. 

By Lemma \ref{lem:even-coef} and Corollary \ref{cor:part-trace}, there exist $\wh{Q}_i\in \Hom_{\PB(\sdim)}(0, 0)$ for $1\le i\le d$ such that $Q_i= \CF_\U(\wh{Q}_i)$,  and hence $Q_i\ot \id_V= \CF_\U(\wh{\Q}_i)$. Since $E=\CF_\U(\BH)$, the left hand side of  \eqref{eq:char-2} is the image of the element $\wh Q$ given in Theorem \ref{thm:char-id} under $\CF_\U$. This proves the theorem. 
\end{proof}

\begin{example}\label{eg:sp2} When $V=\C^{0|2}$, we are dealing with the Lie algebra $\mathfrak{sp}_2$. 
Then 
\[
\wh\Q=\BH^2 - 2\BH + \frac{1}{2}Z_2\ot I 
\]
belongs to the kernel of $\CF_\U|_1$, and  
the characteristic identity in this case is given by $\CF_\U(\wh\Q)=0$.

To be more explicit, let $X, Y, T$ be the standard generators of $\mathfrak{sp}_2=\fsl_2$ with 
\[
[T, X]=2X,  \quad [T, Y]=-2Y, \quad [X, Y]=T.
\]  
Then  the quadratic Casimir operator \eqref{eq:C-formula} is given by 
\beq\label{eq:C-sp2}
C= - 2 \left(\frac{H^2}{2}+ X Y + Y X\right)
\eeq
in the current notation. As $V=\C^{0|2}$ has highest weight $1$, the eigenvalue of $C$ in $V$ is equal to $-2\left(\frac{1}{2}+1\right)= - 3= \sdim-1$.  
The standard generators $X,Y,T$ act on $V$ with respect to the standard basis as the matrices $\begin{pmatrix} 0 & 1\\0 & 0 \end{pmatrix}$,
$\begin{pmatrix} 0 & 0\\1 & 0 \end{pmatrix}$ and $\begin{pmatrix} 1 & 0\\0 & -1 \end{pmatrix}$ respectively.

The corresponding tempered Casimir element is  $t=- 2\left( \frac{T\ot T}{2}+ X\ot Y + Y\ot X\right)$. Hence
\[
\begin{aligned}
\CF_\U(\BH)&=(\id_\U\ot\mu)(t)
 =-2 \begin{pmatrix} \frac{T}{2} & Y\\  X & -  \frac{T}{2} \end{pmatrix},  \\
\CF_\U(Z_2)&=\text{str} (\CF_\U(\BH)^2)=2C.
\end{aligned}
\]
\end{example}

\subsubsection{Comments on characteristic identities for Lie superalgebras and quantum groups}\label{sect:s-char-id}
There are discussions of characteristic identities for Lie superalgebras in the literature (see, e.g., \cite{JG}), with most being ``local'' results like \eqref{eq:char-1}
for any finite dimensional simple module $L_\lambda$ (that is, replacing $M_\lambda$ by $L_\lambda$). 
It appears to be difficult to generalise the method of \cite{Go} to lift \eqref{eq:char-1} for all $\lambda$ to the identity \eqref{eq:char-2} 
in $\U\ot\End_\C(V)$ for $\osp_{m|2n}$ and $\gl_{m|n}$ for general $m, n$. 
The problem lies in the fact that the $P_i(\lambda)$ for atypical $\lambda$ are not invariant 
under translation by scalar multiples of atypical roots $\gamma$, that is, 
$P_i(\lambda)\ne P_i(\lambda+c\gamma)$ in general for nonzero $c\in\C$ and isotropic root $\gamma$ such that $(\lambda+\rho, \gamma)=0$. By the generalised Harish-Chandra isomorphism for Lie superalgebras,  there exist no central elements 
of $\U$ whose eigenvalues are $P_i(\lambda)$ in $M_\lambda$ for all $\lambda$.  This problem does not exist for  $\osp_{1|2n}$, as this Lie superalgebra has no atypical weights.  

However, any non-zero element $\wh\Q\in \ker(\CF_\U|_1)$ would lead to an identity
$\CF_\U(\wh\Q)=0$  in $\left(\U\ot\End_\C(V)\right)^G$, which might be considered a generalisation of the characteristic identities above. 
For the purpose of constructing such identities, we need only consider the generators of $\ker(\CF_\U|_1)$ as a $2$-sided ideal in $\Hom_{\PB(\sdim)}(1, 1)$.

Characteristic identities for quantum groups were investigated in \cite{GZB} (also see \cite{ZGB}). 
It is quite straightforward to use the restricted coloured tangle category (with one pole only) introduced
 in \cite[\S3.1, \S 3.2]{ILZ} to give a categorical treatment of the quantum characteristic identities analogous to what has been presented here.

\section{Applications to category $\mathcal O$ of $\fsp_2$}\label{sect:sp2}

The material presented in this section may be considered as the classical (i.e., $q\to 1$) limit of of the results in \cite{ILZ}. 
Throughout this section, we assume that $V=\C^{0|2}$ and $\omega=\begin{pmatrix}0 &1\\ -1 & 0\end{pmatrix}$. Thus $\osp(V; \omega)$ is the Lie algebra $\fsp_2(\C)$. Clearly $\sdim=-2$. 

We focus attention on the functor of Theorem \ref{thm:a-funct} in this special case. The following result shows that the functor
$\CF_M$ factors through the polar Temperley-Lieb category in this case.
\begin{lemma} 
For any  $\fsp_2$-module $M$,  the functor $
\CF_M: \PB(-2)\lra \CT_M(V)$ given in Theorem \ref{thm:a-funct} factors through $\ATL(-2)$, that is, 
there exists a unique functor $\CF^{TL}_M: \ATL(-2) \lra \CT_M(V)$ such that the following diagram commutes, 
\[
\begin{tikzcd}
\PB(-2) \arrow[d] \arrow[r, "\CF_M"] & \CT_M(V)\\
\ATL(-2) \arrow[ur, "\CF^{TL}_M"'pos=0.43]. 
\end{tikzcd}
\]
\end{lemma}
\begin{proof} 
Note that $V\ot V=\C^{0|2}\ot\C^{0|2}=L_s \oplus L_a$ as $\fsp_2$-module, where $L_s$ and $L_a$ are simple submodules of dimensions $1$ and $3$ respectively. Thus the permutation $\tau: V\ot V\lra V\ot V$ is given by 
$\tau = -P_a + P_s = 1 + 2 P_s$.  As $\sdim(V)=-2$, the map $e$ in \eqref{eq:e} is given by $e=-2 P_s$ in this case, and hence
$
\tau = 1- e.
$
This shows that $\CF(\Theta)=0$ because $\delta=\sdim(V)=-2$, where $\Theta$ is the element of $\AB(-2)(2,2)$ depicted in Fig. 14. By the definition of $\ATL$, this shows 
that the functor $\CF_M$ factors through $\ATL(-2)$.
\end{proof}

We adopt the same convention for $\PB(\delta)$ in Remark \ref{rmk:no-m}  to the present case by writing  objects $(m, v^r)$ of $\ATL(\delta)$ simply as $r$.

\begin{lemma} Maintain the notation above and fix $\lambda\in\C$. Assume that the $\fsp_2$-module $M$ is either the Verma module $M_\lambda$ with highest weight $\lambda$ or the simple module $L_\lambda$. Then the functor 
$\CF^{TL}_M: \ATL(-2) \lra \CT_M(V)$ factors through $\TLBC(-2, \lambda)$, that is there exists a unique functor 
$\CF^{TLB}_\lambda: \TLBC(-2, \lambda)\lra \CT_M(V)$
such that the following diagram commutes, 
\[
\begin{tikzcd}
\ATL(-2) \arrow[d] \arrow[r, "\CF^{TL}_M"] & \CT_M(V)\\
\TLBC(-2, \lambda) \arrow[ur, "\CF^{TLB}_\lambda"'pos=0.43], 
\end{tikzcd}
\]
where the vertical functor is the quotient functor (see Definition \ref{def:TL}).  
\end{lemma}
\begin{proof}
We have $\CF^{TL}_M(Z_2)=2C: M\lra M$, where $C$ is the quadratic Casimir given in \eqref{eq:C-sp2}, which acts on $M=M_\lambda$ and {\it a fortiori} on $L_\lambda$ 
as multiplication by the scalar $\chi_\lambda(C)= - \lambda(\lambda +2)$. Hence for $\delta=-2$, 
\[
\CF^{TL}_M\left(\frac{Z_2}{\delta}\right)= \lambda(\lambda +2) = -\lambda\left(\frac{\delta-2}{2}-\lambda\right).
\]
This shows that $\CF^{TL}_M$ factors through $\TLBC(-2, \lambda)$.
\end{proof}

Let us write $\Z_{<-1}=\{-2, -3, -4, \dots\}$.  For any given $\lambda\in\C\backslash \Z_{<-1}$, the Verma module $M_\lambda$ is projective in category ${\CO}$ of $\fsp_2$. 

The next theorem is an analogue in our context of \cite[Thm. 4.9]{ILZ}.

\begin{theorem} Maintain the above notation and let $M_\lambda$ be the Verma module for $\fsp_2$ with highest weight $\lambda\in\C\backslash \Z_{<-1}$. Then 
the functor $\CF^{TLB}_\lambda: \TLBC(-2, \lambda)\lra \CT_{M_\lambda}(V)$
is an isomorphism of categories.  
\end{theorem}
\begin{proof} 
It is evident that the restriction of $\CF^{TLB}_\lambda$  to objects is an isomorphism. Since $M_\lambda$ with $\lambda\in\C\backslash \Z_{<-1}$ is projective in category ${\CO}$, we have 
\[
\begin{aligned}
\dim \Hom_{\fsp_2}(M_\lambda \ot V^{\ot r}, M_\lambda \ot V^{\ot s})
&= \dim \Hom_{\fsp_2}(M_\lambda, M_\lambda \ot V^{\ot (r+s)}) \\
&= \dim (V^{\ot (r+s)})_0, \quad \forall r, s, 
\end{aligned}
\]
where $(V^{\ot (r+s)})_0$ is the $0$-weight subspace of $V^{\ot (r+s)}$, which has dimension $\begin{pmatrix}2N \\ N\end{pmatrix}$ if $r+s=2N$ is even, and is $0$ otherwise. 
 It follows  from Theorem \ref{thm:ATL-dim} that for all $r, s$, 
\[
\dim \Hom_{\TLBC(-2, \lambda)}(r, s) =\dim \Hom_{\fsp_2}(M_\lambda \ot V^{\ot r}, M_\lambda \ot V^{\ot s}).
\]
Hence it suffices to show that the restriction of $\CF^{TLB}_\lambda$  to morphisms is injective.  
Only the case with $r+s=2N$ requires proof, and we will adapt the main idea in the treatment of $\U_q(\fsl_2)$ representations in \cite{ILZ} to the present context.  

Write $\CW(r, 2N-r)=\Hom_{\TLBC(-2, \lambda)}(r, 2N-r)$ for all $r$, and $\CW(2N)=\CW(0, 2N)$.
Let $E^0_{2N}$ be the subalgebra of $\CW(2N, 2N)$ spanned by diagrams without connectors, which is isomorphic to the Temperley-Lieb algebra $TL_{2N}(-2)$. Note that 
$\CF^{TLB}_\lambda(E^0_{2N})=\CF(E^0_{2N}) \cong E^0_{2N}$ as algebra. Furthermore, $\CW(2N)$ is a module for this algebra.  
We want to show that $\CF^{TLB}_\lambda(\CW(2N))\cong \CW(2N)$ as $TL_{2N}(-2)$-module.

Consider images of the standard polar Temperley-Lieb diagrams as shown in Figure \ref{fig:ATL2N}, which will be called standard 
polar Temperley-Lieb diagrams of type $B$. They form a basis of $\CW(2N)$. As in the situation of the proof of Theorem \ref{thm:ATL-dim},  
the standard $(0, 2N)$-diagrams with at most $i\le N$ connectors span a $TL_{2N}(-2)$-submodule $F_i\CW(2N)$ of $\CW(2N)$, and 
$\CW_{2i}(2N) = \frac{F_i \CW(2N)}{F_{i-1}\CW(2N)}$ is a simple $TL_{2N}(-2)$-module.  
Hence $\CF^{TLB}_\lambda(\CW_{2i}(2N)) = \frac{\CF^{TLB}_\lambda(F_i \CW(2N))}{\CF^{TLB}_\lambda(F_{i-1}\CW(2N))}\cong \CW_{2i}(2N)$ or $0$.  We show that it is non-zero for all $i$. 

Let $\BD_t$ be the standard $(0, 2N)$-diagram with $t$ connectors as depicted below. 
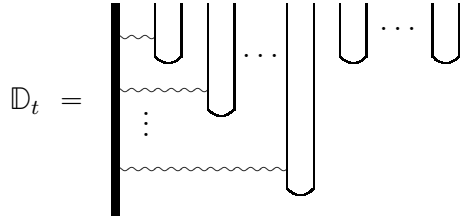
\[
\begin{picture}(180, 80)(-40, 20)
\put(-40, 60){$\BD_t\  =$}
{
\linethickness{1mm}
\put(0, 20){\line(0, 1){80}}
}
\put(0, 90){\uwave{\hspace{5mm}}}

\put(15, 100){\line(0, -1){20}}
\put(25, 100){\line(0, -1){20}}
\qbezier(15, 80)(20, 75)(25, 80)

\put(0, 70){\uwave{\hspace{12mm}}}

\put(35, 100){\line(0, -1){40}}
\put(45, 100){\line(0, -1){40}}
\qbezier(35, 60)(40, 55)(45, 60)

\put(0, 40){\uwave{\hspace{23mm}}}
\put(65, 100){\line(0, -1){70}}
\put(75, 100){\line(0, -1){70}}
\qbezier(65, 30)(70, 25)(75, 30)

\put(48, 80){$\dots$}
\put(10, 50){$\vdots$}

\put(85, 100){\line(0, -1){20}}
\put(95, 100){\line(0, -1){20}}
\qbezier(85, 80)(90, 75)(95, 80)

\put(100, 90){$\dots$}

\put(120, 100){\line(0, -1){20}}
\put(130, 100){\line(0, -1){20}}
\qbezier(120, 80)(125, 75)(130, 80)


\put(135, 25){.}
\end{picture}
\]
Then $\BD_t$ belongs to $F_t\CW(2N)$, but is not in $F_{t-1}\CW(2N)$. 
Postmultiplying $\I_{2t}\ot \underbrace{\cap \dots \cap}_{N-t}$ with $\BD_t$, we obtain the standard $(0, 2t)$-diagram $\wt\BD_t$ shown below
up to a scalar factor. 
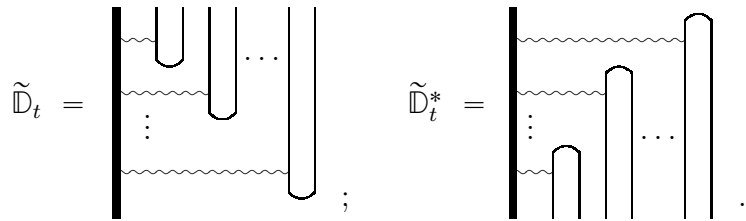
\[
\begin{picture}(150, 80)(-40, 20)
\put(-40, 60){$\wt\BD_t\  =$}
{
\linethickness{1mm}
\put(0, 20){\line(0, 1){80}}
}
\put(0, 90){\uwave{\hspace{5mm}}}

\put(15, 100){\line(0, -1){20}}
\put(25, 100){\line(0, -1){20}}
\qbezier(15, 80)(20, 75)(25, 80)

\put(0, 70){\uwave{\hspace{12mm}}}

\put(35, 100){\line(0, -1){40}}
\put(45, 100){\line(0, -1){40}}
\qbezier(35, 60)(40, 55)(45, 60)

\put(0, 40){\uwave{\hspace{23mm}}}
\put(65, 100){\line(0, -1){70}}
\put(75, 100){\line(0, -1){70}}
\qbezier(65, 30)(70, 25)(75, 30)

\put(48, 80){$\dots$}
\put(10, 50){$\vdots$}

\put(85, 25){;}
\end{picture}
\begin{picture}(150, 80)(-40, 20)
\put(-40, 60){$\wt\BD_t^*\  =$}
{
\linethickness{1mm}
\put(0, 20){\line(0, 1){80}}
}

\put(0, 90){\uwave{\hspace{23mm}}}
\put(0, 70){\uwave{\hspace{12mm}}}
\put(0, 40){\uwave{\hspace{5mm}}}

\put(15, 20){\line(0, 1){25}}
\put(25, 20){\line(0, 1){25}}
\qbezier(15, 45)(20, 50)(25, 45)

\put(35, 20){\line(0, 1){55}}
\put(45, 20){\line(0, 1){55}}
\qbezier(35, 75)(40, 80)(45, 75)

\put(65, 20){\line(0, 1){75}}
\put(75, 20){\line(0, 1){75}}
\qbezier(65, 95)(70, 100)(75, 95)

\put(48, 50){$\dots$}
\put(5, 50){$\vdots$}

\put(85, 25){.}
\end{picture}
\]
Now premultiply the tensor product $\wt{\BD}_t\ot I_{2t}$ with the $(0, 4t)$-diagram 
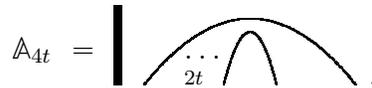
\[
\begin{picture}(150, 40)(-40, -5)
\put(-40, 10){$\A_{4t}\ =$}
{
\linethickness{1mm}
\put(0, 0){\line(0, 1){30}}
}

\qbezier(10,0)(50, 50)(90, 0)
\qbezier(40,0)(50, 40)(60, 0)

\put(25, 10){$\dots$}
\put(25, 0){\tiny$2t$}

\put(95, 0){.}
\end{picture}
\] 
One obtains essentially (up to a non-zero scalar) the $(2t, 0)$-diagram $\wt{\BD}_t^*$ shown above.  
Denote by $m_+$ the highest weight vector of $M_\lambda$.  It is easy to show that 
\beq
\CF^{TLB}_\lambda(\wt{\BD}_t^*)(m_+\ot \underbrace{ v_{-1}\ot\dots\ot v_{-1}}_{2t})= (-2 Y)^t m_+\ne 0. 
\eeq

Let $\BD_{<t}$ be an element in $F_{t-1}\CW(2N)$.  We perform the same operations as above on $\BD_{<t}$ to obtain an element $\wt\BD_{<t}^* \in \Hom_{\TLBC(-2, \lambda)}(2t, 0)$, which can be expressed as a linear combination of $(2t, 0)$-diagrams with no more than $t-1$ connectors. Thus each $(2t, 0)$-diagram in $\wt\BD_{<t}^*$ with non-zero coefficient must have a 
component of the form $\{\cap\}$ which is not connected to the pole via a connector. 
Note that for all $i+j=2t-2$, we have 
$\id_{M_\lambda}\ot \id_V^{\ot i} \ot \CF(\cap)\ot \id_V^{\ot j} (m_+\ot \underbrace{ v_{-1}\ot\dots\ot v_{-1}}_{2t})=0$ because $\omega(v_{-1}, v_{-1})=0$. Thus 
\beq
\CF^{TLB}_\lambda(\wt\BD_{<t}^*)(m_+\ot \underbrace{ v_{-1}\ot\dots\ot v_{-1}}_{2t})=0. 
\eeq

Comparing the two equations above, we conclude that 
$
\CF^{TLB}_\lambda(\CW_{2t}(2N)) \ne 0
$
for all $t$, and hence by cellular theory \cite{GL96}, $\CF^{TLB}_\lambda(\CW_{2t}(2N))$  $\cong \CW_{2t}(2N)
$.
This shows that $\CF^{TLB}_\lambda$ is injective when restricted to morphisms, completing the proof of the theorem. 
\end{proof}

\bigskip

\noindent{\bf Acknowledgements}. 
We thank Tony Bracken, Mark Gould, Alistair Savage, Catherina Stroppel and Yang Zhang for discussions. 





\begin{thebibliography}{9999}

\bibitem{Betal}
Balagovic, M.;  Daugherty, Z.;   Entova--Aizenbud, I.;  Halacheva, I.;   Hennig, J.; Im, M. S.;  Letzter, G.;  Norton, E.;  Serganova, V.;  Stroppel, C. ``The affine VW supercategory''.
 
\bibitem{BE}Benkart, G.;  Elduque, A., ``Cross products, invariants, and centralizers'',  {\sl J. Algebra \bf 500} (2018) 69 - 102.

 \bibitem{Bin} Bincer,  A., {\em Lie Groups and Lie Algebras - A Physicist's Perspective}.  
 University of Oxford press, Oxford, Illustrated edition, 2012.
 
 \bibitem{BCNR}
 J. Brundan, J. Comes, D. Nash, A. Reynolds, 
 ``A basis theorem for the affine oriented Brauer category and its cyclotomic quotients''.
{\sl Quantum Topology \bf 8} (2017), 75 - 112.

\bibitem{BSW}
Brundan, Jonathan; Savage, Alistair; Webster, Ben, ``Foundations of Frobenius Heisenberg categories'', 
{\em J. Algebra \bf 578} (2021), 115 - 185. 

\bibitem{B} Bracken, A. J.;  {\em Group-theoretical applications in a tri-local model for baryons}. University of Adelaide PhD thesis, 1970.  

\bibitem{BG} Bracken, A. J.; Green, H. S. ``Vector operators and a polynomial identity for SO(n)''. {\sl J. Mathematical Phys. \bf 12} (1971), 2099--2106.

\bibitem{BGZ90} Bracken,  A. J.; Gould,  M. D.;   Zhang, R. B.
``Quantum supergroups and solutions of the Yang-Baxter equation''.
{\sl Modern Physics Letters \bf A5} (1990) No. 11,  831 - 840.

\bibitem{Br37} Richard Brauer, ``On algebras which are connected with the semisimple continuous groups'', {\sl Ann. of Math. (2) \bf 38} (1937), no. 4, 857--872. 

\bibitem{DLZ}
Deligne, P.; Lehrer, G. I.; Zhang, R. B. ``The first fundamental theorem of invariant theory for the orthosymplectic super group'', {\sl Adv. Math. \bf 327} (2018), 4--24.

\bibitem{DM}
Deligne, Pierre; Morgan, John W. {\em Notes on supersymmetry (following Joseph Bernstein)}.
Quantum fields and strings: a course for mathematicians,
Vol. 1, 2  (Princeton, NJ, 1996/1997), 41–97,  Amer. Math. Soc.,  Providence,  RI, 1999.

\bibitem{ES}  Ehrig, Michael;  Stroppel, Catharina, ``Nazarov--Wenzl algebras, coideal subalgebras
and categorified skew Howe duality.'' {\sl Adv. Math. \bf 331}  (2018) 58--142.


\bibitem{Go} Gould, M. D. ``Characteristic identities for semisimple Lie algebras''. {\sl J. Austral. Math. Soc. Ser. \bf B 26} (1985), no. 3, 257--283.

\bibitem{GZB} Gould, M. D.;  Zhang, R. B.;  Bracken, A. J.  ``Generalized Gel’fand invariants and characteristic identities for quantum groups.'' {\sl J. Math. Phys. \bf 32} (1991) 2298--2303. 

\bibitem{GL96}  Graham, J.~J. and   Lehrer, G.~I., ``Cellular algebras'', {\sl Invent. Math. \bf123} (1996), no. 1, 1--34.

\bibitem{GL98} Graham, J.~J. and  Lehrer, G.~I.,  ``The representation theory of affine Temperley-Lieb algebras'', {\sl Enseign. Math. (2) \bf 44} (1998), no. 3-4, 173--218. 

\bibitem{GL03}  Graham, J.~J. and Lehrer, G.~I., ``Diagram algebras, Hecke algebras and decomposition numbers at roots of unity'',
{\sl Ann. Sci. \'Ecole Norm. Sup. (4) \bf 36} (2003), no. 4, 479--524.

\bibitem{G}  Green, H. S. ``Characteristic identities for generators of GL(n), O(n) and SP(n)''. {\sl J. Mathematical Phys. \bf 12} (1971), 2106--2113.

\bibitem{ILZ} Iohara, K.; Lehrer, G. I.; Zhang, R. B. ``Equivalence of a tangle category and a category of infinite dimensional $\U_q(\mathfrak{sl}_2)$-modules.'' {\sl  Represent. Theory \bf 25} (2021), 265--299.  

\bibitem{IWD}
Isaac, Phillip S.; Werry, Jason L,;  Gould, Mark D . ``Characteristic identities for Lie (super)algebras". {\sl J. Phys.: Conf. Ser.} {\bf 597} (2015) 012045. 

\bibitem{JG} Jarvis, P. D.; Green, H. S. ``Casimir invariants and characteristic identities for generators of the general linear, special linear and orthosymplectic graded Lie algebras.'' {\sl J. Math. Phys. \bf  20} (1979), no. 10, 2115--2122.

\bibitem{K} Kassel, Christian,  {\em Quantum groups}. Graduate Texts in Mathematics, {\bf 155}. Springer-Verlag, New York, 1995.

\bibitem{Ki1}
Kirillov, A. A., ``Family algebras'', {\sl Electron. Res. Announc. Amer. Math. Soc. \bf 6} (2000), 7 - 20.

\bibitem{Ki2} Kirillov, A. A., ``Introduction to family algebras'', {\sl Moscow Math. J. \bf 1}(1) (2001), 49 - 63.

\bibitem{Ko} Kostant, B.,  ``On the tensor product of a finite and an infinite dimensional representation'',
{\sl J. Funct. Anal. \bf20}(4) (1975), 257 - 285.

\bibitem{Kg} G. Kuperberg, ``Spiders for rank 2 Lie algebras'', {\sl Comm. Math. Phys. \bf 180} (1996), no. 1, 109 - 151. 

\bibitem{LZ06} Lehrer, G. I.; Zhang, R. B. ``Strongly multiplicity free modules for Lie algebras and quantum groups". {\sl J. Algebra \bf 306} (2006), no. 1, 138--174.

\bibitem{LZ10}  Lehrer, G. I.; Zhang, R. B., ``A Temperley-Lieb analogue for the BMW algebra'',
{\sl Representation theory of algebraic groups and quantum groups}, 155--190, {\sl Progr. Math., \bf 284}, Birkh{\"a}user/Springer, New York, 2010.

\bibitem{LZ12}  Lehrer, G. I.; Zhang, R. B., ``The second fundamental theorem of invariant theory for the orthogonal group''. {\sl Ann. of Math. (2) \bf 176} (2012), no. 3, 2031--2054. 

\bibitem{LZ15} Lehrer, G. I.; Zhang, R. B.  ``The Brauer category and invariant theory'', 
{\sl J. Europ. Math. Soc. \bf 17}  (2015) 2311-- 2351. 

\bibitem{LZ17} Lehrer, G. I.; Zhang, R. B. ``The first fundamental theorem of invariant theory for the orthosymplectic supergroup.'' {\sl Comm. Math. Phys. \bf 349} (2017), no. 2, 661--702.

\bibitem{LZ17-Ecate}  Lehrer, Gustav I.; Zhang, Ruibin ``Invariants of the special orthogonal group and an enhanced Brauer category.'' {\sl  Enseign. Math. \bf 63} (2017), no. 1 - 2, 181--200.

\bibitem{LZ17-Pf} Lehrer, G. I.; Zhang, R. B. ``Invariants of the orthosymplectic Lie superalgebra and super Pfaffians.'' {\sl  Math. Z. \bf 286} (2017), no. 3--4, 89--917.

\bibitem{LZ21}  Lehrer, G. I.; Zhang, R. B. ``The second fundamental theorem of invariant theory for the orthosymplectic supergroup.'' {\sl  Nagoya Math. J. \bf 242} (2021), 52--76. 

\bibitem{LZ23}  Lehrer, G. I.; Zhang, R. B. ``Diagram categories and invariant theory for classical groups and supergroups''. To appear in 
Proceedings of the $9^{th}$ International Congress of Chinese Mathematicians, June 2019, Beijing. 

\bibitem{LZZ20}
Lehrer, G., Zhang, H., Zhang, R.  ``First fundamental
theorems of invariant theory for quantum supergroups''.
{\sl European J. of Math. \bf 6} (2020) No. 3, 928 - 976. 

\bibitem{McS}  Peter J. McNamara, Alistair Savage, ``The spin Brauer category '', arXiv2312:11766.

\bibitem{Mor11} Morrison, Scott, ``The braid group surjects onto $G_2$ tensor space'',
{\sl Pacific J. Math. \bf 249 } (2011), no. 1, 189--198.

\bibitem{Rn}  Rozhkovskaya, N. ``Commutativity of quantum family algebras''. 
{\sl Lett. Math. Phys. \bf 63} (2003), no. 2, 87 - 103.


\bibitem{Ros96} Rost, M., ``On the Dimension of a Composition Algebra'', {\sl Doc. Math. \bf 1}  (1996) No. 10, 209 - 214.

\bibitem{Ros04} Rost, M., ``On vector product algebras'',
https://www.math.uni-bielefeld.de/~rost/data/vpg.pdf

\bibitem{RSo} Rui, Hebing; Song, Linliang.  ``Affine Brauer category and parabolic category O in types B, C, D''. {\sl Math. Z. \bf 293} (2019), no. 1-2, 503--550. 

\bibitem{RS} Rui, Hebing; Su, Yucai,  ``Affine walled Brauer algebras and super Schur-Weyl duality''. {\sl Adv. Math. \bf 285} (2015), 28--71.

\bibitem{RT} Reshetikhin, N.; Turaev, V. G. “Ribbon graphs and their invariants derived from quantum groups.” {\sl Comm. Math. Phys. \bf 127} (1990), 1-26.

\bibitem{We05} Tuba, Imre and Wenzl, Hans, ``On braided tensor categories of type BCD''
{\sl J. Reine Angew. Math. \bf 581} (2005), 31--69.

\bibitem{T} Turaev, V.G. {\em Quantum invariants of knots and 3-manifolds}. de Gruyter Studies in Mathematics, 18. Walter de Gruyter, Berlin, 1994.

\bibitem{Z94} Zhang, R. B.  ``Three-manifold invariants arising from $\U_q(\osp(1|2))$''.  {\sl  Modern Physics Letters \bf A9} (1994), No. 16, 1453 - 1465.  

\bibitem{Z95} Zhang, R. B.  ``Quantum supergroups and topological invariants 
of three-manifolds''. {\sl Rev. Math. Physics \bf 7} (1995), No. 05,  809 - 831. 

\bibitem{Z98} Zhang, R. B.  ``Structure and representations of the quantum general linear supergroup''. {\sl Comm. Math. Phys. \bf 195} (1998), no. 3, 525--547.

\bibitem{Z08} Zhang, R. B. ``Orthosymplectic Lie superalgebras in superspace analogues of quantum Kepler problems.''  {\sl Comm. Math. Phys. \bf 280} (2008), no. 2, 545--562. 

\bibitem{ZGB} Zhang, R. B.;   Gould,  M. D.;  Bracken,  A. J.   ``Quantum group invariants and link polynomials'', 
{\sl Commun. Math. Physics \bf 37} (1991) 13 -- 27.  

\bibitem{Zy}Zhang, Yang,   ``On the second fundamental theorem of invariant theory for the orthosymplectic supergroup.''  {\sl J. Algebra \bf 501} (2018), 394--434. 

\end{thebibliography}
\end{document}